\documentclass[12pt, reqno]{amsart}

\usepackage{
amsmath,
amssymb,
amsthm,
mathrsfs,
amsfonts,
ytableau,
enumitem,
comment,
upgreek,
soul,
yhmath,
mathtools,
shuffle,
kotex,
array,
caption,
tabularx
}

\usepackage[linktocpage]{hyperref}

\hypersetup{
    colorlinks=true,
    linkcolor = blue,
    citecolor=magenta,
    urlcolor=cyan,
}

\usepackage[capitalise, nameinlink]{cleveref}

\crefname{equation}{}{}
\crefname{figure}{{\sc Figure}}{{\sc Figure}}
\crefname{subsection}{Subsection}{Subsections}

\usepackage[euler]{textgreek}

\usepackage{tikz,tikz-cd}

\usepackage[colorinlistoftodos, textwidth = 2.3cm]{todonotes}

\ytableausetup{mathmode, boxsize=1.2em}

\newtheorem{theorem}{Theorem}[section]
\newtheorem{proposition}[theorem]{Proposition}
\newtheorem{lemma}[theorem]{Lemma}
\newtheorem{corollary}[theorem]{Corollary}
\newtheorem{conjecture}[theorem]{Conjecture}

\newtheorem*{claim*}{Claim}

\theoremstyle{definition}

\newtheorem{example}[theorem]{Example}
\newtheorem{definition}[theorem]{Definition}
\newtheorem{remark}[theorem]{Remark}

\numberwithin{equation}{section} \numberwithin{figure}{section}
\numberwithin{table}{section}

\def\Z{\mathbb Z}
\def\C{\mathbb C}

\newcommand{\nc}{\newcommand}

\nc{\sfP}{\mathsf{P}}
\nc{\sfS}{\mathsf{S}}
\nc{\sfDS}{\mathsf{DS}}
\nc{\sfst}{\mathsf{st}}
\nc{\tDC}{\mathtt{DC}}
\nc{\scrC}{\mathscr{C}}
\nc{\SR}{\mathcal{SR}}
\nc{\sfposet}
{\mathsf{poset}}
\nc{\rmcard}{\mathrm{card}}
\nc{\ocard}{\overline{\mathrm{card}}}
\nc{\omin}{\overline{\mathrm{min}}}
\nc{\omax}{\overline{\mathrm{max}}}
\nc{\calL}{\mathcal{L}}
\nc{\calK}{\mathcal{K}}
\nc{\calI}{\mathcal{I}}
\nc{\sfM}{\mathsf{M}}
\nc{\sfm}{\mathsf{m}}
\nc{\Ptab}{\mathtt{ins}}
\nc{\Qtab}{\mathtt{rec}}
\nc{\pt}{T^{\scalebox{0.5}{$\nearrow$}}}
\nc{\rmIm}{\mathrm{Im}}
\nc{\rmInv}{\mathrm{Inv}}
\nc{\rmInt}{\mathrm{Int}}
\nc{\Deq}{\overset{D}{\simeq}}
\nc{\Dleq}{\preceq_D}
\nc{\bdP}{\boldsymbol{P}}
\nc{\bdI}{\boldsymbol{I}}
\nc{\soc}{\mathrm{soc}}
\nc{\balproj}{\bal_{\mathrm{proj}}}
\nc{\balinj}{\bal_{\mathrm{inj}}}

\nc{\SG}{\mathfrak{S}}
\nc{\frakR}{\mathfrak{R}}
\nc{\frakL}{\mathfrak{L}}
\nc{\PCT}{\mathrm{PCT}}
\nc{\SPCT}{\mathrm{SPCT}}
\nc{\RT}{\mathrm{RT}}
\nc{\SRT}{\mathrm{SRT}}
\nc{\SYT}{\mathrm{SYT}}
\nc{\SSYT}{\mathrm{SSYT}}
\nc{\LR}{\mathrm{LR}}
\nc{\sgn}{\mathrm{sgn}}
\nc{\RCT}{\mathrm{RCT}}
\nc{\SRCT}{\mathrm{SRCT}}
\nc{\SYRT}{\mathrm{SYRT}}
\nc{\SYCT}{\mathrm{SYCT}}
\nc{\SPYCT}{\mathrm{SPYCT}}
\nc{\tst}{\mathtt{st}}
\nc{\Span}{\mathrm{span}}
\nc{\comp}{\mathrm{comp}}
\nc{\rmst}{\mathrm{st}}
\nc{\Des}{\mathrm{Des}}
\nc{\set}{\mathrm{set}}
\nc{\wt}{\mathrm{wt}}
\nc{\ch}{\mathrm{ch}}
\nc{\id}{\mathrm{id}}
\nc{\Sym}{\mathrm{Sym}}
\nc{\Qsym}{\mathrm{QSym}}
\nc{\Nsym}{\mathrm{NSym}}
\nc{\sh}{\mathrm{sh}}
\nc{\bfS}{\mathbf{S}}
\nc{\bfm}{\mathbf{m}}
\nc{\hbfS}{\widehat{\mathbf{S}}}
\nc{\bfF}{\mathbf{F}}

\nc{\calA}{\mathcal{A}}
\nc{\calB}{\mathcal{B}}
\nc{\calG}{\mathcal{G}}
\nc{\calR}{\mathcal{R}}
\nc{\calS}{\mathcal{S}}
\nc{\calV}{\mathcal{V}}

\nc{\sfR}{\mathsf{R}}

\nc{\tal}{\lambda(\alpha)}
\nc{\tbe}{\widetilde{\beta}}
\nc{\opi}{\overline{\pi}}
\nc{\calP}{\mathcal{P}}
\nc{\rmtop}{\mathrm{top}}
\nc{\rad}{\mathrm{rad}}
\nc{\bfP}{\mathbf{P}}
\nc{\bfw}{\mathbf{w}}
\nc{\SET}{\mathrm{SET}}

\nc{\tcd}{\mathtt{cd}}
\nc{\tyd}{\mathtt{yd}}
\nc{\trd}{\mathtt{rd}}

\nc{\rmr}{\mathrm{r}}
\nc{\rmc}{\mathrm{c}}
\nc{\rmt}{\mathrm{t}}

\nc{\len}{\mathsf{len}}
\nc{\col}{\mathrm{col}}
\nc{\row}{\mathrm{row}}
\nc{\calE}{\mathcal{E}}
\nc{\calT}{\mathscr{T}}
\nc{\sfT}{\mathsf{T}}
\nc{\calEsa}{\mathcal{E}^\upsig(\alpha)}
\nc{\tauC}{\tau_{\scalebox{0.5}{$C$}}}
\nc{\sytabC}{\sytab_{\scalebox{0.5}{$C$}}}
\nc{\pr}{\mathsf{pr}}
\nc{\Ups}{\Upsilon}
\nc{\pact}{\diamond}
\nc{\tauE}{\tau_{E}^{~}}
\nc{\tauF}{\tau_{\scalebox{0.5}{$F$}}}
\nc{\tauG}{\tau_{\scalebox{0.5}{$G$}}}
\nc{\rtE}{T_{\scalebox{0.5}{$E$}}}
\nc{\rtF}{T_{\scalebox{0.5}{$F$}}}
\nc{\rtG}{T_{\scalebox{0.5}{$G$}}}
\nc{\sytab}{\widehat{\tau}}
\nc{\hatE}{\widehat{E}}
\nc{\hati}{\hat{i}}
\nc{\hcalE}{\widehat{\calE}}
\nc{\hatC}{\widehat{C}}
\nc{\bal}{{\boldsymbol{\upalpha}}}
\nc{\bbe}{{\boldsymbol{\upbeta}}}
\nc{\SPYRT}{\mathrm{SPYRT}}

\nc{\bgam}{{\boldsymbol{\upgamma}}}
\nc{\bdel}{{\boldsymbol{\updelta}}}
\nc{\weakcon}{\odot}
\nc{\basisI}{I}

\nc{\ldalpha}{\lambda(\alpha)}
\nc{\SRIT}{\mathrm{SRIT}}
\nc{\re}{\mathrm{rev}}
\nc{\otau}{\overline{\tau}}
\nc{\rtop}{{\rm top}}
\nc{\sfc}{\mathsf{c}}
\nc{\sfC}{\mathsf{C}}
\nc{\sfr}{\mathsf{r}}
\nc{\sfSP}{\mathsf{SP}}
\nc{\sfRP}{\mathsf{RP}}
\nc{\sfRSP}{\mathsf{RSP}}
\nc{\tH}{\mathtt{H}}
\nc{\tV}{\mathtt{V}}
\nc{\rpi}{\mathring{\pi}}
\nc{\cpi}{\check{\pi}}
\nc{\frakm}{\mathfrak{m}}
\nc{\sfem}{\mathsf{em}}
\nc{\Hom}{\mathrm{Hom}}
\nc{\module}{\mathrm{mod} \, }
\nc{\rmread}{\mathsf{read}}
\nc{\tread}{\underline{\mathsf{read}}}
\nc{\ocalE}{\overline{\calE}}
\nc{\oE}{\overline{E}}

\nc{\SPCTsa}{\SPCT^\upsig(\alpha)}
\nc{\bfSsa}{\bfS_\alpha^\upsig}
\nc{\bfSsaC}{{\bfS}^\upsig_{\alpha,C}}
\nc{\hbfSsa}{\widehat{\bfS}_\alpha^\upsig}
\nc{\upineq}{\rotatebox{90}{$<$}}
\nc{\downineq}{\rotatebox{270}{$<$}}
\nc{\diagineq}{\rotatebox{135}{$<$}}
\nc{\sfB}{\mathsf{B}}
\nc{\hxi}{\widehat{\xi}}
\nc{\hxidwJ}{\hxi_{\scalebox{0.55}{$J$}}}
\nc{\hxiupJ}{\hxi^{\scalebox{0.55}{$J$}}}
\nc{\scrS}{\mathscr{S}}
\nc{\bfT}{\mathbf{T}}
\nc{\tshuffle}{\,\widetilde{\shuffle}\,}
\nc{\sJ}{\scalebox{0.55}{$J$}}
\nc{\sJo}{\scalebox{0.55}{$J_1$}}
\nc{\sJt}{\scalebox{0.55}{$J_2$}}
\nc{\Keq}{\overset{K}{\cong}}
\nc{\dKeq}{\overset{K^*}{\cong}}
\nc{\Rect}{\mathrm{Rect}}

\nc{\ra}{\rightarrow}

\nc{\matr}[2]{\left( \hspace{-1ex} \begin{array}{c} #1 \\ #2 \end{array} \hspace{-1ex} \right)}

\definecolor{wsgreen}{rgb}{0,0.5,0}

\nc{\DIRT}{\mathrm{DIRT}}
\nc{\hpi}{\pi}
\nc{\frakI}{\mathfrak{I}}
\nc{\hfrakI}{\widehat{\mathfrak{I}}}
\nc{\orho}{\overline{\rho}}
\nc{\autotheta}{\uptheta}
\nc{\calW}{\mathcal{W}}
\nc{\Endo}{\mathrm{End}}

\nc{\autophi}{\upphi}
\nc{\autochi}{\upchi}
\nc{\autoomega}{\upomega}
\nc{\hIM}{\widehat{\sfB}}
\nc{\bfpi}{\boldsymbol{\uppi}}
\nc{\bfopi}{\overline{\boldsymbol{\uppi}}}
\nc{\osfB}{\overline{\sfB}}
\nc{\rmw}{\mathrm{w}}
\nc{\conc}{\; {\bullet} \;}
\nc{\ostar}{\; \overline{\bullet} \;}
\nc{\rank}{\mathrm{rank}}
\nc{\fkp}{\mathfrak{p}}
\nc{\bfR}{\mathbf{R}}
\nc{\calD}{\overline{\mathrm{Des}}}

\nc{\upsig}{{\boldsymbol{\upsigma}}}
\nc{\bfSsaE}{{\bfS}^\upsig_{\alpha,E}}

\nc{\hfkp}{\widehat{\mathfrak{p}}}

\nc{\hautophi}{{\widehat{\autophi}}}
\nc{\hautotheta}{{\widehat{\autotheta}}}
\nc{\hautoomega}{{\widehat{\autoomega}}}
\nc{\rmperm}{\mathrm{perm}}
\nc{\Hnmod}{\text{$H_n(0)$-$\mathbf{mod}$}}
\nc{\modHn}{\text{$\mathbf{mod}$-$H_n(0)$}}

\nc{\bfsigJ}{\sigma_{\scalebox{0.55}{$J$}}}
\nc{\bfrhoJ}{\rho^{\scalebox{0.55}{$J$}}}

\nc{\teta}{\widetilde{\eta}}
\nc{\urmw}{\underline{\mathrm{w}}}
\nc{\sfcnt}{\mathsf{cnt}}

\nc{\pistar}[1]{\pi_{#1}^*}
\nc{\wfkp}{\widetilde{\mathfrak{p}}}
\nc{\bfpsi}{\boldsymbol{\uppsi}}

\nc{\yt}[1]{\todo[size=\tiny,color=blue!10]{#1 \\ \hfill --- Young-Tak}}
\nc{\YT}[1]{\todo[size=\tiny,inline,color=blue!10]{#1
		\\ \hfill --- Young-Tak}}
		
\nc{\yh}[1]{\todo[size=\tiny,color=cyan!10]{#1 \\ \hfill --- Young-Hun}}
\nc{\YH}[1]{\todo[size=\tiny,inline,color=cyan!10]{#1
		\\ \hfill --- Young-Hun}}
		
\nc{\sy}[1]{\todo[size=\tiny,color=magenta!10]{#1 \\ \hfill --- So-Yeon}}
\nc{\SY}[1]{\todo[size=\tiny,inline,color=magenta!10]{#1
		\\ \hfill --- So-Yeon}}

\nc{\nt}[1]{\todo[size=\tiny,color=green!10]{#1 \\ \hfill --- Note}}
\nc{\NT}[1]{\todo[size=\tiny,inline,color=green!10]{#1
		\\ \hfill --- Note}}

\definecolor{purple}{rgb}{0.44, 0.0, 1.0}

\definecolor{yhblue}{rgb}{0,0,0.6}

\newenvironment{red}{\relax\color{red}}{\hspace*{.5ex}\relax}
\newenvironment{blue}{\relax\color{yhblue}}{\hspace*{.5ex}\relax}
\newenvironment{green}{\relax\color{wsgreen}}{\hspace*{.5ex}\relax}
\newenvironment{magenta}{\relax\color{magenta}}{\hspace*{.5ex}\relax}
\newenvironment{purple}{\relax\color{purple}}{\hspace*{.5ex}\relax}

\nc{\ber}{\begin{red}}
\nc{\er}{\end{red}}
\nc{\beb}{\begin{blue}}
\nc{\eb}{\end{blue}}
\nc{\bema}{\begin{magenta}}
\nc{\ema}{\end{magenta}}
\nc{\begr}{\begin{green}}
\nc{\egr}{\end{green}}

\nc{\bepu}{\begin{purple}}
\nc{\epu}{\end{purple}}

\title[Regular Schur labeled skew shape posets and their 0-Hecke modules]{
Regular Schur labeled skew shape posets and their 0-Hecke modules}

\author[Y.-H. Kim]{Young-Hun Kim}
\address[Y.-H. Kim]{Center for Quantum structures in Modules and Spaces, Seoul National University, Seoul 08826, Republic of Korea}
\email{ykim.math@gmail.com}

\author[S.-Y. Lee]{So-Yeon Lee}
\address[S.-Y. Lee]{Department of Mathematics, Sogang University, Seoul 04107, Republic of Korea}
\email{sylee0814@sogang.ac.kr}

\author[Y.-T. Oh]{Young-Tak Oh}
\address[Y.-T. Oh]{Department of Mathematics / Institute for Mathematical and Data Sciences, Sogang University, Seoul 04107, Republic of Korea}
\email{ytoh@sogang.ac.kr}


\date{\today}

\begin{document}

\maketitle

\begin{abstract}
Assuming Stanley's $P$-partitions conjecture holds, the regular Schur labeled skew shape posets are precisely the finite posets $P$ with underlying set $\{1, 2, \ldots, |P|\}$
such that the $P$-partition generating function is symmetric and the set of linear extensions of $P$, denoted  $\Sigma_L(P)$, is a left weak Bruhat interval in the symmetric group $\mathfrak{S}_{|P|}$.
We describe the permutations in $\Sigma_L(P)$ in terms of reading words of standard Young tableaux when $P$ is a regular Schur labeled skew shape poset, and classify $\Sigma_L(P)$'s up to descent-preserving isomorphism as $P$ ranges over regular Schur labeled skew shape posets.
The results obtained are then applied to classify the $0$-Hecke modules $\mathsf{M}_P$ associated with regular Schur labeled skew shape posets $P$ up to isomorphism. 
Then we characterize  regular Schur labeled skew shape posets as the finite posets $P$ whose linear extensions form a dual plactic-closed subset of  $\mathfrak{S}_{|P|}$. 
Using this characterization, we construct  distinguished filtrations of $\mathsf{M}_P$ with respect to the Schur basis 
when $P$ is a regular Schur labeled skew shape poset.
Further issues concerned with the classification and decomposition of the $0$-Hecke modules $\mathsf{M}_P$ are also discussed.
\end{abstract}

\tableofcontents

\section{Introduction}
\emph{Schur labeled skew shape posets} naturally appear in the context of the celebrated  Stanley's $P$-partitions conjecture. 
Let $\sfP_n$ be the set of posets with underlying set $[n] := \{1,2,\ldots, n\}$. Each poset $P \in \sfP_n$ can be identified with the labeled poset $(P, \omega)$ with the labeling $\omega: P \ra [n]$ given by $\omega(i) = i$.
Consequently, to each poset $P\in \sfP_n$, 
one can associate the following generating function for its $P$-partitions: 
\[
K_{P} := \sum_{f: \text{$P$-partition}}
x_1^{|f^{-1}(1)|} x_2^{|f^{-1}(2)|} \cdots.
\]
In 1972, Stanley \cite[p. 81]{72Stanley} proposed a conjecture stating that $K_P$ is a symmetric function if and only if $P$ is a Schur labeled skew shape poset. 
For the precise definition of Schur labeled skew shape posets, refer to ~\cref{Regular posets and Schur labeled skew shape posets}.
While this conjecture has been verified to be true for all posets $P$ with $|P| \le 8$, it remains an open question in the general case (see \cite{06McNamara}).
We denote by $\sfSP_n$ the set of all Schur labeled skew shape posets in $\sfP_n$.

On the other hand, \emph{regular posets} were introduced by Bj\"{o}rner--Wachs ~\cite{91BW} during their investigation of the convex subsets of the symmetric group $\SG_n$ on $\{1,2, \ldots n\}$ under the right weak Bruhat order.
For $P \in \sfP_n$ with the partial order $\preceq$, let $\Sigma_R(P)$ be the set  of permutations $\pi \in \SG_n$ satisfying that if $x \preceq y$, then $\pi^{-1}(x) \le \pi^{-1}(y)$.
They 
observed  that every convex subset  of $\SG_n$ under the right weak Bruhat order appears as $\Sigma_R(P)$ for some $P \in \sfP_n$, and 
every right weak Bruhat interval in $\SG_n$ is convex.
This observation led them to characterize the posets $P \in \sfP_n$ satisfying that $\Sigma_R(P)$ is a right weak Bruhat interval. 
They introduced the notion of regular posets, and  proved that $P \in \sfP_n$ is a regular poset if and only if $\Sigma_R(P)$ is a right weak Bruhat interval in $\SG_n$.
For the definition of regular posets, refer to \cref{def: regular posets}.
We denote by $\sfRP_n$ the set of all regular posets in $\sfP_n$.

Let $\sfRSP_n := \sfRP_n \cap \sfSP_n$.
In the following, we explain the reason why we consider regular Schur labeled skew shape posets from the perspective of the representation theory of the $0$-Hecke algebra.

In 1996, Duchamp, Krob, Leclerc, and Thibon~\cite{96DKLT} showed that the Grothendieck ring of the tower of $0$-Hecke algebras $\bigoplus_{n\ge 0}H_n(0)$, when equipped with addition and multiplication from direct sum and induction product, is isomorphic to the ring $\Qsym$ of quasisymmetric functions.
To be precise, they showed that 
the map  
\begin{align*}
\ch : \bigoplus_{n \ge 0} \calG_0(\Hnmod) \ra \Qsym, \quad [\bfF_\alpha] \mapsto F_{\alpha},
\end{align*}
called the \emph{quasisymmetric characteristic},
is a ring isomorphism.
Here, $\calG_0(\Hnmod)$ is the Grothendieck group of the category $\Hnmod$ of finitely generated left $H_n(0)$-modules, $\alpha$ is a composition, $\bfF_\alpha$ is the irreducible $H_n(0)$-module attached to $\alpha$, and $F_\alpha$ is the fundamental quasisymmetric function attached to $\alpha$ (for more details, see \cref{subsec: 0-Hecke alg and QSym}).
Afterwards, Bergeron--Li \cite{09BL} showed that the map $\ch$ is not just a ring isomorphism but also a Hopf algebra isomorphism.
In 2002, Duchamp--Hivert--Thibon \cite{02DHT} associated a right $H_n(0)$-module $M_P$ with each poset $P \in \sfP_n$, such that the image of $M_P$ under the quasisymmetric characteristic is $K_P$. 
This was achieved by defining a suitable right $H_n(0)$-action on  $\Sigma_R(P)$.

Since the middle of 2010, various  left $0$-Hecke modules, each equipped with a tableau basis and yielding an important quasisymmetric characteristic image, have been constructed (\cite{20BS, 15BBSSZ, 19Searles, 15TW, 19TW}).
In order to handle these modules in a uniform manner, Jung--Kim--Lee--Oh \cite{22JKLO} introduced a left $H_n(0)$-module $\sfB(I)$, referred to as \emph{the weak Bruhat interval module associated with $I$}, for each left weak Bruhat interval $I$ in $\SG_n$.
Furthermore, they showed that $\bigoplus_{n \ge 0} \mathcal{G}_0(\mathscr{B}_n)$ is isomorphic to $\Qsym$ as Hopf algebras, where $\mathscr{B}_n$ is the full subcategory of $\Hnmod$ consisting of objects that are direct sums of finitely many isomorphic copies of weak Bruhat interval modules of $H_n(0)$.
Recently, Choi--Kim--Oh \cite{23CKO} clarified the exact relationship between the weak Bruhat interval modules and the $0$-Hecke modules $M_P$, using Bj\"{o}rner--Wachs' characterization.
More precisely, they constructed a contravariant functor $\mathcal{F}:\Hnmod \ra \modHn$ that preserves the quasisymmetric characteristic and showed that $M_P = \mathcal{F}(\sfB(\Sigma_L(P)))$, where $\modHn$ is the category of finitely generated right $H_n(0)$-modules and $\Sigma_L(P) := \{\gamma^{-1} \mid \gamma \in \Sigma_R(P)\}$ for $P \in \sfRP_n$.
For technical reasons, we use a slightly different $0$-Hecke module, denoted as $\sfM_P$, instead of Duchamp, Hivert, and Thibon's module $M_P$. This module is a left $H_n(0)$-module with the basis $\Sigma_L(P)$.
For the detailed definition of $\sfM_P$, refer to \cref{def: poset module}.

The aim of this paper is to give a comprehensive investigation of regular Schur labeled skew shape posets and their associated $0$-Hecke modules.

In ~\cref{sec: int str for regular Schur}, we provide an explicit description of  $\Sigma_L(P)$ for $P \in \sfRSP_n$.  
We first introduce a Schur labeling $\tau_P$, which is a bijective tableau uniquely determined by suitable conditions.
For details, see \cref{eq: def of tau_P}.  
Let $\lambda/\mu$ be the shape of $\tau_P$.
Then $\tau_P$ gives rise to a reading, denoted $\rmread_{\tau_P}$, on 
the set $\SYT(\lambda/\mu)$ of standard Young tableaux of shape $\lambda/\mu$.
We show that all permutations in $\Sigma_L(P)$ appear  
as reading words of standard Young tableaux of shape $\lambda/\mu$, i.e.,  
$
\Sigma_L(P) = \rmread_{\tau_P}(\SYT(\lambda/\mu))
$
(\cref{lem: identification of P and S}).
Then, we derive that 
\[
\Sigma_L(P) = [\rmread_{\tau_P}(T_{\lambda/\mu}), \rmread_{\tau_P}(T'_{\lambda/\mu})]_L,\]
where $T_{\lambda/\mu}$(resp. $T'_{\lambda/\mu}$) is the standard Young tableau obtained by filling the Young diagram of shape $\lambda/\mu$ 
by $1, 2, \ldots, n$ from left to right starting with the top row (resp. from top to bottom starting with leftmost column)
(\cref{thm: interval descriptrion for regular skew schur poset}).

In ~\cref{sec: interval equal upto desc pres iso},
we introduce an equivalence relation $\Deq$ on the set $\rmInt(n)$ of left weak Bruhat intervals in $\SG_n$. This relation is defined by
$I_1 \Deq I_2$ if there is a descent-preserving poset isomorphism between $I_1$ and $I_2$.
We show that every  equivalence class $C$ is of the form
\[
\{[\gamma,  \xi_C\gamma]_L \mid \gamma \in [\sigma_0, \sigma_1]_R\},
\]
where $\sigma_0$ and $\sigma_1$ are the minimal and maximal elements in $\{\sigma \mid [\sigma, \rho]_L \in C\}$, 
respectively, and $\xi_C = \rho \sigma^{-1}$ for any $[\sigma, \rho]_L \in C$
(\cref{thm: desc pres equiv}).
In the case where $P \in \sfRSP_n$, 
we show in \cref{thm: equiv class X lam mu} that the equivalence class of $\Sigma_L(P)$ is given by
\begin{align}\label{eq: equiv class for P in RSPn}
\{\Sigma_L(Q) \mid \text{$Q \in \sfRSP_n$ with $\sh(\tau_Q) = \sh(\tau_P)$} \}.
\end{align}

In ~\cref{sec: classification of X}, we classify the $H_n(0)$-modules $\sfM_P$
up to isomorphism as $P$ ranges over $\sfRSP_n$.
We show in \cref{thm: classification} that for $P, Q \in \sfRSP_n$, 
\begin{align*}
\sfM_P \cong \sfM_Q \quad \text{if and only if} \quad \sh(\tau_P) = \sh(\tau_Q).
\end{align*}
The ``if'' part is straightforward and can be derived from \cref{eq: equiv class for P in RSPn}. As for the ``only if'' part, it can be verified by showing that when $\tau_P$ and $\tau_Q$ have different shapes, it results in either nonisomorphic projective covers or nonisomorphic injective hulls of $\sfM_P$ and $\sfM_Q$. 
To accomplish this, we compute both a projective cover and an injective hull of $\sfM_P$ for $P \in \sfRSP_n$ (\cref{lem: proj cov and inj hull}).

In ~\cref{Sec: filtration}, we first prove that a poset $P \in \sfP_n$ is a regular Schur labeled skew shape poset if and only if $\Sigma_L(P)$ is dual plactic-closed (\cref{thm: RSP iff DPC}).
This improves Malvenuto's result \cite[Theorem 1]{93M}, which states that if $\Sigma_L(P)$ is dual plactic-closed, then $P \in \sfSP_n$.
Then, we introduce the notion of a \emph{distinguished filtration} of an $H_n(0)$-module $M$ with respect to a linearly independent subset $\calB$ of $\Qsym_n$ (\cref{def: dist filt}).
If such a filtration is available, we have a  representation theoretic interpretation of the expansion of $\ch([M])$ in $\calB$.
The existence of a distinguished filtration is quite nontrivial as seen in ~\cref{eg: not dist filt}.
However, using the characterization given in ~\cref{thm: RSP iff DPC}, we show that $\sfM_P$  admits a distinguished filtration with respect to the Schur basis 
when $P \in \sfRSP_n$ (\cref{thm: filt of X lam mu}).

The final section is mainly devoted to further issues concerned with the classification and decomposition of the $0$-Hecke modules $\sfM_P$.
We discuss the classification problem for $\{\sfM_P \mid P \in \sfSP_n\}$ and $\{\sfM_P \mid P \in \sfRP_n\}$.
In particular, we expect that for $P, Q \in \sfRP_n$,
$\sfM_P \cong \sfM_Q$  if and only if  
$\Sigma_L(P) \Deq \Sigma_L(Q)$
(\cref{conj: first one}). 
The decomposition problem is also discussed for the $0$-Hecke modules $\sfM_P$ when $P \in \sfRSP_n$.
Based on experimental data, we expect that for $P \in \sfRSP_n$,
$\sfM_P$ is indecomposable if and only if $\sh(\tau_P)$ is disconnected and does not contain any disconnected ribbon
(\cref{conj: on indecomp}).
At the end of this section, we provide a remark on how to recover $\sfM_P$ for $P \in \sfRSP_n$ from a module of the generic Hecke algebra $H_n(q)$ by specializing $q$ to $0$.

In the appendix, we give a tableau description of $\sfM_P$ for $P \in \sfRSP_n$.
For a skew partition $\lambda/\mu$ of size $n$, we construct an $H_n(0)$-module $X_{\lambda/\mu}$ with standard Young tableaux of shape $\lambda/\mu$ as basis elements.
This module can be viewed as a representative of the isomorphism class of $\sfM_P$  in the category $\Hnmod$ for every $P \in \sfRSP_n$ with $\sh(\tau_P) = \lambda/\mu$.

\section{Preliminaries}\label{sec: preliminary}

For integers $m$ and $n$, we define $[m,n]$ and $[n]$ to be the intervals $\{t\in \mathbb Z \mid m\le t \le n\}$ and $\{t\in \mathbb Z \mid 1\le t \le n\}$, respectively.
Throughout this paper, $n$ will denote a nonnegative integer unless otherwise stated.

\subsection{Compositions, Young diagrams, and bijective tableaux}\label{subsec: comp and diag}

A \emph{composition} $\alpha$ of $n$, denoted by $\alpha \models n$, is a finite ordered list of positive integers $(\alpha_1,\alpha_2,\ldots, \alpha_k)$ satisfying $\sum_{i=1}^k \alpha_i = n$.
We call $\alpha_i$ ($1 \le i \le k$) a \emph{part} of $\alpha$, $k =: \ell(\alpha)$ the \emph{length} of $\alpha$, and $n =:|\alpha|$ the \emph{size} of $\alpha$. 
And, we define the empty composition $\varnothing$ to be the unique composition of size and length $0$.
Whenever necessary, we set $\alpha_i = 0$ for all $i > \ell(\alpha)$.

Given $\alpha = (\alpha_1,\alpha_2,\ldots,\alpha_k) \models n$ and $I = \{i_1 < i_2 < \cdots < i_{l}\} \subseteq [n-1]$, 
let 
\begin{align*}
    \set(\alpha) &:= \{\alpha_1,\alpha_1+\alpha_2,\ldots, \alpha_1 + \alpha_2 + \cdots + \alpha_{k-1}\}, \text{ and} \\\comp(I) &:= (i_1,i_2 - i_1, i_3 - i_2, \ldots,n-i_{l}).
\end{align*}
The set of compositions of $n$ is in bijection with the set of subsets of $[n-1]$ under the correspondence $\alpha \mapsto \set(\alpha)$ (or $I \mapsto \comp(I)$).
The \emph{reverse composition $\alpha^\rmr$ of $\alpha$} is defined to be the composition $(\alpha_k, \alpha_{k-1}, \ldots, \alpha_1)$ and
the \emph{complement $\alpha^\rmc$ of $\alpha$} is defined to be the unique composition satisfying $\set(\alpha^c) = [n-1] \setminus \set(\alpha)$.

If a composition $\lambda = (\lambda_1, \lambda_2, \ldots, \lambda_k) \models n$ satisfies $\lambda_1 \ge \lambda_2 \ge \cdots \ge \lambda_k$, then it is called a \emph{partition} of $n$ and denoted as $\lambda \vdash n$.
Given two partitions $\lambda$ and $\mu$ with $\ell(\lambda) \ge \ell(\mu)$, we write $\lambda \supseteq \mu$ if $\lambda_i \geq  \mu_i$ for all $1 \le i \le \ell(\mu)$. 
A \emph{skew partition} $\lambda/\mu$ is a pair  $(\lambda, \mu)$ of partitions with $\lambda \supseteq \mu$.
We call $|\lambda/\mu| := |\lambda| - |\mu|$ the \emph{size} of $\lambda/\mu$.
In the case where $\lambda \supset \mu \supset \nu$, we say that  $\lambda/\mu$ \emph{extends} $\mu/\nu$.

Given a partition $\lambda$, we define the \emph{Young diagram $\tyd(\lambda)$ of $\lambda$} to be the left-justified array of $n$ boxes, where the $i$th row from the top has $\lambda_i$ boxes for $1 \le i \le k$.
Similarly, given a skew partition $\lambda / \mu$, we define the \emph{Young diagram $\tyd(\lambda / \mu)$ of $\lambda / \mu$} to be the Young diagram $\tyd(\lambda)$ with all boxes belonging to $\tyd(\mu)$ removed.
A Young diagram is called \emph{connected}
if for each pair of consecutive rows, there are at least two boxes (one in each row) which have a common edge.
A skew partition is called \emph{connected} if the corresponding Young diagram is connected, and it is called \emph{basic} if the corresponding Young diagram contains neither empty rows nor empty columns.
In this paper, every skew partition is assumed to be  basic unless otherwise stated.

For two skew partitions $\lambda/\mu$ and $\nu/\kappa$, we define
$\lambda/\mu \star \nu/\kappa$ to be the skew partition whose Young diagram is obtained  by taking a rectangle of empty squares with the same number of rows as  $\tyd(\lambda/\mu)$ and the same number of columns as $\tyd(\nu/\kappa)$, and putting
$\tyd(\nu/\kappa)$ below and $\tyd(\lambda/\mu)$ to the right of this rectangle.
For instance, if $\lambda/\mu = (2,2)$ and $\nu/\kappa = (3,2)/(1)$,
then $\lambda/\mu \star \nu/\kappa = (5,5,3,2)/(3,3,1)$ and
\[
\tyd(\lambda/\mu \star \nu/\kappa) = 
\begin{array}{l}
\begin{ytableau}
\none & \none & \none & \ & \  \\
\none & \none & \none & \ & \  \\
\none & \ & \  \\
\ & \ 
\end{ytableau}
\end{array}.
\]

Given a skew partition $\lambda/\mu$ of size $n$, a \emph{bijective tableau} of shape $\lambda / \mu$ is a filling of $\tyd(\lambda / \mu)$ with distinct entries in $[n]$.
For later use, we denote by $\tau_0^{\lambda/\mu}$ (resp. $\tau_1^{\lambda/\mu}$) the bijective tableau of shape $\lambda / \mu$ obtained by filling $1, 2, \ldots, n$ from right to left starting with the top row (resp. from top to bottom starting with the rightmost column).
If $\lambda/\mu$ is clear in the context, we will drop the superscript $\lambda/\mu$ from $\tau_0^{\lambda/\mu}$ and $\tau_1^{\lambda/\mu}$.
Again, letting $\lambda / \mu = (2,2) \star (3,2) / (1)$,
we have
\[
\tau_0 = \begin{array}{l}
\begin{ytableau}
\none & \none & \none & 2 & 1  \\
\none & \none & \none & 4 & 3  \\
\none & 6 & 5  \\
8 & 7 
\end{ytableau}
\end{array} 
\quad \text{and} \quad
\tau_1 = \begin{array}{l}
\begin{ytableau}
\none & \none & \none & 3 & 1  \\
\none & \none & \none & 4 & 2  \\
\none & 6 & 5  \\
8 & 7 
\end{ytableau}
\end{array}.
\]    
A bijective tableau is referred to as a \emph{standard Young tableau} if the elements in each row are arranged in increasing order from left to right, and the elements in each column are arranged in increasing order from top to bottom.
We denote by $\SYT(\lambda / \mu)$ the set of all standard Young tableaux of shape $\lambda / \mu$.
And, we let $\SYT_n := \bigcup_{\lambda \vdash n} \SYT(\lambda)$.

\subsection{Weak Bruhat orders on the symmetric group}\label{subsec: Symmetric group}

Let $\SG_n$ denote the symmetric group on $[n]$. Every permutation $\sigma \in \SG_n$ can be expressed as a product of simple transpositions $s_i := (i,i+1)$ for $1 \leq i \leq n-1$.
A \emph{reduced expression for $\sigma$} is an expression that represents $\sigma$ in the shortest possible length, and 
the \emph{length $\ell(\sigma)$ of $\sigma$} is the number of simple transpositions in any reduced expression for $\sigma$.
Let
\[
\Des_L(\sigma):= \{i \in [n-1] \mid \ell(s_i \sigma) < \ell(\sigma)\}
\ \  \text{and} \ \  
\Des_R(\sigma):= \{i \in [n-1] \mid \ell(\sigma s_i) < \ell(\sigma)\}.
\]
It is well known that if $\sigma = w_1 w_2 \cdots w_n$ in one-line notation, then
\begin{align*}\label{eq: Des alternating def}
\begin{aligned}
\Des_L(\sigma) & = \{ i \in [n-1] \mid \text{$i$ is right of $i+1$ in $w_1 w_2 \cdots w_n$} \} \quad \text{and} \\
\Des_R(\sigma) & = \{ i \in [n-1] \mid w_i > w_{i+1}
\}.
\end{aligned} 
\end{align*}

The \emph{left weak Bruhat order} $\preceq_L$ 
(resp. \emph{right weak Bruhat order} $\preceq_R$) on $\SG_n$
is the partial order on $\SG_n$ whose covering relation $\preceq_L^{\rmc}$ (resp. $\preceq_R^\rmc$) is given as follows: 
\[
\sigma \preceq_L^\rmc s_i \sigma \text{ if }i \notin \Des_L(\sigma)
\quad \text{(resp. $\sigma \preceq_R^\rmc \sigma s_i$ if  $i \notin \Des_R(\sigma)$).}
\]
Although these two weak Bruhat orders are not identical,  
there exists a poset isomorphism $$(\SG_n, \preceq_L)\to (\SG_n, \preceq_R),\quad \sigma \mapsto \sigma^{-1}.$$
For each $\gamma \in \SG_n$, let
\begin{align*}
\rmInv_L(\gamma) & := \{(i,j) \mid 1 \le i < j \le n \text{ and } \gamma(i) > \gamma(j) \} \quad \text{and} \\
\rmInv_R(\gamma) & :=
\{(\gamma(i), \gamma(j)) \mid 1 \le i < j \le n \text{ and } \gamma(i) > \gamma(j)\}.
\end{align*}
Then, for $\sigma, \rho \in \SG_n$, 
\begin{align*}
& \sigma \preceq_L \rho 
\quad \text{if and only if} \quad
\rmInv_L(\sigma) \subseteq \rmInv_L(\rho) \quad \text{and}\\
& \sigma \preceq_R \rho
\quad \text{if and only if} \quad
\rmInv_R(\sigma) \subseteq \rmInv_R(\rho).
\end{align*}

Given $\sigma, \rho \in \SG_n$, the \emph{left weak Bruhat interval $[\sigma, \rho]_L$} (resp. the \emph{right weak Bruhat interval $[\sigma, \rho]_R$}) denotes the closed interval $\{\gamma \in \SG_n \mid \sigma \preceq_L \gamma \preceq_L \rho \}$ (resp. $\{\gamma \in \SG_n \mid \sigma \preceq_R \gamma \preceq_R \rho \}$) with respect to the left weak Bruhat order (resp. the right weak Bruhat order).

For later use, we introduce the following lemma.

\begin{lemma}{\rm (\cite[Proposition 3.1.6]{05BB})}
\label{lem: interval translation}
For $\sigma,\rho \in \SG_n$ with $\sigma \preceq_R \rho$, 
the map $[\sigma, \rho]_R \ra [\id, \sigma^{-1}\rho]_R, \gamma \mapsto \sigma^{-1} \gamma$ is a poset isomorphism.
Equivalently, for $\sigma,\rho \in \SG_n$ with $\sigma \preceq_L \rho$, the map $[\sigma,\rho]_L \ra [\id, \rho \sigma^{-1}]_L,\gamma \mapsto \gamma \sigma^{-1}$ is a poset isomorphism.
\end{lemma}

Let us collect notations which will be used later.
For $S \subseteq \SG_n$ and $\xi \in \SG_n$, let
\[
S \cdot \xi := \{\gamma \xi \mid \gamma \in S \} 
\quad \text{and} \quad 
\xi \cdot S := \{ \xi \gamma \mid \gamma \in S \}.
\]
We use $w_0$ to denote the longest element in $\SG_n$.
For $I \subseteq [n-1]$, let $\SG_{I}$ be the parabolic subgroup of $\SG_n$ generated by $\{s_i \mid i\in I\}$ and $w_0(I)$ the longest element in $\SG_{I}$.
For $\alpha \models n$, let $w_0(\alpha) := w_0(\set(\alpha))$.
Finally, for $\sigma \in \SG_n$, we let $\sigma^{w_0} := w_0 \sigma w_0$.

\begin{lemma}{\rm (\cite[Theorem 6.2]{88BW})}
\label{lem: same desc int}
For $I \subseteq J \subseteq [n-1]$, we have
\[
\{\sigma \in \SG_n \mid I \subseteq \Des_L(\sigma) \subseteq J\} = [w_0(I), w_0 (J^\rmc) w_0]_R.
\]
\end{lemma}

\subsection{Regular posets and Schur labeled skew shape posets}
\label{Regular posets and Schur labeled skew shape posets}
Let $\sfP_n$ be the set of posets whose underlying set is $[n]$.
Given $P \in \sfP_n$, we write the partial order of $P$ as $\preceq_P$.

\begin{definition}{\rm (\cite[p. 110]{91BW})}\label{def: regular posets}
A poset $P \in \sfP_n$ is said to be \emph{regular} if the following holds:
for all $x,y,z \in [n]$ with
$x \preceq_P z$, if $x < y < z$ or $z<y<x$, then $x \preceq_P y$ or $y \preceq_P z$.
\end{definition}

We denote by $\sfRP_n$ the set of all regular posets in $\sfP_n$.
In the following, we will explain how regular posets can be characterized in terms of left weak Bruhat intervals.

Given $P \in \sfP_n$, 
let
\[
\Sigma_{L}(P) := \{\sigma \in \SG_n \mid \text{$\sigma(i) \leq \sigma(j)$ for all $i, j\in [n]$ with $i \preceq_P j$}\}.
\]
Throughout this paper, $\Sigma_L(P)$ is considered as the set of all linear extensions of $P$ under the correspondence $\sigma \mapsto ([n], \preceq_E)$, where $\preceq_E$ is the total order on $[n]$ given by 
$\sigma^{-1}(1) \preceq_E \sigma^{-1}(2) \preceq_E \cdots \preceq_E \sigma^{-1}(n)$.

\begin{theorem}{\rm (\cite[Theorem 6.8]{91BW})}\label{thm: left interval and regular}
Let $U \subseteq \SG_n$ with $|U| > 1$.
The following conditions are equivalent:
\begin{enumerate}[label = {\rm (\arabic*)}]
    \item $U$ is a left weak Bruhat interval.
    \item $U = \Sigma_L(P)$ for some $P \in \sfRP_n$.
\end{enumerate}  
\end{theorem}
Consider the map
\[
\upeta: \sfP_n \ra \mathscr{P}(\SG_n), \quad P \mapsto \Sigma_L(P),
\]
where $\mathscr{P}(\SG_n)$ is the power set of $\SG_n$.
One can see that $\upeta$ is injective.
Combining this with \cref{thm: left interval and regular}, we obtain a one-to-one correspondence 
\[
\upeta|_{\sfRP_n}: \sfRP_n \ra \rmInt(n), \quad P \mapsto \Sigma_L(P),
\]
where $\rmInt(n)$ is the set of nonempty left weak Bruhat intervals in $\SG_n$.

Next, let us introduce Schur labeled skew shape posets.
Let $\lambda / \mu$ be a skew partition of size $n$.
Given a bijective tableau $\tau$ of shape $\lambda/\mu$, we define $\sfposet(\tau)$ to be 
the poset $([n], \preceq_\tau)$, where
\begin{align}\label{eq: order of poset(tau)}
\text{$i \preceq_\tau j$ if and only if  $i$ lies weakly upper-left of $j$ in $\tau$.}
\footnotemark
\end{align}
\footnotetext{
Note that $\sfposet(\tau) \in \sfP_n$. Following our convention, the partial order $\preceq_\tau$ can also be written as $\preceq_{\sfposet(\tau)}$.
}
The Hasse diagram of $\sfposet(\tau)$ can be obtained by rotating $\tau$ $135^\circ$ counterclockwise.

\begin{example}
Let $\lambda / \mu = (2,2) \star (3,2) / (1)$.
For the bijective tableaux 
$\tau_0$ and $\tau_1$ of shape $\lambda/\mu$
introduced in \cref{subsec: comp and diag}, we have
\[
\def \pp {1}
\def \hp {0.7}
\sfposet(\tau_0) = 
\begin{array}{l}
\begin{tikzpicture}
\node at (0,0) {1};
\node at (1*\hp, -1*\pp) {2};
\node at (1*\hp, 1*\pp) {3};
\node at (2*\hp, 0*\pp) {4};

\node at (3*\hp, 1*\pp) {5};
\node at (4*\hp, 0*\pp) {6};
\node at (5*\hp, 1*\pp) {7};
\node at (6*\hp, 0*\pp) {8};

\draw (0.25*\hp, -0.25*\pp) -- (0.75*\hp, -0.75*\pp);
\draw (1.25*\hp, -0.75*\pp) -- (1.75*\hp, -0.25*\pp);
\draw (0.25*\hp, 0.25*\pp) -- (0.75*\hp, 0.75*\pp);
\draw (1.25*\hp, 0.75*\pp) -- (1.75*\hp, 0.25*\pp);
\draw (3.25*\hp, 0.75*\pp) -- (3.75*\hp, 0.25*\pp);
\draw (4.25*\hp, 0.25*\pp) -- (4.75*\hp, 0.75*\pp);
\draw (5.25*\hp, 0.75*\pp) -- (5.75*\hp, 0.25*\pp);
\end{tikzpicture}
\end{array}
\quad \text{and} \quad
\sfposet(\tau_1) = 
\begin{array}{l}
\begin{tikzpicture}
\node at (0,0) {1};
\node at (1*\hp, -1*\pp) {3};
\node at (1*\hp, 1*\pp) {2};
\node at (2*\hp, 0*\pp) {4};

\node at (3*\hp, 1*\pp) {5};
\node at (4*\hp, 0*\pp) {6};
\node at (5*\hp, 1*\pp) {7};
\node at (6*\hp, 0*\pp) {8};

\draw (0.25*\hp, -0.25*\pp) -- (0.75*\hp, -0.75*\pp);
\draw (1.25*\hp, -0.75*\pp) -- (1.75*\hp, -0.25*\pp);
\draw (0.25*\hp, 0.25*\pp) -- (0.75*\hp, 0.75*\pp);
\draw (1.25*\hp, 0.75*\pp) -- (1.75*\hp, 0.25*\pp);
\draw (3.25*\hp, 0.75*\pp) -- (3.75*\hp, 0.25*\pp);
\draw (4.25*\hp, 0.25*\pp) -- (4.75*\hp, 0.75*\pp);
\draw (5.25*\hp, 0.75*\pp) -- (5.75*\hp, 0.25*\pp);
\end{tikzpicture}
\end{array}.
\]
\end{example}
A \emph{Schur labeling of shape $\lambda/ \mu$} is a bijective tableau of shape $\lambda / \mu$ such that
the entries in each row decrease from left to right and the entries in each column increase from top to bottom.
Let $\sfS(\lambda/\mu)$ be the set of all Schur labelings of shape $\lambda / \mu$. 
Since $\tau_0$ and $\tau_1$ are Schur labelings of shape $\lambda/\mu$, $\sfS(\lambda/\mu)$ is nonempty.
Set
\[
\sfSP(\lambda/\mu) := \{\sfposet(\tau) \mid \tau \in \sfS(\lambda/\mu)
\}
\quad \text{and} \quad
\sfSP_n := \bigcup_{
|\lambda/\mu| = n
} \sfSP(\lambda/\mu).
\]

\begin{definition}\label{def: Schur labeled skew shape posets}
A poset $P \in \sfP_n$ is said to be a \emph{Schur labeled skew shape poset} if it is contained in $\sfSP_n$.
\end{definition}

\begin{remark}
In some papers, for instance \cite{93M, 06McNamara}, authors used a different convention than ours for Schur labeling.
We adopt the definition of Schur labeling used in Stanley's paper \cite{72Stanley}.
\end{remark}

For simplicity, we set $\sfRSP_n := \sfRP_n \cap \sfSP_n$.

\subsection{The $0$-Hecke algebra and the quasisymmetric characteristic}\label{subsec: 0-Hecke alg and QSym}
The $0$-Hecke algebra $H_n(0)$ is the associative $\C$-algebra with $1$ generated by $\pi_1,\pi_2,\ldots,\pi_{n-1}$ subject to the following relations:
\begin{align*}
\pi_i^2 &= \pi_i \quad \text{for $1\le i \le n-1$},\\
\pi_i \pi_{i+1} \pi_i &= \pi_{i+1} \pi_i \pi_{i+1}  \quad \text{for $1\le i \le n-2$},\\
\pi_i \pi_j &=\pi_j \pi_i \quad \text{if $|i-j| \ge 2$}.
\end{align*}
For each $1 \leq i \leq n-1$, let $\opi_i := \pi_i - 1$. 
Then, $\{\opi_i \mid i = 1, 2, \ldots, n-1\}$ is also a generating set of $H_n(0)$.

For any reduced expression $s_{i_1} s_{i_2} \cdots s_{i_p}$ for $\sigma \in \SG_n$, let 
\[
\pi_{\sigma} := \pi_{i_1} \pi_{i_2 } \cdots \pi_{i_p} \quad \text{and} \quad \opi_{\sigma} := \opi_{i_1} \opi_{i_2} \cdots \opi_{i_p}.
\]
It is well known that these elements are independent of the choices of reduced expressions, and both $\{\pi_\sigma \mid \sigma \in \SG_n\}$ and $\{\opi_\sigma \mid \sigma \in \SG_n\}$ are $\mathbb C$-bases for $H_n(0)$.

According to \cite{79Norton}, there are $2^{n-1}$ pairwise nonisomorphic irreducible $H_n(0)$-modules which are naturally indexed by compositions of $n$.
To be precise, for each composition $\alpha$ of $n$, there exists an irreducible $H_n(0)$-module $\bfF_{\alpha}:=\C v_{\alpha}$ endowed with the $H_n(0)$-action defined as follows: for each $1 \le i \le n-1$,
\[
\pi_i \cdot v_\alpha = \begin{cases}
0 & i \in \set(\alpha),\\
v_\alpha & i \notin \set(\alpha).
\end{cases}
\]

Let $\Hnmod$ be the category of finite dimensional left $H_n(0)$-modules and $\calR(H_n(0))$ the $\Z$-span of the set of (representatives of) isomorphism classes of modules in $\Hnmod$.
We denote by $[M]$ the isomorphism class corresponding to an $H_n(0)$-module $M$. 
The \emph{Grothendieck group $\calG_0(H_n(0))$ of $\Hnmod$} is the quotient of $\calR(H_n(0))$ modulo the relations $[M] = [M'] + [M'']$ whenever there exists a short exact sequence $0 \ra M' \ra M \ra M'' \ra 0$. 
The equivalence classes of the irreducible $H_n(0)$-modules form a $\Z$-basis for $\calG_0(H_n(0))$. Let
\[
\calG := \bigoplus_{n \ge 0} \calG_0(H_n(0)).
\]

Let us review the connection between $\calG$ and the ring $\Qsym$ of quasisymmetric functions.
For the definition of quasisymmetric functions, see~\cite[Section 7.19]{99Stanley}.
For a composition $\alpha$, the \emph{fundamental quasisymmetric function} $F_\alpha$, which was firstly introduced in~\cite{84Gessel}, is defined by
\[
F_\varnothing = 1 
\quad \text{and} \quad
F_\alpha = \sum_{\substack{1 \le i_1 \le i_2 \le \cdots \le i_n \\ i_j < i_{j+1} \text{ if } j \in \set(\alpha)}} x_{i_1}x_{i_2} \cdots x_{i_n} \quad \text{if $\alpha \neq \varnothing$}.
\]
It is known that $\{F_\alpha \mid \text{$\alpha$ is a composition}\}$ is a $\mathbb Z$-basis for $\Qsym$.
When $M$ is an $H_m(0)$-module and $N$ is
an $H_n(0)$-module, we write $M \boxtimes N$ for the induction product of $M$ and $N$, that is,
\begin{align*}
M \boxtimes N := M \otimes N \uparrow_{H_m(0) \otimes H_n(0)}^{H_{m+n}(0)}.
\end{align*}
Here, $H_m(0) \otimes H_n(0)$ is viewed as the subalgebra of $H_{m+n}(0)$ generated by $\{\pi_i \mid i \in [m+n-1] \setminus \{m\} \}$.
The induction product induces a multiplication on $\calG$.
It was shown in \cite{96DKLT} that the linear map
\begin{equation*}\label{quasi characteristic}
\ch : \calG \ra \Qsym, \quad [\bfF_{\alpha}] \mapsto F_{\alpha},
\end{equation*}
called \emph{quasisymmetric characteristic}, is a ring isomorphism.
Indeed, it turns out to be a Hopf algebra isomorphism when $\calG$ has the comultiplication induced from restriction. 

It is well known that $H_n(0)$ has the automorphisms $\autotheta$ and $\autophi$, as well as the anti-automorphism $\autochi$, defined in the following manner:
\begin{align*}
&\autophi: H_n(0) \ra H_n(0), \quad \pi_i \mapsto \pi_{n-i} \quad \text{for $1 \le i \le n-1$},\\
&\autotheta: H_n(0) \ra H_n(0), \quad \pi_i \mapsto - \opi_i \quad \text{for $1 \le i \le n-1$}, \\
&\autochi: H_n(0) \ra H_n(0), \quad \pi_i \mapsto \pi_i \quad \text{for $1 \le i \le n-1$}
\end{align*}
(for instance, see \cite{05Fayers, 97KT, 97LLT, 84LS}).
These maps commute with each other.

Note that an automorphism $\mu$ of $H_n(0)$  induces a covariant functor
\[
\bfT^{+}_\mu: \Hnmod \ra \Hnmod
\]
called the  \emph{$\mu$-twist}.
Similarly, an anti-automorphism $\nu$ of $H_n(0)$ induces a contravariant functor \[
\bfT^{-}_\nu: \Hnmod \ra \Hnmod
\]
called the \emph{$\nu$-twist}.
For the precise definitions of $\bfT^{+}_\mu$ and $\bfT^{-}_\nu$, see \cite[Subsection 3.4]{22JKLO}.
In \cite[Proposition 3.3.]{05Fayers}, it was shown that
\[
\bfT^{+}_\autophi(\bfF_\alpha) = \bfF_{\alpha^\rmr},
\quad
\bfT^{+}_\autotheta(\bfF_\alpha) = \bfF_{\alpha^\rmc},
\quad \text{and} \quad
\bfT^{-}_\autochi(\bfF_\alpha) = \bfF_{\alpha}
\]
for  $\alpha \models n$.
Let $\uprho$ and $\uppsi$ be the automorphisms  of $\Qsym$ defined by
\begin{align*}
\uprho(F_\alpha) = F_{\alpha^\rmr}
\quad \text{and} \quad
\uppsi(F_\alpha) = F_{\alpha^\rmc}
\end{align*}
for every composition $\alpha$.
For a finite dimensional $H_n(0)$-module $M$, it holds that 
\begin{align}\label{lifts of involutions on QSYM}
\begin{array}{c}
\ch( [\bfT^{+}_\autophi(M)] ) = \uprho \circ \ch([M]), 
\quad
\ch ( [\bfT^{+}_\autotheta(M)] )  = \uppsi \circ \ch([M]),\\[1ex]
\text{and} \quad
\ch ( [\bfT^{-}_\autochi(M)] ) =  \ch([M]).
\end{array}
\end{align}

\subsection{Modules arising from posets and weak Bruhat interval modules}
Let $P\in \sfP_n$. In \cite[Definition 3.18]{02DHT}, Duchamp, Hivert, and Thibon defined a right $H_n(0)$-module $M_P$ associated with $P$. In this paper, we are primarily concerned with left modules, so we introduce a left $H_n(0)$-module, denoted as $\sfM_P$, associated with $P$.

\begin{definition}\label{def: poset module}
Let $P \in \sfP_n$. 
Define $\sfM_P$ to be the left $H_n(0)$-module with $\C \Sigma_{L}(P)$ as the underlying space and with the $H_n(0)$-action defined by
\begin{align}\label{left action}
\pi_{i} \cdot \gamma :=
\begin{cases}
\gamma & \text{if $i \in \Des_L(\gamma)$}, \\
0 & \text{if $i \notin \Des_L(\gamma)$ and $s_i\gamma \notin \Sigma_{L}(P)$,} \\
s_i \gamma & \text{if $i \notin \Des_L(\gamma)$ and $s_i\gamma \in \Sigma_{L}(P)$}.
\end{cases} 
\end{align}
\end{definition}

One can see that the $H_n(0)$-action provided in \cref{left action} is well-defined through a slight modification of the proof in \cite[Subsection 3.9]{02DHT} that $M_P$ is a well-defined right $H_n(0)$-module.
Indeed, there is a close connection between $\sfM_P$ and $M_P$.
Let $\modHn$ be the category of finite dimensional right $H_n(0)$-modules.
In \cite[Subsection 4.3]{23CKO}, the authors introduced a contravariant functor 
$$\mathcal{F}_n: \Hnmod \ra \modHn$$
that preserves quasisymmetric characteristics.
\footnote{In \cite[Subsection 4.3]{23CKO}, the authors considered both left and right quasisymmetric characteristics because they were simultaneously working with two categories, $\Hnmod$ and $\modHn$.}
Using this functor, it is not difficult to see that
\[
\sfM_P \cong \bfT^{+}_\autotheta \circ \bfT^{-}_\autochi \circ \mathcal{F}_n^{-1}(M_P).
\]

Since the underlying set of $P$ is $[n]$, we can regard $P$ as the labeled poset $(P, \omega)$ with the labeling $\omega: P \ra [n]$ given by $\omega(i) = i$.
Under this consideration, 
a map $f: [n] \ra \Z_{\ge 0}$ is called a \emph{$P$-partition} if it satisfies the following conditions:
\begin{enumerate}[label = {\rm (\arabic*)}]
\item 
If $i \preceq_P j$, then $f(i) \le f(j)$.
\item 
If $i \preceq_P j$ and $i > j$, then $f(i) < f(j)$.
\end{enumerate}
We define the \emph{$P$-partition generating function $K_P$}  of $P$ by
\[
K_{P} := \sum_{f:  \text{$P$-partition}} x_1^{|f^{-1}(1)|} x_2^{|f^{-1}(2)|} \cdots.
\]

\begin{theorem}\label{thm: char of Mp}
For $P \in \sfP_n$, the following hold. 
\begin{enumerate}[label = {\rm (\arabic*)}]
\item $\ch([\sfM_P]) = \uppsi(K_{P})$.
\item 
  If $P \in \sfSP(\lambda/\mu)$ for a skew partition $\lambda/\mu$, 
  then 
$\ch([\sfM_P]) = s_{\lambda/\mu}$.
\end{enumerate}
\end{theorem}
\begin{proof}
(1)
It was shown in~\cite[Theorem 3.21(i)]{02DHT} that the (right) quasisymmetric characteristic of $M_P$ is given by $K_P$. 
Since $\sfM_P \cong \bfT^{+}_\autotheta \circ \bfT^{-}_\autochi \circ \mathcal{F}_n^{-1}(M_P)$, the assertion follows from \cref{lifts of involutions on QSYM}.

(2)
In the same manner as in \cite[Subsection 7.19]{99Stanley}, one can see that $K_P = s_{\lambda^\rmt/\mu^\rmt}$.
Here, $\lambda^\rmt$ and $\mu^\rmt$ are the transposes of $\lambda$ and $\mu$, respectively.
Now the assertion can be derived from the well known identity 
$\uppsi(s_{\lambda/\mu}) = s_{\lambda^\rmt / \mu^\rmt}$
(for instance, see \cite[Subsection 3.6]{13LMvW}).
\end{proof}

Next, let us introduce weak Bruhat interval modules, which were introduced by Jung, Kim, Lee, and Oh~\cite{22JKLO} to provide a unified method for dealing with  $H_n(0)$-modules constructed using tableaux.

\begin{definition}{\rm (\cite[Definition 1]{22JKLO})}\label{def: WBI action}
For each left weak Bruhat interval $[\sigma, \rho]_L$ in $\SG_n$,
define $\sfB([\sigma, \rho]_L)$ (simply, $\sfB(\sigma, \rho)$) to be  the $H_n(0)$-module with  $\C[\sigma,\rho]_L$ as the underlying space and with the $H_n(0)$-action defined by
\begin{align*}
\pi_i \cdot \gamma := \begin{cases}
\gamma & \text{if $i \in \Des_L(\gamma)$}, \\
0 & \text{if $i \notin \Des_L(\gamma)$ and $s_i\gamma \notin [\sigma,\rho]_L$,} \\
s_i \gamma & \text{if $i \notin \Des_L(\gamma)$ and $s_i\gamma \in [\sigma,\rho]_L$}.
\end{cases} 
\end{align*}
This module is called the \emph{weak Bruhat interval module associated to $[\sigma,\rho]_L$}.
\end{definition}

We can deduce from \cref{thm: left interval and regular} that for every $[\sigma, \rho]_L \in \rmInt(n)$, there exists a unique poset $P\in \sfP_n$ such that $\Sigma_L(P)=[\sigma, \rho]_L$. Since both $\sfB(\sigma, \rho)$ and $\sfM_P$ share $[\sigma, \rho]_L$ as their basis and exhibit identical $H_n(0)$-actions on this set, we can conclude that $\sfB(\sigma, \rho)$ is indeed equal to $\sfM_P$.

\begin{remark}
The weak Bruhat interval modules are equipped with the structure of semi-combinatorial $H_n(0)$-modules due to Hivert, Novelli, and Thibon~\cite{06HNT} and also that of diagram modules due to Searles \cite{22Searles}. 
More precisely,

\begin{enumerate}[label = {\rm (\arabic*)}]
\item $\sfB(\sigma,\rho)$ is the semi-combinatorial $H_n(0)$-module associated to the Yang-Baxter interval $[Y_{\sigma}(\id), Y_{\rho}(\id)]$, and

\item it was shown in \cite[Subsection 7.3]{22Searles} that $\sfB(\sigma,\rho)$ is isomorphic to a diagram module $\widehat{\mathbf{N}}_{\mathsf{StdTab}(D)}$.
\end{enumerate}

\end{remark}

\section{The weak Bruhat interval structure of $\Sigma_L(P)$ for $P \in \sfRSP_n$}\label{sec: int str for regular Schur}

Let $P \in \sfRSP_n$.
In this section, we explicitly describe the left weak Bruhat interval $\Sigma_L(P)$ in terms of reading words of standard Young tableaux.
To begin with, we introduce readings for bijective tableaux.

\begin{definition}\label{def: read tau}
Let $\tau$ be a bijective tableau of shape $\lambda / \mu$.
The \emph{$\tau$-reading} is the map 
$$
\rmread_\tau: 
\{\text{bijective tableaux of shape $\lambda / \mu$}\} \ra \SG_n, \quad T \mapsto \rmread_\tau(T),
$$ 
where $\rmread_\tau(T)$ is the permutation in $\SG_n$ given by
$\rmread_\tau(T)(k) = T_{\tau^{-1}(k)}$ for $1 \le k \le n$.
We call $\rmread_\tau(T)$ the \emph{$\tau$-reading word of $T$}.
\end{definition}

Given a bijective tableaux $T$ of shape $\lambda / \mu$,  the permutation $\rmread_{\tau}(T)$ in one-line notation can be obtained by reading the entries of $T$ in the order given by $\tau^{-1}(1),  \tau^{-1}(2), \ldots, \tau^{-1}(n)$. For instance, if $\tau = \begin{array}{l}\begin{ytableau}
4  & 2 &  3\\
5  &  1
\end{ytableau}
\end{array}
$ and $T = \begin{array}{l}\begin{ytableau}
1  & 3 &  4\\
2  & 5
\end{ytableau}\end{array}$, then
$\rmread_{\tau}(T) = 53412$.
With this definition, we have the following lemma.

\begin{lemma}\label{lem: identification of P and S}
For any bijective tableau $\tau$ of shape $\lambda / \mu$, 
$\Sigma_L(\sfposet(\tau)) = \rmread_\tau(\SYT(\lambda/\mu))$.
\end{lemma}

\begin{proof}
We first show that $\rmread_\tau(\SYT(\lambda/\mu)) \subseteq \Sigma_L(\sfposet(\tau))$.
To do this, take any $T \in \SYT(\lambda/\mu)$ and $i,j \in [n]$ with $i \preceq_\tau j$ (for the definition of $\preceq_\tau$, see \cref{eq: order of poset(tau)}).
Let $B_1$ and $B_2$ be the boxes in $\tyd(\lambda/\mu)$ such that $\tau_{B_1} = i$ and $\tau_{B_2} = j$.
Since $i \preceq_\tau j$, 
$B_1$ is weakly upper-left of $B_2$ in $\tyd(\lambda/\mu)$.
This implies that 
$$
\rmread_\tau(T)(i) = T_{B_1} \leq T_{B_2} = \rmread_\tau(T)(j).
$$
Therefore, $\rmread_\tau(T) \in \Sigma_L(\sfposet(\tau))$.

To complete the proof, let us show that $|\Sigma_L(\sfposet(\tau))| = |\rmread_\tau(\SYT(\lambda/\mu))|$.
Note that 
\begin{align*}
\ch([\sfM_{\sfposet(\tau^{\lambda/\mu}_0)}]) = \sum_{\gamma \in \Sigma_L(\sfposet(\tau^{\lambda/\mu}_0))} F_{\comp(\Des_L(\gamma))^\rmc}
\quad \text{and} \quad
\ch([\sfM_{\sfposet(\tau^{\lambda/\mu}_0)}]) =  s_{\lambda/\mu},
\end{align*}
where the second equality follows from 
\cref{thm: char of Mp}(2).
Putting these together with the well known equality $s_{\lambda / \mu} = \sum_{T \in \SYT(\lambda / \mu)} F_{\comp(T)}$, we have
\begin{align}\label{eq: F expansion equal}
\sum_{\gamma \in \Sigma_L(\sfposet(\tau^{\lambda/\mu}_0))} F_{\comp(\Des_L(\gamma))^\rmc} 
= \sum_{T \in \SYT(\lambda/\mu)} F_{\comp(T)}.
\end{align}
Here, $\comp(T) = \comp(\{i \in [n-1] \mid \text{$i$ is weakly right of $i+1$ in $T$}\})$.
As a consequence of \cref{eq: F expansion equal}, we have
\[
|\Sigma_L(\sfposet(\tau))| = |\Sigma_L(\sfposet(\tau^{\lambda/\mu}_0))| 
= |\SYT(\lambda/\mu)| = |\rmread_\tau(\SYT(\lambda/\mu))|.
\qedhere
\]
\end{proof}

The purpose of the remainder of this section is to describe the minimum and maximum of $\Sigma_L(P)$ with respect to $\preceq_L$.

As a first step, we deal with a characterization of regular Schur labeled skew shape posets.
For this purpose, the following definition is necessary.  
\begin{definition}\label{def: distinguished poset}
Let $\lambda/\mu$ be a skew partition of size $n$.
A Schur labeling $\tau$ of shape $\lambda / \mu$ is said to be \emph{distinguished} if
$\tau_B \geq \tau_{B'}$ whenever $B$ is weakly below and weakly left of $B'$ for boxes $B, B' \in \tyd(\lambda/\mu)$.
\end{definition}

\begin{example}
Consider the  Schur labelings of shape $(3,2,2)/(1)$
\[
\tau_0^{(3,2,2)/(1)} = 
\begin{array}{l}\begin{ytableau}
\none & \none & 1\\ 
3 & 2 \\
5 & 4
\end{ytableau}
\end{array}, 
 \quad
\tau_1^{(3,2,2)/(1)} = 
\begin{array}{l}\begin{ytableau}
\none & \none & 1\\ 
4 & 2 \\
5 & 3
\end{ytableau}
\end{array}, 
\quad \text{and} \quad
\tau = 
\begin{array}{l}\begin{ytableau}
\none & \none & 2\\ 
3 & 1 \\
5 & 4
\end{ytableau}
\end{array}.
\]
One sees that $\tau_0^{(3,2,2)/(1)}$ and $\tau_1^{(3,2,2)/(1)}$ are distinguished, 
whereas
$\tau$ is non-distinguished since $1$ appears weakly below and weakly left of $2$ in $\tau$.
\end{example}

Let $\sfDS(\lambda / \mu)$ be the set of all distinguished Schur labelings of shape $\lambda / \mu$.
For any Schur labeling $\tau$, let $\sfcnt_i(\tau)$ be the set of entries in the $i$th connected component of $\tau$ from the top.
For each $P \in \sfSP_n$, there exists a unique Schur labeling $\tau$ such that 
\begin{equation}\label{eq: def of tau_P}
\text{\parbox{.9\textwidth}{
\begin{enumerate}[label = {\rm (\roman*)}, leftmargin = 4ex]
\item $\sh(\tau)$ is basic,
\item $\sfposet(\tau) = P$, and
\item $\mathrm{min}(\sfcnt_i(\tau)) < \mathrm{min}(\sfcnt_j(\tau))$ for $1 \leq i < j \leq k$, where $k$ is the number of connected components of $P$.
\end{enumerate}
}}
\end{equation}
We denote this Schur labeling as $\tau_P$.
One can easily see that for $P \in \sfSP_n$, $\tau_P$ is distinguished if and only if every connected component of $\tau_P$ is filled with consecutive integers.

\begin{example}
Given two Schur labeled skew shape posets
\[ 
\def \pp {1}
\def \hp {0.7}
P = 
\begin{array}{l}
\begin{tikzpicture}
\node at (-1.5*\hp, -0.5*\pp]) {3};
\node at (0,0) {4};
\node at (1*\hp, -1*\pp) {5};
\node at (2*\hp, 0*\pp) {6};
\node at (3.5*\hp, 0*\pp) {1};
\node at (4.5*\hp, -1*\pp) {2};
\draw (0.25*\hp, -0.25*\pp) -- (0.75*\hp, -0.75*\pp);
\draw (1.25*\hp, -0.75*\pp) -- (1.75*\hp, -0.25*\pp);
\draw (3.75*\hp, -0.25*\pp) -- (4.25*\hp, -0.75*\pp);
\end{tikzpicture}
\end{array}
\quad \text{and} \quad
Q = 
\begin{array}{l}
\begin{tikzpicture}
\node at (-1.5*\hp, -0.5*\pp]) {4};
\node at (0,0) {1};
\node at (1*\hp, -1*\pp) {3};
\node at (2*\hp, 0*\pp) {5};
\node at (3.5*\hp, 0*\pp) {2};
\node at (4.5*\hp, -1*\pp) {6};
\draw (0.25*\hp, -0.25*\pp) -- (0.75*\hp, -0.75*\pp);
\draw (1.25*\hp, -0.75*\pp) -- (1.75*\hp, -0.25*\pp);
\draw (3.75*\hp, -0.25*\pp) -- (4.25*\hp, -0.75*\pp);
\end{tikzpicture}
\end{array},
\]
we have that
\[
\tau_P = 
\begin{array}{l}
\begin{ytableau}
\none & \none & \none & 2 & 1 \\
\none & \none & 3 \\
5 & 4 \\
6    
\end{ytableau}
\end{array}
\quad\; \text{and} \quad\;
\tau_Q = 
\begin{array}{l}
\begin{ytableau}
\none & \none & \none & 3 & 1 \\
\none & \none & \none & 5\\
\none & 6 & 2\\
4
\end{ytableau}
\end{array}.
\]
\end{example}

\begin{lemma}\label{lem: characterization: regular}
For $P \in \sfSP_n$, $P$ is regular if and only if $\tau_P$ is distinguished.
\end{lemma}

\begin{proof}
To prove the ``only if'' part, assume that $P$ is a regular Schur labeled skew shape poset and $\lambda / \mu$ is the shape of $\tau_P$.
We claim that every connected component of $\tau_P$ is filled with consecutive integers.
Take an arbitrary connected component $\sfC$ of $\tau_P$.
Let $B_1$ be the box at the top of the rightmost column of $\sfC$ and $B_2$ the box at the bottom of the leftmost column of $\sfC$. 
Then, we may choose boxes $A_0:= B_1, A_2, A_3, \ldots, A_k := B_2$ satisfying that for all $1 \le i \le k-1$, $A_{i+1}$ is  weakly below and weakly left of $A_i$ and $A_{i}, A_{i+1}$ are in the same row or in the same column.
Let $m \in [(\tau_P)_{B_1}, (\tau_P)_{B_2}]$.
Then, there exists a unique index $1 \leq i \leq k-1$ such that $(\tau_P)_{A_i} \leq m \leq (\tau_P)_{A_{i+1}}$.
Since $P$ is regular, one of the following holds:
\begin{enumerate}[label = {\rm (\roman*)}]
\item 
If $(\tau_P)_{A_i} \preceq_P (\tau_P)_{A_{i+1}}$, then $(\tau_P)_{A_i} \preceq_P m$ or $m \preceq_P (\tau_P)_{A_{i+1}}$.
\item 
If $(\tau_P)_{A_{i+1}} \preceq_P (\tau_P)_{A_{i}}$, then $(\tau_P)_{A_{i+1}} \preceq_P m$ or $m \preceq_P (\tau_P)_{A_{i}}$.
\end{enumerate}
It follows that $(\tau_P)_{A_i}$, $m$, and $(\tau_P)_{A_{i+1}}$ appear in the same connected component, that is, $m$ appears in $\sfC$.
Thus, $\tau_P$ is distinguished.

Next, to prove the ``if'' part, assume that $\tau_P$ is a distinguished Schur labeling and $\lambda / \mu$ is the shape of $\tau_P$.
Let $B_1, B_2 \in \tyd(\lambda/\mu)$ with $(\tau_P)_{B_1} \prec_{\tau_P} (\tau_P)_{B_2}$.
By the definition of $\prec_{\tau_P}$, $B_1$ and $B_2$ are in the same connected component.
In order to establish the regularity of $P$, we need to prove that either $(\tau_P)_{B_1} \preceq_{\tau_P} (\tau_P)_C$ or $(\tau_P)_{C} \preceq_{\tau_P} (\tau_P)_{B_2}$ for all $C \in \tyd(\lambda/\mu)$ satisfying $(\tau_P)_{B_1} < (\tau_P)_C < (\tau_P)_{B_2}$ or $(\tau_P)_{B_1} > (\tau_P)_C > (\tau_P)_{B_2}$.

Assume that there exists $C \in \tyd(\lambda/\mu)$ such that $(\tau_P)_{B_1} < (\tau_P)_C < (\tau_P)_{B_2}$.
Since $\tau_P$ is a Schur labeling and $(\tau_P)_{B_1} \prec_{\tau_P} (\tau_P)_{B_2}$, the inequality $(\tau_P)_{B_1} < (\tau_P)_{B_2}$ implies that $B_2$ is strictly below $B_1$.
In addition, since $\tau_P$ is distinguished and $(\tau_P)_{B_1},(\tau_P)_{B_2}$ appear in the same connected component in $\tau_P$, $(\tau_P)_{C}$ appears in the same connected component with them in $\tau_P$.
Suppose for the sake of contradiction that $(\tau_P)_{B_1} \not\preceq_{\tau_P} (\tau_P)_C$ and $(\tau_P)_{C} \not\preceq_{\tau_P} (\tau_P)_{B_2}$.
Then $C$ satisfies one of the following conditions:
\begin{enumerate}[label = {\rm (\roman*)}]
\item $C$ is strictly above $B_1$ and strictly right of $B_2$.
\item $C$ is strictly left of $B_1$ and strictly below $B_2$.
\end{enumerate}
However, since $\tau_P$ is a Schur labeling and $(\tau_P)_{B_1} < (\tau_P)_C$, $C$ cannot satisfy (i).
Similarly, since $\tau_P$ is a Schur labeling and $(\tau_P)_{C} < (\tau_P)_{B_2}$, $C$ cannot satisfy (ii).
Therefore, $(\tau_P)_{B_1} \preceq_{\tau_P} (\tau_P)_C$ or $(\tau_P)_{C} \preceq_{\tau_P} (\tau_P)_{B_2}$.
In a similar way, one can show that if there exists $C \in \tyd(\lambda/\mu)$ such that $(\tau_P)_{B_1} > (\tau_P)_C > (\tau_P)_{B_2}$, then $(\tau_P)_{B_1} \preceq_{\tau_P} (\tau_P)_C$ or $(\tau_P)_{C} \preceq_{\tau_P} (\tau_P)_{B_2}$.
Thus, $P$ is regular.
\end{proof}

Note that $\tau_{\sfposet(\tau)} = \tau$ for any $\tau \in \sfDS(\lambda / \mu)$.
Considering this property together with \cref{lem: characterization: regular}, one can see that 
the map \begin{align}\label{eq: bijection btw RSP and DS}
\Phi: 
\sfRSP_n \ra \bigcup_{|\lambda/\mu| = n} \sfDS(\lambda/\mu),
 \quad
 P \mapsto \tau_P
 \end{align}
is a bijection and its inverse is given by
$\tau \mapsto \sfposet(\tau)$.

As a second step, we provide a lemma that will be used throughout this paper.

\begin{lemma}\label{lem: fixed T int}
For $T \in \SYT(\lambda/\mu)$, 
$\{\rmread_\tau(T) \mid \tau \in \sfDS(\lambda/\mu)\} = [\rmread_{\tau_0}(T), \rmread_{\tau_1}(T)]_R$.
\end{lemma}

\begin{proof}
Let us show the inclusion $\{\rmread_\tau(T) \mid \tau \in \sfDS(\lambda/\mu)\} \subseteq [\rmread_{\tau_0}(T), \rmread_{\tau_1}(T)]_R$.
This can be done by proving $\rmread_{\tau_0}(T) \preceq_R \rmread_{\tau}(T)$ and $\rmread_{\tau}(T) \preceq_R \rmread_{\tau_1}(T)$ for all $\tau \in \sfDS(\lambda/\mu)$.
Since the method of proof for the latter inequality is essentially the same as that for the former one, we omit the proof for the latter inequality.
Let $\tau \in \sfDS(\lambda/\mu)$ and
$(i,j) \in \rmInv_R(\rmread_{\tau_0}(T))$.
Since $i > j$ and $\rmread_{\tau_0}(T)^{-1}(i) < \rmread_{\tau_0}(T)^{-1}(j)$,
the box $T^{-1}(j)$ is placed strictly left and weakly below $T^{-1}(i)$.
This, together with the definition of distinguished Schur labeling, implies that $\tau_{T^{-1}(i)} < \tau_{T^{-1}(j)}$, equivalently, $\rmread_{\tau}(T)^{-1}(i) < \rmread_{\tau}(T)^{-1}(j)$.
It follows that $(i,j) \in \rmInv_R(\rmread_{\tau}(T))$.
Since we chose an arbitrary $(i,j) \in \rmInv_R(\rmread_{\tau_0}(T))$, we have the inclusion $\rmInv_R(\rmread_{\tau_0}(T)) \subseteq \rmInv_R(\rmread_{\tau}(T))$.
Therefore, $\rmread_{\tau_0}(T) \preceq_R \rmread_{\tau}(T)$.

Let us show the opposite inclusion $\{\rmread_\tau(T) \mid \tau \in \sfDS(\lambda/\mu)\} \supseteq [\rmread_{\tau_0}(T), \rmread_{\tau_1}(T)]_R$.
Since $\{\rmread_\tau(T) \mid \text{$\tau$ is a bijective tableau of shape $\lambda / \mu$} \}$ is equal to $\SG_n$ as a set,
the inclusion can be obtained by proving that $\rmread_\tau(T) \notin [\rmread_{\tau_0}(T), \rmread_{\tau_1}(T)]_R$ for all bijective tableaux $\tau$ of shape $\lambda/\mu$ with $\tau \notin \sfDS(\lambda/\mu)$.
To prove it, choose an arbitrary bijective tableau $\tau$ of shape $\lambda/\mu$ with $\tau \notin \sfDS(\lambda/\mu)$.
Since $\tau \notin \sfDS(\lambda/\mu)$, there exists $ 1 \leq i < j \leq n$ such that $i$ is weakly below and weakly left of $j$ in $\tau$.
Set $x := T_{\tau^{-1}(i)}$ and $y := T_{\tau^{-1}(j)}$.
Then, $x$ appears left of $y$ in $\rmread_\tau(T)$.
On the other hand, since $\tau_0, \tau_1 \in \sfDS(\lambda/\mu)$, 
we have $(\tau_0)_{\tau^{-1}(i)} > (\tau_0)_{\tau^{-1}(j)}$ and $(\tau_1)_{\tau^{-1}(i)} > (\tau_1)_{\tau^{-1}(j)}$, which implies that $x$ appears right of $y$ in both $\rmread_{\tau_0}(T)$ and $\rmread_{\tau_1}(T)$.
If $x < y$, then $(y, x) \in \rmInv_R(\rmread_{\tau_0}(T))$ and $(y, x) \notin \rmInv_R(\rmread_\tau(T))$, thus $\rmInv_R(\rmread_{\tau_0}(T)) \not\subseteq \rmInv_R(\rmread_{\tau}(T))$.
Similarly, if $x > y$, then 
$\rmInv_R(\rmread_\tau(T)) \not\subseteq \rmInv_R(\rmread_{\tau_1}(T))$.
Hence, $\rmread_\tau(T) \notin [\rmread_{\tau_0}(T), \rmread_{\tau_1}(T)]_R$.
\end{proof}

As a last step, we define two specific standard Young tableaux. For a skew partition $\lambda/\mu$ of size $n$, let $T_{\lambda/\mu}$(resp. $T'_{\lambda/\mu}$) be the standard Young tableau obtained by filling $\tyd(\lambda/\mu)$ by $1, 2, \ldots, n$ from left to right starting with the top row (resp. from top to bottom starting with leftmost column).

\begin{example}\label{eg: source and sink tab}
Let $\lambda / \mu = (2,2) \star (3,2) / (1)$.
Then 
\begin{align*}
T_{\lambda/\mu} = 
\begin{array}{l}
\begin{ytableau}
\none & \none & \none & 1 & 2 \\
\none & \none & \none & 3 & 4 \\
\none & 5 & 6 \\
7 & 8
\end{ytableau}
\end{array}
\quad \text{and} \quad
T'_{\lambda/\mu} =
\begin{array}{l}
\begin{ytableau}
\none & \none & \none & 5 & 7 \\
\none & \none & \none & 6 & 8 \\
\none & 2 & 4 \\
1 & 3
\end{ytableau}
\end{array}.
\end{align*}
\end{example}

Now, we are ready prove the main theorem of this section.

\begin{theorem}\label{thm: interval descriptrion for regular skew schur poset}
Let $P \in \sfRSP_n$ and $\lambda/\mu = \sh(\tau_P)$.
Then, 
\[
\Sigma_L(P) = [\rmread_{\tau_P}(T_{\lambda/\mu}), \rmread_{\tau_P}(T'_{\lambda/\mu})]_L.
\]
\end{theorem}

\begin{proof}
Due to \cref{thm: left interval and regular} and \cref{lem: identification of P and S},
it suffices to show that $\rmread_{\tau_P}(T_{\lambda/\mu})$ is minimal and $\rmread_{\tau_P}(T'_{\lambda/\mu})$ is maximal in $\rmread_{\tau_P}(\SYT(\lambda/\mu))$ with respect to $\preceq_L$.
Let $T \in \SYT(\lambda / \mu)$.
In the case where $\tau_P = \tau_0$,
one can easily see that
\begin{align}\label{eq: read tau0 min and max}
\rmInv_L(\rmread_{\tau_0}(T_{\lambda/\mu})) 
\subseteq \rmInv_L(\rmread_{\tau_0}(T))
\subseteq
\rmInv_L(\rmread_{\tau_0}(T_{\lambda/\mu}')). 
\end{align}
In the case where $\tau_P \neq \tau_0$, we consider the equality
\begin{align}\label{eq: Im read tau0 and tau}
\rmread_{\tau_P}(T) = \rmread_{\tau_0}(T)
\rmread_{\tau_0}(\tau_P)^{-1}
\end{align}
which follows from \cref{def: read tau}.
By \cref{lem: characterization: regular}, we have $\tau_P \in \sfDS(\lambda/\mu)$, 
therefore combining \cref{lem: fixed T int} with \cref{eq: Im read tau0 and tau} yields that
\begin{align*}
\ell(\rmread_{\tau_P}(T)) = \ell(\rmread_{\tau_0}(T)) + \ell(\rmread_{\tau_0}(\tau_P)^{-1}).
\end{align*}
Now, we have that
\begin{align*}
\ell(\rmread_{\tau_P}(T)) - \ell(\rmread_{\tau_P}(T_{\lambda/\mu}))
& =
\ell(\rmread_{\tau_0}(T)) - \ell(\rmread_{\tau_0}(T_{\lambda/\mu}))  
\\
& =
\ell(\rmread_{\tau_0}(T) \rmread_{\tau_0}(T_{\lambda/\mu})^{-1}) 
& \text{by \cref{eq: read tau0 min and max}} \\
& =
\ell(\rmread_{\tau_P}(T) \rmread_{\tau_P}(T_{\lambda/\mu})^{-1})
& \text{by \cref{def: read tau}}.
\end{align*}
Therefore, $\rmread_{\tau_P}(T_{\lambda/\mu}) \preceq_L \rmread_{\tau_P}(T)$.
In the same manner, we can prove that
$\rmread_{\tau_P}(T) \preceq_L \rmread_{\tau_P}(T'_{\lambda/\mu})$.
\end{proof}

\section{An equivalence relation on $\rmInt(n)$} 
\label{sec: interval equal upto desc pres iso}

Recall that  $\rmInt(n)$ denotes the set of nonempty left weak Bruhat intervals in $\SG_n$, that is,
$$
\rmInt(n) = \{[\sigma,\rho]_L \mid 
\sigma,\rho \in \SG_n \text{ and } \sigma \preceq_L \rho\}.
$$
For $I_1, I_2 \in \rmInt(n)$, a poset isomorphism $f: (I_1, \preceq_L) \ra (I_2, \preceq_L)$ is called  \emph{descent-preserving} if $\Des_L(\gamma) = \Des_L(f(\gamma))$ for all $\gamma \in I_1$.
In this section, we study the classification of left weak Bruhat intervals in $\rmInt(n)$ up to descent-preserving poset isomorphism.
In particular, in the case where $P \in \sfRSP_n$, we explicitly describe the isomorphism class of $\Sigma_L(P)$.

We begin by 
explaining the reason why we consider descent-preserving poset isomorphisms.
Note that every interval $I \in \rmInt(n)$ can be represented by the colored digraph whose vertices are given by the permutations in $I$ and $\{1,2,\ldots, n-1\}$-colored arrows are given by
\begin{align*}
\gamma \overset{i}{\rightarrow} \gamma' 
\quad \text{if and only if} \quad
\text{$\gamma \preceq_L \gamma'$ and $s_i \gamma = \gamma'$}. 
\end{align*}
For intervals $I_1, I_2 \in \rmInt(n)$, 
a map $f: I_1 \ra I_2$ is called a \emph{colored digraph isomorphism} if $f$ is bijective and satisfies that for all $\gamma, \gamma' \in I_1$ and $1 \leq i \leq n-1$, 
\begin{align*}
\gamma \overset{i}{\rightarrow} \gamma' 
\quad \text{if and only if} \quad
f(\gamma) \overset{i}{\rightarrow} f(\gamma').
\end{align*}
If there exists a descent-preserving colored digraph isomorphism between two intervals $I_1$ and $I_2$, then $\sfB(I_1)$ is isomorphic to $\sfB(I_2)$.
Motivated by this fact, in \cite[Subsection 3.1]{22JKLO}, the authors posed the classification problem of weak Bruhat intervals up to descent-preserving colored digraph isomorphism.

A colored digraph isomorphism between  $I_1$ and $I_2$ is  a poset isomorphism with respect to $\preceq_L$, but a poset isomorphism $f: I_1 \ra I_2$  may not be a colored digraph isomorphism.
For instance, the poset isomorphism $f: [1234, 2134]_L \ra [1234, 1324]_L$ 
defined by $f(1234) = 1234$ and $f(2134) = 1324$ is not a colored digraph isomorphism since
\[ 
\begin{array}{l}
\begin{tikzpicture}
\node[] at (0,1) {1234};
\node[] at (0.2,0.5) {\tiny 1};
\draw[->] (0,.7) -- (0,.3);
\node[] at (0,0) {2134};
\end{tikzpicture}
\end{array}
\overset{f}{\longrightarrow} 
\begin{array}{l}
\begin{tikzpicture}
\node[] at (0,1) {1234};
\node[] at (0.2,0.5) {\tiny 2};
\draw[->] (0,.7) -- (0,.3);
\node[] at (0,0) {1324};
\end{tikzpicture}
\end{array}.
\]
However, if a poset isomorphism between left weak Bruhat intervals is descent-preserving,
it indeed proves to be a colored digraph isomorphism.

\begin{proposition}\label{Prop: desc pres isom and 0-Hecke}
Let $I_1, I_2 \in \rmInt(n)$. 
Every descent-preserving poset isomorphism $f: I_1 \ra I_2$ is a colored digraph isomorphism.
\end{proposition}

\begin{proof}
Let $\gamma, \gamma' \in I_1$ with $\gamma \preceq_L^c \gamma'$. 
Since $f$ is a poset isomorphism, we have
$f(\gamma) \preceq_L^c f(\gamma')$.
Let $i,j \in [n-1]$ satisfying $\gamma' = s_i \gamma$ and $f(\gamma') = s_j f(\gamma)$.
For the assertion, it suffices to show that $i = j$.
Let $D_1 := \Des_L(\gamma)$ and $D_2 := \Des_L(\gamma')$.
Since $\gamma' = s_i \gamma$, we have
\[
\{i\} \subseteq (D_1 \cup D_2) \setminus (D_1 \cap D_2) \subseteq \{i-1, i, i+1\}.
\]
In addition, since $D_1 = \Des_L(f(\gamma))$, $D_2 = \Des_L(f(\gamma'))$ and $f(\gamma') = s_j f(\gamma)$, we have $j \in D_2 \setminus D_1$, and it follows that $j$ is one of $i-1$, $i$, and $i+1$.

If $j = i-1$, then $i-1, i \in D_2$.
It follows that
\[
\gamma' = \cdots \  i+1 \ \cdots \ i \   \cdots \  i-1 \  \cdots 
\quad
\text{in one-line notation},
\]
equivalently,
\[
\gamma = \cdots \  i \ \cdots \  i+1 \   \cdots \  i-1  \ \cdots
\quad \text{in one-line notation.}
\]
This implies that $j=i-1 \in D_1$, which is a contradiction to $j \notin \Des_L(f(\gamma))$.
Therefore, $j\neq i-1$.

In a similar manner, one can show that $j \neq i+1$.
Hence, $j = i$, as required.
\end{proof}

\cref{Prop: desc pres isom and 0-Hecke} says that   classifying
weak Bruhat intervals up to descent-preserving colored digraph isomorphism
is equivalent to classifying 
weak Bruhat intervals up to descent-preserving poset isomorphism.
With this equivalence in mind, we introduce an equivalence relation, whose reflexivity, symmetricity, and transitivity are obvious.
\begin{definition}
We define an equivalence relation $\Deq$ on $\rmInt(n)$ by $I_1 \Deq I_2$ if there is a descent-preserving (poset) isomorphism between $(I_1, \preceq_L)$ and $(I_2, \preceq_L)$.
\end{definition}

For each equivalence class  $C$, we define 
\[
\xi_C := \rho \sigma^{-1} \quad \text{for any $[\sigma, \rho]_L \in C$.}
\]
By \cref{Prop: desc pres isom and 0-Hecke}, $\xi_C$ does not depend on the choice of $[\sigma, \rho]_L \in C$. 
We also define
\begin{align*}
\omin(C) := \{\sigma \mid [\sigma, \rho]_L \in C\} \quad \text{and} \quad \omax(C) := \{\rho \mid [\sigma, \rho]_L \in C \}.
\end{align*}
From now on, we always regard $\omin(C)$ and $\omax(C)$ as subposets of $(\SG_n, \preceq_R)$.

The following lemma is the initial step in the proof of the main result of this section (\cref{thm: desc pres equiv}).

\begin{lemma}\label{lem: min C is weak Bruhat interval card 2}
Let $C$ be an equivalence class under $\Deq$ with $\ell(\xi_C) = 1$.
Then, $\omin(C)$ is a right weak Bruhat interval in $(\SG_n, \preceq_R)$.
\end{lemma}
Before proving the lemma, we provide an outline of the proof for the reader's understanding.
We first classify the equivalence classes under consideration according to the set $X$ in \cref{eq: X(C)}.
Then, case by case, we show that $\omin(C)$ has a unique minimal element $\sigma_0$.
In particular, in {\bf Case 3}, we introduce a specific permutation $\bfw_0$ and show that it is the unique minimal element of $\omin(C)$.
Using these results, we next show that $\omin(C)$ has the unique maximal element $\sigma_1$.
Finally, we show that $[\sigma_0, \sigma_1]_R \subseteq \omin(C)$.
\begin{proof}
From the condition $\ell(\xi_C) = 1$ it follows that  $\xi_C = s_{i_0}$ 
for some $i_0 \in [n-1]$.

First, let us prove that  there exists a unique minimal element in $\omin(C)$.
Let
\begin{align}\label{eq: D_1 and D_2}
D_1 := \Des_L(\sigma)
\quad \text{and} \quad
D_2 := \Des_L(s_{i_0}\sigma)
\quad \text{for any $[\sigma, s_{i_0}\sigma]_L \in C$}
\end{align}
and 
\begin{align}\label{eq: X(C)}
X := (D_1 \cup D_2) \setminus (D_1 \cap D_2).
\end{align} 
One sees that $\{i_0\} \subseteq X \subseteq \{i_0 - 1, i_0, i_0 + 1\}$, and therefore $X$ can be one of the following:
\[
\{i_0-1, i_0, i_0+1\}, \ \
\{i_0\}, \ \  
\{i_0-1, i_0\}, \ \ \text{and} \ \  
\{i_0, i_0+1\}.
\]

{\bf Case 1: $X = \{i_0 - 1, i_0, i_0 + 1\}$.} 
Since  $i_0 - 1, i_0+1 \in D_1$ and $i_0 \notin D_1$,
\[
w_0(D_1) = \cdots \ i_0 \  i_0 - 1 \  \cdots \  i_0+2 \  i_0+1 \ \cdots 
\]
in one-line notation.
Considering this equality, one can see that
$[w_0(D_1), s_{i_0} w_0(D_1)]_L \in C$.
By \cref{lem: same desc int}, 
$w_0(D_1) \preceq_R \sigma$ for all $\sigma \in \SG_n$ with $\Des_L(\sigma) = D_1$.
Thus, $w_0(D_1)$ is a unique minimal element in $\omin(C)$.

{\bf Case 2: $X = \{i_0\}$.}
In this case, we have
\[
w_0(D_2) = \cdots \  i_0 + 1 \  i_0 \cdots \  \quad 
\]
in one-line notation.
Considering this equality, one can see that $[s_{i_0} w_0(D_2),  w_0(D_2)]_L \in C$.
Again, by \cref{lem: same desc int}, 
$w_0(D_2) \preceq_R \sigma$ for all $\sigma \in \SG_n$ with $\Des_L(\sigma) = D_2$. 
Thus, $s_{i_0} w_0(D_2)$ is a unique minimal element in $\omin(C)$.

{\bf Case 3: $X = \{i_0-1, i_0\}$.}
When $i_0 + 1 \notin D_1$,
following the way as in {\bf Case 1}, 
one can see that  $w_0(D_1)$ is a unique minimal element in $\omin(C)$.

From now on, assume that $i_0 + 1 \in D_1$.
We begin by introducing necessary notation.
Let
\[
\sfm_1 := \min\{m \in [n - 1] \mid [m, i_0 - 1] \subseteq D_1 \}
\quad \text{and} \quad
\sfm_2 := \max\{m  \in [n - 1] \mid [i_0 + 1, m] \subseteq D_1 \}.
\]
And, set
\[
p_1 := \sfm_1 - 1,
\quad
p_2 := \sfm_2 - i_0,
\quad
p_3 := i_0 - \sfm_1, 
\quad \text{and} \quad
p_4 := n - (\sfm_2 + 1).
\]
Let
\begin{enumerate}[label = $\bullet$]
\item 
$\bfw^{(1)}$ be the longest element of the subgroup $\SG_{D_1 \cap [p_1 - 1]}$ of $\SG_{p_1}$,
\item 
$\bfw^{(2)}$ be the longest element of $\SG_{p_2}$,
\item
$\bfw^{(3)}$ be the longest element of $\SG_{p_3}$, and
\item 
$\bfw^{(4)}$ be the longest element of the subgroup $\SG_{D_1 \cap [p_4 -1]}$ of $\SG_{p_4}$.
\end{enumerate}
With this notation, we define $\bfw_0$ to be the permutation given by
\[
\bfw_0(k) := \begin{cases}
\bfw^{(1)}(k) 
& \text{
if $k \in [p_1]$,
} \\
\bfw^{(2)}(k - p_1) + (i_0 + 1)
& \text{
if $k \in [p_1 +1, p_1 + p_2]$,
} \\
i_0
& \text{
if $k = p_1 + p_2 + 1$,
} \\
\bfw^{(3)}(k - (p_1 + p_2 + 1)) + (\sfm_1 - 1)
& \text{
if $k \in [p_1 + p_2 + 2, p_1 + p_2 + p_3 + 1]$,
} \\
i_0 + 1 
& \text{
if $k = p_1 + p_2 + p_3 + 2$,
} \\
\bfw^{(4)}(k - (p_1 + p_2 + p_3 + 2)) + (\sfm_2 + 1)
& \text{
if $k \in [p_1 + p_2 + p_3 + 3, n]$.
}
\end{cases}
\]
It should be remarked that 
\begin{align*}
\bfw_0([1,p_1]) &= [1,\sfm_1-1], 
\\
\bfw_0([p_1+1,p_1+p_2]) & = [i_0+2,\sfm_2+1], 
\\ 
\bfw_0(p_1+p_2+1) &= i_0,
\\ \bfw_0([p_1+p_2+2,p_1+p_2+p_3+1]) &= [\sfm_1, i_0-1], 
\\
\bfw_0(p_1+p_2+p_3+2) &= i_0+1,
\quad  \text{and} \\
\bfw_0([p_1+p_2+p_3+3,n]) &= [\sfm_2+2, n].
\end{align*}
From the definition of $\bfw_0$ it follows that $[\bfw_0, s_{i_0} \bfw_0]_L \in C$, equivalently, $\bfw_0 \in \omin(C)$.
We claim that $\bfw_0$ is a unique minimal element in $\omin(C)$.
This can be verified by showing that every minimal element in $\omin(C)$ is equal to $\bfw_0$.
Let $\sigma_0$ be a minimal element in $\omin(C)$. 
Set 
\begin{align*}
\begin{array}{l}
\calI_{\mathrm{L}} := \{\sigma_0(k) \mid 1 \le k < \sigma_0^{-1}(i_0)\}, \\[.5ex]
\calI_{\mathrm{C}} := \{\sigma_0(k) \mid \sigma_0^{-1}(i_0) < k < \sigma_0^{-1}(i_0+1) \}, 
\\[.5ex]
\calI_{\mathrm{R}} := \{\sigma_0(k) \mid \sigma_0^{-1}(i_0 + 1) < k \le n\}.
\end{array}
\end{align*}

To begin with, we establish the equalities
\begin{align}\label{eq: calI_L calI_C calI_R}
\calI_\mathrm{L} = [\sfm_1 -1] \cup [i_0+2, \sfm_2+1], 
\quad
\calI_\mathrm{C} = [\sfm_1, i_0 - 1],
\quad \text{and} \quad
\calI_\mathrm{R} = [\sfm_2 + 2,n].
\end{align}
These equalities are derived by verifying the following claims.
\smallskip

{\it \underline{Claim 1.} $\calI_\mathrm{C} = [\sfm_1, i_0 - 1]$.}
Let us show $\calI_\mathrm{C} \subseteq [\sfm_1, i_0 - 1]$.
First, to prove $\calI_\mathrm{C} \subseteq [i_0 - 1]$, we
assume that there exists $i \in \calI_\mathrm{C}$ such that $i \geq i_0$.
Let
\begin{align*}
k_1 & := \mathrm{max}\{k \in [n] \mid \sigma_0(k) \in \calI_\mathrm{C} \ \text{and} \  \sigma_0(k) \ge i_0 \}.
\end{align*}
By the definition of $\calI_\mathrm{C}$, $i_0,i_0+1 \notin \calI_\mathrm{C}$.
If $i_0 + 2 \in \calI_\mathrm{C}$, then $i_0 + 1 \in D_1 \setminus D_2$, which contradicts the assumption that $i_0 +1 \notin X$.
Since $\sigma_0(k_1) \in \calI_\mathrm{C}$,
it follows that $\sigma_0(k_1) \geq i_0 + 3$.
And, by the choice of $k_1$, we have $\sigma_0(k_1 + 1) \leq i_0+1$.
Putting these together yields that
$[\sigma_0 s_{k_1}, s_{i_0} \sigma_0s_{k_1}]_L \Deq [\sigma_0, s_{i_0}\sigma_0]_L$ and $\sigma_0 s_{k_1} \prec_R \sigma_0$.
This contradicts the minimality of $\sigma_0$ in $\omin(C)$, therefore $\calI_\mathrm{C} \subseteq [i_0 - 1]$.
Next, to prove $\calI_\mathrm{C} \subseteq [\sfm_1, n]$, we assume that there exists $i \in \calI_\mathrm{C}$ such that $i < \sfm_1$.
Let
\[
k_2 := \mathrm{min}\{ 
k \in [n] \mid \sigma_0(k) \in \calI_\mathrm{C} \ \text{and} \ 
\sigma_0(k) < \sfm_1\}.
\]
By the choice of $k_2$, we have $\sigma_0(k_2) + 1 \le \sfm_1 \le \sigma_0(k_2 - 1)$.
In addition, if $\sigma_0(k_2) + 1 = \sfm_1 = \sigma_0(k_2 - 1)$, then $\sfm_1 - 1 \in \Des_L(\sigma_0)$, which cannot happen by the definition of $\sfm_1$.
Therefore, $\sigma_0(k_2) + 1 < \sigma_0(k_2 - 1)$, which implies that $[\sigma_0 s_{k_2 - 1}, s_{i_0} \sigma_0 s_{k_2 - 1}]_L \Deq [\sigma_0, s_{i_0}\sigma_0]_L$ and
$\sigma_0 s_{k_2 - 1} \prec_R \sigma_0$.
This contradicts the minimality of $\sigma_0$ in $\omin(C)$.
Thus, $\calI_\mathrm{C} \subseteq [\sfm_1, i_0 - 1]$.

Let us show $\calI_\mathrm{C} \supseteq [\sfm_1, i_0 - 1]$.
Assume for the sake of contradiction that there exists $i \in [\sfm_1, i_0 - 1]$ such that $i \notin \calI_\mathrm{C}$.
Let $j$ be the maximal element in $[\sfm_1, i_0 - 1]$ such that $j \notin \calI_\mathrm{C}$.
Since $i_0 - 1 \in X$, we have $i_0 - 1 \in \calI_\mathrm{C}$.
It follows that $j < i_0 - 1$, so $j+1 \in [\sfm_1,i_0-1]$.
Combining this with the maximality of $j$, we have $j+1 \in \calI_\mathrm{C}$. 
And, by the definition of $\sfm_1$, we have $j \in D_1$.
Putting these together yields that $j \in \calI_\mathrm{R}$.
Let
\[
k_3 := \mathrm{min}\{
k \in [n] 
\mid  
\sigma_0(k) \in \calI_\mathrm{R} \  \text{and} \ 
\sigma_0(k)  \leq j \}.
\]
If $\sigma_0(k_3 -1) \le \sigma_0(k_3) + 1$, then $\sigma_0(k_3 -1) \le j + 1 < i_0$.
So, $\sigma_0(k_3 - 1) \neq i_0 + 1$ which implies $\sigma_0(k_3 -1) \in \calI_\mathrm{R}$.
This, together with the minimality of $k_3$, yields that $j+1 \le \sigma_0(k_3 -1)$.
It follows that $\sigma_0(k_3 - 1)  = j + 1$, 
which is a contradiction because $\sigma_0(k_3 -1) \in \calI_\mathrm{R}$, but $j+1 \in \calI_\mathrm{C}$.
Therefore, we have
$\sigma_0(k_3) + 1 < \sigma_0(k_3-1)$.
In addition, since $j < i_0 - 1$, we have $\sigma_0(k_3) < i_0 - 1$.
Putting these together yields that $[\sigma_0 s_{k_3 - 1}, s_{i_0} \sigma_0 s_{k_3 - 1}]_L \Deq [\sigma_0, s_{i_0}\sigma_0]_L$ and $\sigma_0 s_{k_3 - 1} \prec_R \sigma_0$, which contradicts the minimality of $\sigma_0$ in $\omin(C)$.
Thus, $\calI_\mathrm{C} \supseteq [\sfm_1, i_0 - 1]$.
\smallskip

{\it \underline{Claim 2.} $[\sfm_1 -1] \cup [i_0+2, \sfm_2+1] \subseteq \calI_\mathrm{L}$.}
By the definition of $\sfm_2$,
we have $[i_0 + 1, \sfm_2] \subseteq \Des_L(\sigma_0)$.
Since $i_0+2 \notin \calI_\mathrm{C}$, we have $[i_0+2, \sfm_2+1] \subseteq \calI_\mathrm{L}$.
To prove $[\sfm_1 - 1] \subseteq \calI_\mathrm{L}$, suppose that there exists $i \in [\sfm_1 - 1]$ such that $i \notin \calI_\mathrm{L}$.
Let
\[
k_4 := \mathrm{min}\{k \in [n] \mid \sigma_0(k) \in [\sfm_1 - 1] 
 \ \text{and} \ 
\sigma_0(k) \notin \calI_\mathrm{L}\}.
\]
Since $\sigma_0(k_4) \notin \calI_\mathrm{L}$ and $\sigma_0(k_4) < \sfm_1$, we have $\sigma_0(k_4) \in \calI_\mathrm{R}$. 
This implies that $\sigma_0(k_4-1) \notin \calI_\mathrm{L}  \cup  \{i_0\}  \cup  \calI_\mathrm{C}$.
In addition, the minimality of $k_4$ gives $\sigma_0(k_4 - 1) \geq \sfm_1$.
Since $[\sfm_1, i_0-1] = \calI_\mathrm{C}$, we have $\sigma_0(k_4 - 1) \geq i_0+1$.
Putting the above inequalities together,
we have
\[
\sigma_0(k_4) < \sfm_1 \le i_0-1 < i_0 + 1 \le \sigma_0(k_4-1),
\]
and so $\sigma_0(k_4) + 2 < \sigma_0(k_4 - 1)$.
It follows that $[\sigma_0 s_{k_4 - 1}, s_{i_0} \sigma_0 s_{k_4 - 1}]_L \Deq [\sigma_0, s_{i_0}\sigma_0]_L$ and $\sigma_0 s_{k_4 - 1} \prec_R \sigma_0$.
This contradicts the minimality of $\sigma_0$ in $\omin(C)$, thus $[\sfm_1 - 1] \subseteq \calI_\mathrm{L}$.
\smallskip

{\it \underline{Claim 3.} $[\sfm_2 + 2,n] \subseteq \calI_\mathrm{R}$.}
Suppose that there exists $i \in [\sfm_2 + 2,n]$ such that $i \notin \calI_\mathrm{R}$.
Let
\[
k_5 := \mathrm{max}\{k \in [n] \mid \sigma_0(k) \notin \calI_\mathrm{R} \ \  \text{and} \ \  \sigma_0(k)  \in [\sfm_2 + 2,n]\}.
\]
Since $k_5 \notin \calI_\mathrm{R}$ and $\sigma_0(k_5) >
i_0 + 1$, we have $\sigma_0(k_5) \in \calI_\mathrm{L}$, which implies that $\sigma_0(k_5 + 1) \notin \calI_\mathrm{R}$.
By the maximality of $k_5$, we have $\sigma_0(k_5+1) \le \sfm_2 + 1$.
If $\sigma_0(k_5+1) < \sfm_2 + 1$, then 
\[
\sigma_0(k_5+1) +1 < \sfm_2 + 2 \le \sigma_0(k_5).
\]
If $\sigma_0(k_5+1) = \sfm_2 + 1$, then $\sigma_0^{-1}(\sfm_2 + 2) > k_5+1$ due to the maximality of $\sfm_2$,  so $\sigma_0(k_5) \neq \sfm_2 + 2$ which implies
\[
\sigma_0(k_5+1) + 1 < \sigma_0(k_5).
\]
Putting these together with the inequalities $\sigma_0(k_5) \ge \sfm_2 + 2 >i_0 + 2$ yields that $[\sigma_0 s_{k_5}, s_{i_0} \sigma_0 s_{k_5}]_L \Deq [\sigma_0, s_{i_0}\sigma_0]_L$ and $\sigma_0 s_{k_5} \prec_R \sigma_0$.
This contradicts the minimality of $\sigma_0$ in $\omin(C)$, thus $[\sfm_2 + 2, n] \subseteq \calI_\mathrm{R}$.
\smallskip

Now, we are ready to show that $\sigma_0 = \bfw_0$.
Let 
\[
\calI^{(1)}_{\mathrm{L}} := \{ \sigma_0(k) \in \calI_{\mathrm{L}} \mid 1 \le k \le \sfm_1-1 \}
\quad \text{and} \quad
\calI^{(2)}_{\mathrm{L}} := \{ \sigma_0(k) \in \calI_{\mathrm{L}} \mid \sfm_1 \le k <\sigma_0^{-1}(i_0)  \}.
\]
We claim that $\calI^{(1)}_{\mathrm{L}} = [\sfm_1-1]$ and $\calI^{(2)}_{\mathrm{L}} = [i_0+2, \sfm_2+1]$.
We may assume that $\sfm_1 > 1$, otherwise, the claim is obvious.
To prove our claim, suppose that there exists $i \in \calI^{(1)}_{\mathrm{L}}$ such that $i \in [i_0+2,\sfm_2+1]$. 
Then, there exists $1 \le k < \sigma_0^{-1}(i_0) - 1$ such that $\sigma_0(k) \in [i_0+2, \sfm_2 + 1]$ and $\sigma_0(k + 1) \in [\sfm_1 - 1]$.
It follows that $[\sigma_0 s_{k}, s_{i_0} \sigma_0 s_{k}]_L \Deq [\sigma_0, s_{i_0}\sigma_0]_L$ and
$\sigma_0 s_{k} \prec_R \sigma_0$.
Again, this contradicts the minimality of $\sigma_0$ in $\omin(C)$, so 
\begin{align}\label{eq: calIL{(1)} calIL{(2)}}
\calI^{(1)}_{\mathrm{L}} = [\sfm_1-1]
\quad \text{and} \quad \calI^{(2)}_{\mathrm{L}} = [i_0+2, \sfm+1].
\end{align}
Putting
\cref{lem: same desc int},
\cref{eq: calI_L calI_C calI_R},
\cref{eq: calIL{(1)} calIL{(2)}}, and the minimality of $\sigma_0$ together, we conclude that $\sigma_0 = \bfw_0$.

{\bf Case 4: $X = \{i_0, i_0 + 1\}$.}
Take $[\sigma, s_{i_0} \sigma]_L \in C$ and let $C'$ be the equivalence class of $[\sigma^{w_0}, (s_{i_0} \sigma)^{w_0}]_L$.
By mimicking \cref{eq: D_1 and D_2} and \cref{eq: X(C)},  
we define 
\[
D'_1 := \Des_L(\sigma^{w_0}), 
\quad 
D'_2 := \Des_L((s_{i_0}\sigma)^{w_0}), 
\quad \text{and} \quad
X' := (D'_1 \cup D'_2) \setminus (D'_1 \cap D'_2).
\]
Since 
$D'_1 = \{
n-i \mid i \in D_1\}$ and $D'_2 = \{
n-i \mid i \in D_2\}$,
we have
\[
X' =  \{n-i_0, (n-i_0)+1\}.
\]
Following the proof of {\bf Case 3}, we see that  $\omin(C')$ has a unique minimal element $\bfw'_0$.
And, one can easily see that the map $f: \omin(C) \ra \omin(C')$, $\gamma \mapsto \gamma^{w_0}$ is a well-defined bijection and that for $\gamma_1, \gamma_2 \in \omin(C)$, $\gamma_1 \preceq_R \gamma_2$ if and only if $f(\gamma_1) \preceq_R f(\gamma_2)$.
Thus, $(\bfw'_0)^{w_0}$ is a unique minimal element in $\omin(C')$.
\medskip

Second, we will show that $\omin(C)$ has a unique maximal element.
Recall that we take $[\sigma, s_{i_0} \sigma]_L \in C$.
Let $C''$ be the equivalence class of $[s_{i_0} \sigma w_0, \sigma w_0]_L$.
Due to the previous arguments, we know that there is a unique minimal element $\gamma_0$ in $\omin(C'')$.
One can easily see that the map $g: \omin(C) \ra \omin(C'')$, $\gamma \mapsto \gamma w_0$ is a well-defined bijection and that for $\gamma_1, \gamma_2 \in \omin(C)$, $\gamma_1 \preceq_R \gamma_2$ if and only if $g(\gamma_1) \succeq_R g(\gamma_2)$.
Therefore, $\gamma_0 w_0$ is the unique maximal element in $\omin(C)$.

\medskip
Finally, we will show that  $\omin(C)$ is a right weak Bruhat interval in $(\SG_n, \preceq_R)$.
Let $\sigma_0$ and $\sigma_1$ be the minimal and maximal elements in $\omin(C)$, respectively.
Let $\gamma \in [\sigma_0, \sigma_1]_R$.
Since $\Des_L(\sigma_0) = \Des_L(\sigma_1)$,  we have
$\Des_L(\gamma) = \Des_L(\sigma_0)$ by \cref{lem: same desc int}.
Next, let us examine $\Des_L(s_{i_0}\gamma)$.
Since $\Des_L(\sigma_0) = \Des_L(\gamma)$,
it follows that $\gamma \preceq_L s_{i_0} \gamma$.
By \cref{lem: interval translation}, we have that $s_{i_0} \gamma \in [s_{i_0}\sigma_0, s_{i_0}\sigma_1]_{R}$.
Since $\Des_L(s_{i_0} \sigma_0) =\Des_L(s_{i_0} \sigma_1)$, we have $\Des_L(s_{i_0}\gamma) = \Des_L(s_{i_0}\sigma_0)$ by \cref{lem: same desc int}.
Thus, $\gamma \in \omin(C)$.
\end{proof}

\begin{example}\label{eg: C and min(C) and max(C)}
Let $C \subseteq \rmInt(4)$ be the equivalence class of $[2134, 2143]_L$.
One sees that 
\[
C = \{[2134, 2143]_L, [2314, 2413]_L, [2341, 2431]_L\}.
\]
So, $\omin(C) = \{2134,2314,2341\}$ and $\omax(C) = \{2143,2413,2431\}$ which are equal to $[2134, 2341]_R$ and $[2143,2431]_R$, respectively.
For the readers' convenience, we draw the left weak Bruhat intervals in $C$ within the left weak Bruhat graph of $\SG_4$ on the left hand side of \cref{fig: C and min(C) and max(C)}.
We also draw the right weak Bruhat intervals $\omin(C)$ and $\omax(C)$ within the right weak Bruhat graph of $\SG_4$ on the right hand side of \cref{fig: C and min(C) and max(C)}.
\end{example}

\begin{figure}[t]
\centering
\[
\def \hp{0.25}
\def \wp{0.2}
\def \wtab{2.5}
\def \htab{1.5}
\scalebox{0.7}{
\begin{tikzpicture}[baseline = 0mm, scale = 0.8]
\node at (0,0) {\color{lightgray}$4321$};
\node at (-\wtab,1*\htab) {\color{lightgray}$4231$};
\node at (0,1*\htab) {\color{lightgray}$4312$};
\node at (\wtab,1*\htab) {\color{lightgray}$3421$};
\node at (-\wtab*2,2*\htab) {\color{lightgray}$4132$};
\node at (-\wtab,2*\htab) {\color{lightgray}$4213$};
\node at (0,2*\htab) {\color{lightgray}$3241$};
\node at (\wtab,2*\htab) {\color{lightgray}$3412$};
\node at (\wtab*2,2*\htab) {\color{black}$2431$};

\node at (-\wtab/2*5,3*\htab) {\color{lightgray}$4123$};
\node at (-\wtab/2*3,3*\htab) {\color{lightgray}$3214$};
\node at (-\wtab/2,3*\htab) {\color{lightgray}$3142$};
\node at (\wtab/2,3*\htab) {\color{black}$2413$};
\node at (\wtab/2*3,3*\htab) {\color{lightgray}$1432$};
\node at (\wtab/2*5,3*\htab) {\color{black}$2341$};
\node at (-\wtab*2,4*\htab) {\color{lightgray}$3124$};
\node at (-\wtab,4*\htab) {\color{black}$2314$};
\node at (0,4*\htab) {\color{black}$2143$};

\node at (\wtab,4*\htab) {\color{lightgray}$1423$};
\node at (\wtab*2,4*\htab) {\color{lightgray}$1342$};
\node at (-\wtab,5*\htab) {\color{black}$2134$};
\node at (0,5*\htab) {\color{lightgray}$1324$};
\node at (\wtab,5*\htab) {\color{lightgray}$1243$};
\node at (0,6*\htab) {\color{lightgray}$1234$};

\draw [lightgray] (-\wp,6*\htab-\hp) -- (-\wtab+\wp,5*\htab+\hp);
\draw [lightgray] (0,6*\htab-\hp) -- (0,5*\htab+\hp);
\draw [lightgray] (\wp, 6*\htab -\hp) -- (\wtab-\wp,5*\htab+\hp);

\draw [lightgray] (-\wtab-\wp,5*\htab - \hp) -- (-\wtab*2+\wp,4*\htab+\hp);
\draw [black] (\wp-\wtab,5*\htab - \hp) -- (0-\wp,4*\htab+\hp);
\node at (-0.44*\wtab, 4.45*\htab+\hp) {\color{black}\footnotesize $s_3\cdot$};
\draw [lightgray] (-\wp,5*\htab - \hp) -- (-\wtab+\wp,4*\htab+\hp);
\draw [lightgray] (\wp,5*\htab - \hp) -- (\wtab-\wp,4*\htab+\hp);
\draw [lightgray] (-\wp+\wtab,5*\htab - \hp) -- (0+\wp,4*\htab+\hp);
\draw [lightgray] (\wp+\wtab,5*\htab - \hp) -- (2*\wtab-\wp,4*\htab+\hp);

\draw [lightgray] (-\wtab*2-\wp,4*\htab-\hp) -- (-\wtab/2*5+\wp,3*\htab+\hp);
\draw [lightgray] (-\wtab*2+\wp,4*\htab-\hp) -- (-\wtab/2*3-\wp,3*\htab+\hp);
\draw [lightgray] (-\wtab-\wp,4*\htab-\hp) -- (-\wtab/2*3+\wp,3*\htab+\hp);
\draw [black] (-\wtab+\wp,4*\htab-\hp) -- (\wtab/2-\wp,3*\htab+\hp);
\node at (-\wtab * 0.5+\wp,3.6*\htab-\hp) {\footnotesize $s_3\cdot$};
\draw [lightgray] (-\wp,4*\htab-\hp) -- (-\wtab/2+\wp,3*\htab+\hp);
\draw [lightgray] (\wtab-\wp,4*\htab-\hp) -- (\wtab/2+\wp,3*\htab+\hp);
\draw [lightgray] (\wtab+\wp,4*\htab-\hp) -- (\wtab/2*3-\wp,3*\htab+\hp);
\draw [lightgray] (2*\wtab-\wp,4*\htab-\hp) -- (\wtab/2*3+\wp,3*\htab+\hp);
\draw [lightgray] (2*\wtab+\wp,4*\htab-\hp) -- (\wtab/2*5-\wp,3*\htab+\hp);

\draw [lightgray] (-\wtab/2*5+\wp,3*\htab-\hp) -- (-\wtab*2-\wp,2*\htab+\hp);
\draw [lightgray] (-\wtab*2.3+\wp,3*\htab-\hp) -- (-1.2*\wtab-\wp,2*\htab+\hp);
\draw [lightgray] (-\wtab/2-\wp,3*\htab-\hp) -- (-\wtab*2+\wp,2*\htab+\hp);
\draw [lightgray] (-\wtab/2+\wp,3*\htab-\hp) -- (-\wp,2*\htab+\hp);
\draw [lightgray] (\wtab/2+\wp,3*\htab-\hp) -- (\wtab-\wp,2*\htab+\hp);
\draw [lightgray] (\wtab/2*3+\wp,3*\htab-\hp) -- (\wtab*2-\wp,2*\htab+\hp);
\draw [lightgray] (\wtab/2*5-\wp,3*\htab-\hp) -- (\wp,2*\htab+\hp);
\draw [black] (\wtab/2*5,3*\htab-\hp) -- (\wtab*2+\wp,2*\htab+\hp);
\node at (\wtab*2.35+\wp,2.7*\htab-1.5*\hp) {\footnotesize $s_3\cdot$};
\draw [lightgray] (-\wtab/2*3+\wp,3*\htab-\hp) -- (-\wtab-\wp,2*\htab+\hp);

\draw [lightgray] (-\wtab*2+\wp,2*\htab-\hp) -- (-\wtab-\wp, \htab+\hp);
\draw [lightgray] (-\wp,2*\htab-\hp) -- (-\wtab+\wp, \htab+\hp);
\draw [lightgray] (\wtab-\wp,2*\htab-\hp) -- (\wp, \htab+\hp);
\draw [lightgray] (\wtab,2*\htab-\hp) -- (\wtab, \htab+\hp);
\draw [lightgray] (2*\wtab-\wp,2*\htab-\hp) -- (\wtab+\wp, \htab+\hp);
\draw [lightgray] (-\wtab+\wp,2*\htab-\hp) -- (-\wp, \htab+\hp);

\draw [lightgray] (-\wtab+\wp,\htab-\hp) -- (-\wp, \hp);
\draw [lightgray] (0,\htab-\hp) -- (0, \hp);
\draw [lightgray] (\wtab-\wp,\htab-\hp) -- (\wp, \hp);
\end{tikzpicture}
}
\quad 
\scalebox{0.7}{
\begin{tikzpicture}[baseline = 0mm, scale = 0.8]
\node at (0,0) {\color{lightgray}$4321$};
\node at (-\wtab,1*\htab) {\color{lightgray}$4231$};
\node at (0,1*\htab) {\color{lightgray}$3421$};
\node at (\wtab,1*\htab) {\color{lightgray}$4312$};
\node at (-\wtab*2,2*\htab) {\color{black}$2431$};
\node at (-\wtab,2*\htab) {\color{lightgray}$3241$};
\node at (0,2*\htab) {\color{lightgray}$4213$};
\node at (\wtab,2*\htab) {\color{lightgray}$3412$};
\node at (\wtab*2,2*\htab) {\color{lightgray}$4132$};

\node at (-\wtab/2*5,3*\htab) {\color{black}$2341$};
\node at (-\wtab/2*3,3*\htab) {\color{lightgray}$3214$};
\node at (-\wtab/2,3*\htab) {\color{black}$2413$};
\node at (\wtab/2,3*\htab) {\color{lightgray}$3142$};
\node at (\wtab/2*3,3*\htab) {\color{lightgray}$1432$};
\node at (\wtab/2*5,3*\htab) {\color{lightgray}$4123$};
\node at (-\wtab*2,4*\htab) {\color{black}$2314$};
\node at (-\wtab,4*\htab) {\color{lightgray}$3124$};
\node at (0,4*\htab) {\color{black}$2143$};
\node at (\wtab,4*\htab) {\color{lightgray}$1342$};
\node at (\wtab*2,4*\htab) {\color{lightgray}$1423$};
\node at (-\wtab,5*\htab) {\color{black}$2134$};
\node at (0,5*\htab) {\color{lightgray}$1324$};
\node at (\wtab,5*\htab) {\color{lightgray}$1243$};
\node at (0,6*\htab) {\color{lightgray}$1234$};

\draw [lightgray] (-\wp,6*\htab-\hp) -- (-\wtab+\wp,5*\htab+\hp);
\draw [lightgray] (0,6*\htab-\hp) -- (0,5*\htab+\hp);
\draw [lightgray] (\wp, 6*\htab -\hp) -- (\wtab-\wp,5*\htab+\hp);

\draw [black] (-\wtab-\wp,5*\htab - \hp) -- (-\wtab*2+\wp,4*\htab+\hp);
\node at (-1.65*\wtab, 4.5*\htab+\hp) {\color{black}\footnotesize $\cdot s_2$};
\draw [lightgray] (\wp-\wtab,5*\htab - \hp) -- (0-\wp,4*\htab+\hp);
\draw [lightgray] (-\wp,5*\htab - \hp) -- (-\wtab+\wp,4*\htab+\hp);
\draw [lightgray] (\wp,5*\htab - \hp) -- (\wtab-\wp,4*\htab+\hp);
\draw [lightgray] (-\wp+\wtab,5*\htab - \hp) -- (0+\wp,4*\htab+\hp);
\draw [lightgray] (\wp+\wtab,5*\htab - \hp) -- (2*\wtab-\wp,4*\htab+\hp);

\draw [black] (-\wtab*2-\wp,4*\htab-\hp) -- (-\wtab/2*5+\wp,3*\htab+\hp);
\draw [lightgray] (-\wtab*2+\wp,4*\htab-\hp) -- (-\wtab/2*3-\wp,3*\htab+\hp);
\node at (-2.4*\wtab, 3.75*\htab-\hp) {\color{black}\footnotesize $\cdot s_3$};
\draw [lightgray] (-\wtab-\wp,4*\htab-\hp) -- (-\wtab/2*3+\wp,3*\htab+\hp);
\draw [lightgray] (-\wtab+\wp,4*\htab-\hp) -- (\wtab/2-\wp,3*\htab+\hp);
\draw [black] (-\wp,4*\htab-\hp) -- (-\wtab/2+\wp,3*\htab+\hp);
\node at (-0.4*\wtab, 3.75*\htab-\hp) {\color{black}\footnotesize $\cdot s_2$};
\draw [lightgray] (\wtab-\wp,4*\htab-\hp) -- (\wtab/2+\wp,3*\htab+\hp);
\draw [lightgray] (\wtab+\wp,4*\htab-\hp) -- (\wtab/2*3-\wp,3*\htab+\hp);
\draw [lightgray] (2*\wtab-\wp,4*\htab-\hp) -- (\wtab/2*3+\wp,3*\htab+\hp);
\draw [lightgray] (2*\wtab+\wp,4*\htab-\hp) -- (\wtab/2*5-\wp,3*\htab+\hp);

\draw [lightgray] (-\wtab/2*5+\wp,3*\htab-\hp) -- (-\wtab*2-\wp,2*\htab+\hp);
\draw [lightgray] (-\wtab*2.3+\wp,3*\htab-\hp) -- (-1.2*\wtab-\wp,2*\htab+\hp);
\draw [black] (-\wtab/2-\wp,3*\htab-\hp) -- (-\wtab*2+\wp,2*\htab+\hp);
\node at (-1.4*\wtab, 2.8*\htab-\hp) {\color{black}\footnotesize $\cdot s_3$};
\draw [lightgray] (-\wtab/2+\wp,3*\htab-\hp) -- (-\wp,2*\htab+\hp);
\draw [lightgray] (\wtab/2+\wp,3*\htab-\hp) -- (\wtab-\wp,2*\htab+\hp);
\draw [lightgray] (\wtab/2*3+\wp,3*\htab-\hp) -- (\wtab*2-\wp,2*\htab+\hp);
\draw [lightgray] (\wtab/2*5-\wp,3*\htab-\hp) -- (\wp,2*\htab+\hp);
\draw [lightgray] (\wtab/2*5,3*\htab-\hp) -- (\wtab*2+\wp,2*\htab+\hp);
\draw [lightgray] (-\wtab/2*3+\wp,3*\htab-\hp) -- (-\wtab-\wp,2*\htab+\hp);

\draw [lightgray] (-\wtab*2+\wp,2*\htab-\hp) -- (-\wtab-\wp, \htab+\hp);
\draw [lightgray] (-\wp,2*\htab-\hp) -- (-\wtab+\wp, \htab+\hp);
\draw [lightgray] (\wtab-\wp,2*\htab-\hp) -- (\wp, \htab+\hp);
\draw [lightgray] (\wtab,2*\htab-\hp) -- (\wtab, \htab+\hp);
\draw [lightgray] (2*\wtab-\wp,2*\htab-\hp) -- (\wtab+\wp, \htab+\hp);
\draw [lightgray] (-\wtab+\wp,2*\htab-\hp) -- (-\wp, \htab+\hp);

\draw [lightgray] (-\wtab+\wp,\htab-\hp) -- (-\wp, \hp);
\draw [lightgray] (0,\htab-\hp) -- (0, \hp);
\draw [lightgray] (\wtab-\wp,\htab-\hp) -- (\wp, \hp);

\draw [dotted, red, line width=0.4mm] 
(-1.4*\wtab,5.2*\htab)
-- (-0.7*\wtab,5.2*\htab)
-- (-0.7*\wtab,4.7*\htab)
-- (-1.6*\wtab,4*\htab)
-- (-2.2*\wtab,2.8*\htab)
-- (-2.8*\wtab,2.8*\htab)
-- (-2.8*\wtab,3.3*\htab)
-- (-2.4*\wtab,4.2*\htab)
-- (-1.4*\wtab,5.2*\htab);
\node at (-2.5*\wtab, 4.7*\htab) {$\omin(C)$};

\draw [dotted, blue, line width=0.4mm]
(-0.35*\wtab,4.2*\htab)
-- (0.3*\wtab,4.2*\htab)
-- (0.3*\wtab,3.8*\htab)
-- (-0.2*\wtab,2.7*\htab)
-- (-1.85*\wtab,1.75*\htab)
-- (-2.3*\wtab,1.75*\htab)
-- (-2.3*\wtab,2.35*\htab)
-- (-0.8*\wtab,3.2*\htab)
-- (-0.35*\wtab,4.2*\htab);
\node at (0*\wtab, 2.4*\htab) {$\omax(C)$};
\end{tikzpicture}
}
\]
\caption{The left weak Bruhat intervals in $C$ on $(\SG_4, \preceq_L)$ and the right weak Bruhat intervals $\omin(C)$ and $\omax(C)$ on $(\SG_4, \preceq_R)$ in \cref{eg: C and min(C) and max(C)}}
\label{fig: C and min(C) and max(C)}
\end{figure}

\begin{lemma}\label{lem: intersection of int is int}
The intersection of two right weak Bruhat intervals in $\SG_n$ is again a right weak Bruhat interval.
\end{lemma}

\begin{proof}
It is well known that $(\SG_n, \preceq_R)$ is a lattice,  that is, every two-element subset $\{\gamma_1,\gamma_2\} \subseteq \SG_n$ has the least upper bound and greatest lower bound (for example see \cite[Section 3.2]{05BB}).
Combining this with the fact $|\SG_n| < \infty$, we derive the desired result.
\end{proof}

The following theorem provides significant information regarding equivalence classes under $\Deq$.

\begin{theorem}\label{thm: desc pres equiv}
Let $C$ be an equivalence class 
under $\Deq$.
Then $\omin(C)$ and $\omax(C)$ are right weak Bruhat intervals in $(\SG_n, \preceq_R)$.
\end{theorem}

\begin{proof}
Note that $\sigma \preceq_L \xi_C \sigma$ for any $\sigma \in \omin(C)$ and that $\omax(C) = \xi_C \cdot \omin(C)$.
If we prove that $\omin(C)$ is a right weak Bruhat interval, then \cref{lem: interval translation} implies that $\omax(C)$ is also a right weak Bruhat interval.
So we will only prove that $\omin(C)$ is a right weak Bruhat interval.

When $\ell(\xi_C) = 0$, the assertion follows from \cref{lem: same desc int}.
From now on, assume that $\ell(\xi_C) \geq 1$.
We will prove the assertion by using mathematical induction on $\ell(\xi_C)$.
When $\ell(\xi_C) = 1$, the assertion is true by \cref{lem: min C is weak Bruhat interval card 2}.
Let $k$ be an arbitrary positive integer and suppose that the assertion holds for every equivalence class $C \in \scrC(n)$ with $\ell(\xi_C) \leq k$.
Let $C \in \scrC(n)$ with $\ell(\xi_C) = k+1$.
Set
$$
\calA := \{i \in [n-1] \mid s_i \in [\id, \xi_C]_L \}.
$$
Given $i \in \calA$ and $\sigma \in \omin(C)$, note that 
\begin{align*}
[\sigma, s_i \sigma]_L \Deq [\sigma', s_i \sigma']_L
\quad \text{and} \quad
[s_i \sigma, \xi_C \sigma]_L \Deq [s_i\sigma', \xi_C \sigma']_L
\end{align*}
for all $\sigma' \in \omin(C)$.
This says that the equivalence classes of $[\sigma, s_i \sigma]_L$ and $[s_i\sigma, \xi_C \sigma]_L$
do not depend on  $\sigma \in \omin(C)$.
For each $i \in \calA$,  we set 
\begin{align*}
E_i & := \text{the equivalence class of $[\sigma, s_i \sigma]_L$,}
\\
E'_i & := \text{the equivalence class of $[s_i \sigma, \xi_C \sigma]_L$}
\end{align*}
for any  $\sigma \in \omin(C)$.
Then, we set 
\begin{align*}
J_i := \omax(E_i) \cap \omin(E'_i) \quad \text{for $i \in \calA$,} \quad \text{and} \quad 
J := \bigcap_{i \in \calA}s_i \cdot J_i.
\end{align*}
Now, the desired assertion can be achieved by proving the following claims:
\begin{enumerate}[label = {(\roman*)}]
\item $\omin(C) = J$.
\item $J$ is a right weak Bruhat interval.
\end{enumerate}

First, let us prove $\omin(C) = J$.
By the definition of $J_i$, we have $s_i \sigma \in J_i$ for all $i \in \calA$ and $\sigma \in \omin(C)$. 
It follows that $\omin(C) \subseteq J$.
To prove the opposite inclusion $\omin(C) \supseteq J$, take $\sigma \in J$.
By the definition of $J$, we have
\begin{equation*}
[\sigma, s_i \sigma]_L  \in E_i 
\quad \text{and} \quad 
[s_i \sigma ,\xi_C \sigma]_L \in E'_i
\quad \text{for all $i \in \calA$}. 
\end{equation*}
And, \cref{lem: interval translation} implies that 
$$
\{s_i \sigma \mid i \in \calA \} = \{\gamma \in [\sigma, \xi_C\sigma]_L \mid \sigma \prec^\rmc_L \gamma \}.
$$
Putting these together yields that $[\sigma, \xi_C \sigma]_L \in C$, therefore $\sigma \in \omin(C)$.

Next, let us prove that $J$ is a right weak Bruhat interval.
Due to \cref{lem: intersection of int is int}, it suffices to show that $s_i \cdot J_i$ is a right weak Bruhat interval for $i \in \calA$.
Let us fix $i \in \calA$.
Since $\ell(\xi_{E_i}) = 1$ and $\ell(\xi_{E'_i}) = k$, $\omax(E_i)$ and $\omin(E'_i)$ are right weak Bruhat intervals by the induction hypothesis.
Combining this with \cref{lem: intersection of int is int} yields that $J_i$ is a right weak Bruhat interval.
In addition, we have $s_i \gamma \preceq_L \gamma$ for all $\gamma \in J_i$.
Therefore, by \cref{lem: interval translation}, $s_i \cdot J_i$ is a right weak Bruhat interval.
\end{proof}

According to ~\cref{thm: desc pres equiv}, every  equivalence class $C$ can be expressed as follows:
\[
C = \{[\gamma,  \xi_C\gamma]_L \mid \gamma \in [\sigma_0, \sigma_1]_R\},
\]
where $\sigma_0$ and $\sigma_1$ represent the minimal and maximal elements in $\omin(C)$, respectively.
In particular,
when $C$ is the equivalence class of $\Sigma_L(P)$ for $P \in \sfRSP_n$,
we can provide an explicit description of it.

\begin{theorem}\label{thm: equiv class X lam mu}
Let $P \in \sfRSP_n$ and $C$ the equivalence class of $\Sigma_L(P)$ under $\Deq$.
Then, 
\[
C = \{\Sigma_L(Q) \mid \text{$Q \in \sfRSP_n$ with $\sh(\tau_Q) = \sh(\tau_P)$} \}.
\]
\end{theorem}

\begin{proof}
Let $\lambda/\mu = \sh(\tau_P)$.
Combining~\cref{eq: bijection btw RSP and DS} with \cref{thm: interval descriptrion for regular skew schur poset} yields that 
\[
\{\Sigma_L(Q) \mid \text{$Q \in \sfRSP_n$ with $\sh(\tau_Q) = \lambda / \mu$} \} = \{[\rmread_{\tau}(T_{\lambda/\mu}), \rmread_{\tau}(T'_{\lambda/\mu})]_L \mid 
\tau \in \sfDS(\lambda/\mu) \}.
\]
Therefore, for the assertion, we have only to show the equality
\begin{align*}
C = \{[\rmread_{\tau}(T_{\lambda/\mu}), \rmread_{\tau}(T'_{\lambda/\mu})]_L \mid 
\tau \in \sfDS(\lambda/\mu) \}.
\end{align*}

First, let us show that 
$\{[\rmread_{\tau}(T_{\lambda/\mu}), \rmread_{\tau}(T'_{\lambda/\mu})]_L \mid 
\tau \in \sfDS(\lambda/\mu) \} \subseteq C$.
This can be done by proving that for $\tau \in \sfDS(\lambda/\mu)$,
the map 
\begin{align*}
f_{P;\tau}: [\rmread_{\tau_P}(T_{\lambda/\mu}), \rmread_{\tau_P}(T'_{\lambda/\mu})]_L & \ra [\rmread_{\tau}(T_{\lambda/\mu}), \rmread_{\tau}(T'_{\lambda/\mu})]_L
\\ 
\rmread_{\tau_P}(T) &\mapsto \rmread_{\tau}(T)  \quad (T\in \SYT(\lambda / \mu))
\end{align*}
is a descent-preserving isomorphism.
Let us fix $\tau \in \sfDS(\lambda/\mu)$.
The definition of $\tau$-reading implies that 
for any $T_1,T_2 \in \SYT(\lambda/\mu)$, 
$$
\rmread_{\tau}(T_1) \preceq_L^\rmc \rmread_{\tau}(T_2)
\quad \text{if and only if} \quad
\rmread_{\tau_P}(T_1) \preceq_L^\rmc \rmread_{\tau_P}(T_2),
$$
and therefore $f_{P;\tau}$ is a poset isomorphism.
To show that $f_{P;\tau}$ is descent-preserving,   
choose arbitrary $T \in \SYT(\lambda/\mu)$ and $i \in \Des_L(\rmread_{\tau_P}(T))$.
Combining   the conditions $T \in \SYT(\lambda/\mu)$ and $\tau_P \in \sfDS(\lambda /\mu)$ with $i \in \Des_L(\rmread_{\tau_P}(T))$ yields that $i+1$ appears weakly above and strictly right of $i$ in $T$.
It follows that $i \in \Des_L(\rmread_{\tau}(T))$, so $\Des_L(\rmread_{\tau_P}(T)) \subseteq \Des_L(\rmread_{\tau}(T))$.
In the same manner, one can show that $\Des_L(\rmread_{\tau}(T)) \subseteq \Des_L(\rmread_{\tau_P}(T))$.
Therefore, $f_{P;\tau}$ is a descent-preserving isomorphism.

Next, let us show $C \subseteq \{[\rmread_{\tau}(T_{\lambda/\mu}), \rmread_{\tau}(T'_{\lambda/\mu})]_L \mid 
\tau \in \sfDS(\lambda/\mu) \}$.
In the previous paragraph, we prove that  $[\rmread_{\tau}(T_{\lambda/\mu}), \rmread_{\tau}(T'_{\lambda/\mu})]_L \in C$ for any $\tau \in \sfDS(\lambda/\mu)$.
This implies that $\rmread_{\tau}(T'_{\lambda/\mu}) = \xi_C \rmread_{\tau}(T_{\lambda/\mu})$, and so it suffices to show that
$$
\omin(C) \subseteq \{\rmread_\tau(T_{\lambda/\mu}) \mid \tau \in \sfDS(\lambda / \mu)\}.
$$
Due to \cref{lem: fixed T int}, 
this inclusion
can be obtained by proving 
$$
\gamma \in [\rmread_{\tau_0}(T_{\lambda/\mu}), \rmread_{\tau_1}(T_{\lambda/\mu})]_R \quad \text{for any $\gamma \in \omin(C)$.}
$$
Let $\gamma \in \omin(C)$.
Since $\rmread_{\tau_0}(T_{\lambda/\mu}) \in \omin(C)$, we have $\Des_L(\rmread_{\tau_0}(T_{\lambda/\mu})) = \Des_L(\gamma)$.
In addition, by the definitions of $\tau_0$ and $T_{\lambda/\mu}$, we have $\rmread_{\tau_0}(T_{\lambda/\mu}) = w_0(\alpha^\rmc)$, where  
$\alpha = (\lambda_1 - \mu_1, \lambda_2 - \mu_2, \ldots, \lambda_{\ell(\lambda)} - \mu_{\ell(\lambda)})$.
Putting these equalities together with~\cref{lem: same desc int} yields that $\rmread_{\tau_0}(T_{\lambda/\mu}) \preceq_{R} \gamma$.
Similarly, we have
\[
\Des_L(\rmread_{\tau_1}(T'_{\lambda/\mu})) = \Des_L(\xi_C\gamma)
\quad \text{and} \quad
\rmread_{\tau_1}(T'_{\lambda/\mu}) = w_0(\beta^\rmc)w_0,
\]
where $\beta = (\lambda^\rmt_1 - \mu^\rmt_1, \lambda^\rmt_2 - \mu^\rmt_2,
\ldots, \lambda^\rmt_{\ell(\lambda^\rmt)} - \mu^\rmt_{\ell(\lambda^\rmt)})$.
This, together with \cref{lem: same desc int}, yields that $\xi_C\gamma  \preceq_{R} \rmread_{\tau_1}(T'_{\lambda/\mu})$.
Since $\rmread_{\tau_1}(T_{\lambda/\mu}) \preceq_L  \rmread_{\tau_1}(T'_{\lambda/\mu}) = \xi_C \rmread_{\tau_1}(T_{\lambda/\mu})$, we have $\gamma \preceq_{R} \rmread_{\tau_1}(T_{\lambda/\mu})$.
Therefore, $\gamma \in [\rmread_{\tau_0}(T_{\lambda/\mu}), \rmread_{\tau_1}(T_{\lambda/\mu})]_R$ as desired.
\end{proof}

\cref{thm: equiv class X lam mu} tells us that
$\{\Sigma_L(P) \mid P \in \sfRSP_n \}$ is closed under $\Deq$ and the  equivalence classes inside it are parametrized by the skew partitions of size $n$.
Given a skew partition $\lambda/\mu$ of size $n$, 
let $C_{\lambda/\mu}$ be the equivalence class parametrized by $\lambda/\mu$, that is, 
\[
C_{\lambda/\mu} = \{\Sigma_L(P) \mid \text{$P \in \sfRSP_n$ with $\sh(\tau_P)= \lambda/\mu$} \}.
\]

\begin{corollary}\label{cor: equiv class closedness}
With the above notation,
we have
\[
\{\Sigma_L(P) \mid P \in \sfRSP_n \} = \bigsqcup_{|\lambda/\mu| = n} C_{\lambda / \mu} \quad \text{$($disjoint union$)$}.
\]
\end{corollary}

\section{The classification of $\sfM_P$'s  for $P \in \sfRSP_n$}
\label{sec: classification of X}

Let $P, Q \in \sfRSP_n$. By combining \cref{Prop: desc pres isom and 0-Hecke} with \cref{thm: equiv class X lam mu}, we can see that if $\tau_P$ and $\tau_Q$ have the same shape, then the $H_n(0)$-modules $\sfM_P$ and $\sfM_Q$ are isomorphic.
The purpose of this section is to demonstrate that the converse of this implication also holds. 
Let us briefly explain our strategy. First, we provide both a projective cover and an injective hull of $\sfM_P$ for every $P \in \sfRSP_n$. We discover that these modules are completely determined by the shape of $\tau_P$, as demonstrated in \cref{lem: proj cov and inj hull}. Then, we establish that if $\tau_P$ and $\tau_Q$ have different shapes, $\sfM_P$ and $\sfM_Q$ have either nonisomorphic projective covers or nonisomorphic injective hulls, as proven in \cref{thm: classification}.

To begin with, we present a brief overview of the background knowledge concerning projective modules and injective modules of the $0$-Hecke algebras.
In \cite[Proposition 4.1]{02DHT}, it was shown that $H_n(0)$ is a Frobenius algebra.
It is well known that every Frobenius algebra is self injective and for a finitely generated module $M$ of a self injective algebra, $M$ is projective if and only if it is injective (for instance, see \cite[Proposition 1.6.2]{91Benson}).
In \cite{79Norton}, a complete list of non-isomorphic projective indecomposable $H_n(0)$-modules was provided.

In the work \cite{22JKLO}, it was shown that this list can also be expressed in terms of weak Bruhat interval modules, specifically as $\{\bfP_\alpha \mid \alpha \models n\}$, where
\[
\bfP_\alpha := \sfB(w_0(\alpha^\rmc), w_0 w_0(\alpha)) \quad \text{for $\alpha \models n$.}
\]
We note that $\bfP_\alpha / \rad \; \bfP_\alpha$ is isomorphic to $\bfF_\alpha$, where $\rad \; \bfP_\alpha$ is the \emph{radical} of $\bfP_\alpha$, the intersection of maximal submodules of $\bfP_\alpha$.

In the following, we recall the definition of a projective cover and an injective hull.
Let $M$ be a finitely generated $H_n(0)$-module. 
A \emph{projective cover} of $M$ is a pair $(\bdP, f)$ consisting of a projective $H_n(0)$-module $\bdP$ and an $H_n(0)$-module epimorphism $f: \bdP \rightarrow M$ such that $\ker(f) \subseteq \rad(\bdP)$.
An \emph{injective hull} of $M$ is a pair $(\bdI, \iota)$, where $\bdI$ is an injective $H_n(0)$-module and $\iota: M \rightarrow \bdI$ is an $H_n(0)$-module monomorphism satisfying $\iota(M) \supseteq \soc(\bdI)$.
Here, $\soc(\bdI)$  is the \emph{socle} of $\bdI$, the sum of all irreducible submodules of $\bdI$.
A projective cover and an injective hull of $M$ always exist and they are unique up to isomorphism.
For more information, refer to \cite{95ARS, 99Lam}.

The projective modules introduced by Huang \cite{16Huang} play an important role in describing the projective cover and injective hull of $\sfM_P$ for $P \in \sfRSP_n$.
We briefly review these projective modules from the viewpoint of weak Bruhat interval modules.
A \emph{generalized composition} $\bal$ of $n$ is a formal expression $\alpha^{(1)} \star \alpha^{(2)} \star \cdots \star  \alpha^{(k)}$,
where $\alpha^{(i)} \models n_i$ for positive integers $n_i$'s with $n_1 + n_2 + \cdots + n_k = n$.
For compositions $\alpha = (\alpha_1, \alpha_2, \ldots, \alpha_{\ell(\alpha)})$ and $\beta = (\beta_1, \beta_2, \ldots, \beta_{\ell(\beta)})$, let 
$$
\alpha \cdot \beta = (\alpha_1, \alpha_2, \ldots, \alpha_{\ell(\alpha)}, \beta_1, \beta_2, \ldots, \beta_{\ell(\beta)})
\quad \text{and} \quad 
\alpha \odot \beta = (\alpha_1, \alpha_2, \ldots, \alpha_{\ell(\alpha)}+\beta_1, \beta_2, \ldots, \beta_{\ell(\beta)}).
$$
For a generalized composition $\bal = \alpha^{(1)} \star \alpha^{(2)} \star \cdots \star \alpha^{(k)}$, let
\[
\bal_{\bullet} := \alpha^{(1)} \cdot \alpha^{(2)} \cdot \cdots \cdot \alpha^{(k)}, \quad \bal_{\odot} := \alpha^{(1)} \odot \alpha^{(2)} \odot \cdots \odot \alpha^{(k)}
\]
and let 
\[
\bal^\rmc  := (\alpha^{(1)})^\rmc \star (\alpha^{(2)})^\rmc \star \cdots \star (\alpha^{(k)})^\rmc,
\quad
\bal^\rmr  := (\alpha^{(k)})^\rmr \star (\alpha^{(k-1)})^\rmr \star \cdots \star (\alpha^{(1)})^\rmr,
\]
and $\bal^{\rmc \cdot \rmr} := (\bal^\rmc)^\rmr$.
Normally, $(\bal_{\bullet})^\rmc \neq (\bal^\rmc)_{\bullet}$ and $(\bal_{\odot})^\rmc \neq (\bal^\rmc)_{\odot}$ for a generalized composition $\bal$.
Despite the potential for confusion, for the sake of brevity, we denote $(\bal_{\bullet})^\rmc$ and $(\bal_{\odot})^\rmc$  as $\bal_{\bullet}^\rmc$ and $\bal_{\odot}^\rmc$, respectively.
Then, we define 
$$
\bfP_\bal := 
\sfB(w_0(\bal_{\bullet}^\rmc), w_0w_0(\bal_{\odot})).
$$
Huang decomposed $\bfP_{\bal}$ into projective indecomposable modules, thus showed that it is projective.
To be precise, the following lemma was shown.

\begin{lemma}{\rm (\cite[Theorem 3.3]{16Huang})}\label{lem: Huang decomp gen Pa}
For a generalized composition $\bal = \alpha^{(1)} \star \alpha^{(2)} \star \cdots \star \alpha^{(k)}$ of $n$,
\[
\bfP_\bal 
\cong 
\bfP_{\alpha^{(1)}} \boxtimes \bfP_{\alpha^{(2)}} \boxtimes \cdots \boxtimes \bfP_{\alpha^{(k)}}
\cong
\bigoplus_{\beta \in [\bal]} \bfP_\beta,
\]
where $[\bal] := \{\alpha^{(1)} \; \square \; \alpha^{(2)} \; \square \; \cdots \; \square \; \alpha^{(k)} \mid \square = \cdot \; \text{or} \; \odot  \}$.
\end{lemma}

It is clear from \cref{lem: Huang decomp gen Pa} that if $\bal$ and $\bbe$ are distinct generalized compositions of $n$, then 
$\bfP_\bal$ and $\bfP_\bbe$ are nonisomorphic.   
Let $\bal$ be a generalized composition of $n$.
For $\rho \in [w_0(\bal_\odot^\rmc), w_0 w_0(\bal_\odot)]_L$,
let $\Upsilon_{\bal;\rho}: \bfP_{\bal} \ra \sfB(w_0(\bal_\bullet^\rmc), \rho)$ be a $\C$-linear map  given by
\[
\gamma \mapsto 
\begin{cases}
\gamma & \text{if $\gamma \in [w_0(\bal_\bullet^\rmc), \rho]_L$,}\\
0 & \text{if $\gamma \in [w_0(\bal_\bullet^\rmc), w_0 w_0(\bal_\odot)]_L \setminus [w_0(\bal_\bullet^\rmc), \rho]_L$.}
\end{cases}
\]
Clearly, $\Upsilon_{\bal;\rho}$ is an $H_n(0)$-module epimorphism.
In addition, it follows from~\cite[Lemma 6.2]{22KY} that $\ker(\Upsilon_{\bal;\rho}) \subseteq \rad (\bfP_\bal)$.
Consequently, we have the following lemma.

\begin{lemma}\label{lem: radical cond}{\rm (cf. \cite[Lemma 6.2]{22KY})}
For a generalized composition $\bal$ of $n$ and $\rho \in [w_0(\bal_\odot^\rmc), w_0 w_0(\bal_\odot)]_L$,
the pair $(\bfP_\bal, \Upsilon_{\bal;\rho})$ is a projective cover of $\sfB(w_0(\bal_\bullet^\rmc), \rho)$.
\end{lemma}

Let us provide  notation and a lemma needed to describe a projective cover and an injective hull of $\sfM_P$ for $P \in \sfRSP_n$. 
For a connected skew partition $\lambda / \mu$ of size $n$, define
\[
\balproj(\lambda/\mu) := (\lambda_1 - \mu_1, \lambda_2 - \mu_2, \ldots , \lambda_{\ell(\lambda)} - \mu_{\ell(\lambda)}).
\]
And, for a disconnected skew partition $\lambda / \mu$ of size $n$, define
\[
\balproj(\lambda/\mu) := \balproj(\lambda^{(1)} / \mu^{(1)}) 
\star \balproj(\lambda^{(2)} / \mu^{(2)})
\star \cdots \star
\balproj(\lambda^{(k)} / \mu^{(k)}),
\]
where $\lambda/\mu = \lambda^{(1)} / \mu^{(1)} \star \lambda^{(2)} / \mu^{(2)} \star \cdots \star \lambda^{(k)} / \mu^{(k)}$ with connected skew partitions $\lambda^{(i)} / \mu^{(i)}$'s $(1\le i \le k)$.

\begin{lemma}\label{lem: balproj and read}
Let $\lambda/\mu$ be a skew partition of size $n$.
\begin{enumerate}[label = {\rm (\arabic*)}]
\item 
$\rmread_{\tau_0}(T_{\lambda/\mu}) = w_0(\balproj(\lambda/\mu)_{\bullet}^\rmc)$.
\item 
$\rmread_{\tau_0}(T'_{\lambda/\mu}) \in [w_0(\balproj(\lambda/\mu)_{\odot}^\rmc), w_0w_0(\balproj(\lambda/\mu)_{\odot})]_L$.
\end{enumerate}
\end{lemma}

\begin{proof}
By the definition of $\balproj(\lambda/\mu)$, the assertion (1) is clear.
In addition, one can easily see that
$\Des_R(\rmread_{\tau_0}(T'_{\lambda/\mu})) = \set(\balproj(\lambda/\mu)_{\odot}^\rmc)$.
So, by \cref{lem: same desc int}, the assertion (2) follows.
\end{proof}

Let $P \in \sfRSP_n$ and $\lambda/\mu = \sh(\tau_P)$.
By ~\cref{thm: interval descriptrion for regular skew schur poset}, 
$\sfM_P = \sfB(\rmread_{\tau_P}(T_{\lambda/\mu}), \rmread_{\tau_P}(T'_{\lambda/\mu}))$.
Furthermore, by ~\cref{thm: equiv class X lam mu}, we have an $H_n(0)$-module isomorphism
$$
f_{P}: \sfM_{\sfposet(\tau_0)} \ra \sfM_P,
\ \ 
\rmread_{\tau_0}(T) \mapsto \rmread_{\tau_P}(T) \;\;(T \in \SYT(\lambda/\mu)).
$$
Set 
$$
\eta_P := f_P \circ \Upsilon_{\balproj(\lambda/\mu); \rmread_{\tau_0}(T'_{\lambda/\mu})}.
$$
Combining \cref{lem: radical cond} and ~\cref{lem: balproj and read} implies that the pair $\left(\bfP_{\balproj(\lambda/\mu)}, \eta_P\right)$
is a projective cover of $\sfM_P$.

To find an injective hull of $\sfM_P$, we note that
\[
\rmread_{\tau_1}(T_{\lambda^\rmt/\mu^\rmt}) w_0 = \rmread_{\tau_0}(T'_{\lambda/\mu})
\quad \text{and} \quad
\rmread_{\tau_1}(T'_{\lambda^\rmt/\mu^\rmt}) w_0 = \rmread_{\tau_0}(T_{\lambda/\mu}).
\]
Combining these equalities with \cite[Theorem 4]{22JKLO} yields the following $H_n(0)$-module isomorphism:
\begin{align*}
g_1: \bfT^{-}_\hautotheta \left(
\sfM_{\sfposet(\tau_1^{\lambda^\rmt/\mu^\rmt})}\right) &\ra \sfM_{\sfposet(\tau_0^{\lambda/\mu})},
\\ 
\gamma^* &\mapsto (-1)^{\ell(\gamma \; \rmread_{\tau_1}(T'_{\lambda^\rmt/\mu^\rmt})^{-1})} \gamma w_0,
\end{align*}
where $\gamma \in [\rmread_{\tau_1}(T_{\lambda^\rmt/\mu^\rmt}), \rmread_{\tau_1}(T'_{\lambda^\rmt/\mu^\rmt})]_L$ and $\gamma^\ast$ denotes the dual of $\gamma$ with respect to the basis $[\rmread_{\tau_1}(T_{\lambda^\rmt/\mu^\rmt}), \rmread_{\tau_1}(T'_{\lambda^\rmt/\mu^\rmt})]_L$ for $\sfM_{\sfposet(\tau_1^{\lambda^\rmt/\mu^\rmt})}$.
Set
\[
\balinj(\lambda/\mu) := \balproj(\lambda^\rmt / \mu^\rmt)^{\rmc \cdot \rmr}.
\]
Again, by \cite[Theorem 4]{22JKLO}, we have the $H_n(0)$-module isomorphism 
\begin{align*}
g_2: \bfT^{-}_\hautotheta \left( \bfP_{\balproj(\lambda^\rmt / \mu^\rmt)}\right)
&\ra 
\bfP_{\balinj(\lambda/\mu)},
\\ 
\gamma^* &\mapsto (-1)^{\ell(\gamma (w_0w_0(\balinj(\lambda/\mu)_{\odot}))^{-1})} \gamma w_0,
\end{align*}
where $\gamma \in [
w_0(\balproj(\lambda^\rmt / \mu^\rmt)_{\bullet}^\rmc), w_0w_0(\balproj(\lambda^\rmt / \mu^\rmt)_{\odot})
]_L$ and $\gamma^\ast$ denotes the dual of $\gamma$ with respect to the basis $[
w_0(\balproj(\lambda^\rmt / \mu^\rmt)_{\bullet}^\rmc), w_0w_0(\balproj(\lambda^\rmt / \mu^\rmt)_{\odot})
]_L$ for $\bfP_{\balproj(\lambda^\rmt / \mu^\rmt)}$.
Set $\eta_{\sfposet(\tau_1^{\lambda^\rmt/\mu^\rmt})} := f_{\sfposet(\tau_1^{\lambda^\rmt/\mu^\rmt})} \circ \Upsilon_{\balproj(\lambda^\rmt/\mu^\rmt); \rmread_{\tau_0}(T'_{\lambda^\rmt/\mu^\rmt})}$.
As above, the pair
$\left(\bfP_{\balproj(\lambda^\rmt /\mu^\rmt)}, \eta_{\sfposet(\tau_1^{\lambda^\rmt/\mu^\rmt})}\right)$
is a projective cover of $\sfM_{\sfposet(\tau_1^{\lambda^\rmt/\mu^\rmt})}$.
And, since $\bfT^{-}_\hautotheta$ is contravariant, 
$\left( \bfT^{-}_\hautotheta \left(
\bfP_{\balproj(\lambda^\rmt /\mu^\rmt)}
\right), 
\bfT^{-}_\hautotheta 
\left(
\eta_{\sfposet(\tau_1^{\lambda^\rmt/\mu^\rmt})} 
\right)
\right)$
is an injective hull of $\bfT^{-}_\hautotheta(\sfM_{\sfposet(\tau_1^{\lambda^\rmt/\mu^\rmt})})$
(for the definition of $\bfT^{-}_\hautotheta $, see ~\cref{subsec: 0-Hecke alg and QSym}).
Consequently, the pair
$\left(\bfP_{\balinj(\lambda/\mu)}, \iota_P\right)$ is an injective hull of $\sfM_P$, where
\[
\iota_P = g_2 \circ \bfT^{-}_\hautotheta 
\left(
\eta_{\sfposet(\tau_1^{\lambda^\rmt/\mu^\rmt})} 
\right) \circ g_1^{-1} 
\circ f_P^{-1}.
\]

To summarize, we can state the following lemma.
\begin{lemma}\label{lem: proj cov and inj hull}
Let $P \in \sfRSP_n$ and $\lambda/\mu = \sh(\tau_P)$.
\begin{enumerate}[label = {\rm (\arabic*)}]
\item $\left( \bfP_{\balproj(\lambda/\mu)}, \eta_P \right)$ is a projective cover of $\sfM_P$.
\item $\left(\bfP_{\balinj(\lambda/\mu)}, \iota_P \right)$ is an injective hull of $\sfM_P$.
\end{enumerate}
\end{lemma}

Now, we are ready to state the  classification of $\sfM_P$'s for $P \in \sfRSP_n$ up to $H_n(0)$-module isomorphism.

\begin{theorem}\label{thm: classification}
Let $P, Q \in \sfRSP_n$.
Then
\[
\sfM_{P} \cong \sfM_{Q} \quad \text{if and only if} \quad \sh(\tau_P) = \sh(\tau_Q).
\]
\end{theorem}

\begin{proof}
The ``if'' part follows from \cref{Prop: desc pres isom and 0-Hecke} and \cref{thm: equiv class X lam mu}.
To prove the ``only if'' part, suppose that $\sfM_{P} \cong \sfM_{Q}$.
For simplicity, let $\lambda/\mu = \sh(\tau_{P})$ and $\nu/\kappa = \sh(\tau_{Q})$.
By \cref{lem: proj cov and inj hull}, 
$\bfP_{\balproj(\lambda/\mu)} \cong \bfP_{\balproj(\nu/\kappa)}$ and $\bfP_{\balinj(\lambda/\mu)} 
\cong \bfP_{\balinj(\nu/\kappa)}$,
and therefore $\balproj(\lambda/\mu) = \balproj(\nu/\kappa)$ and $\balinj(\lambda/\mu) = \balinj(\nu/\kappa)$.
Since $\balproj(\lambda/\mu) = \balproj(\nu/\kappa)$, the number of boxes in the same row of $\tyd(\lambda/\mu)$ and $\tyd(\nu/\kappa)$ are the same.
Similarly, since $\balinj(\lambda/\mu) = \balinj(\nu/\kappa)$, 
the number of boxes in the same column of $\tyd(\lambda/\mu)$ and $\tyd(\nu/\kappa)$ are same.
Thus, we have  $\lambda/\mu = \nu/\kappa$.
\end{proof}

Note that \cref{thm: classification} is the classification theorem concerning the class of $H_n(0)$-modules $\{\sfM_P \mid P \in \sfRSP_n\}$. Consequently, a natural question arises: can this theorem be extended to the classes $\{\sfM_P \mid P \in \sfRP_n\}$ or $\{\sfM_P \mid P \in \sfSP_n\}$? This question appears to be highly nontrivial, as it involves the investigation of a broader set of modules. 
As a specific instance, let us examine the characterization of posets $Q \in \sfRP_n$ such that $\sfM_Q \cong \sfM_P$ when $P \in \sfRSP_n$. This problem can be readily addressed by assuming the validity of the following conjecture due to Stanley.

\begin{conjecture}{\rm (\cite[p. 81]{72Stanley})}\label{conj: Stanley}
For $P \in \sfP_n$, if $K_P$ is symmetric, then $P \in \sfSP_n$.
\end{conjecture}

In more detail, by combining Stanley's conjecture with \cref{thm: char of Mp}(1), we can deduce that $\ch([\sfM_Q])$ is not symmetric, and as a consequence, $\sfM_Q \not\cong \sfM_P$ unless $Q \in \sfSP_n$. This observation leads to the following conclusion from \cref{thm: classification}:
\[
\{Q \in \sfRP_n \mid \sfM_Q \cong \sfM_P\} 
= 
\{Q \in \sfRSP_n \mid \sh(\tau_P) = \sh(\tau_Q) \}.
\]
If the shape of $\tau_P$ is non-skew, it is indeed possible to derive this conjectural identity without depending on the validity of Stanley's conjecture (see \cref{prop: small evi for conj}). 
However, tackling the general case remains beyond our current comprehension.
For further discussions on classifications, refer to \cref{subsec: classification}.

\section{A characterization of regular Schur labeled skew shape posets $P$ and distinguished filtrations of $\sfM_P$}\label{Sec: filtration}

In this section, we prove that a poset $P \in \sfP_n$ is a regular Schur labeled skew shape poset if and only if $\Sigma_L(P)$ is dual plactic closed (\cref{thm: RSP iff DPC}).
Then, by considering the dual plactic closedness of $\Sigma_L(P)$, we construct   filtrations 
\[
0 =: M_0  \subsetneq M_1 \subsetneq M_2 \subsetneq \cdots \subsetneq M_l := \sfM_P
\]
such that $\ch([M_{k}/M_{k-1}])$ is a Schur function for all $1 \leq k \leq l$
(\cref{thm: filt of X lam mu}).

\subsection{A characterization of regular Schur labeled skew shape posets}\label{subsec: char of RSPn}

Let $P \in \sfP_n$ and let  $\Sigma_R(P) :=\{\gamma^{-1} \mid \gamma \in \Sigma_L(P)\}$.
In~\cite[Fact 1]{93M}, it was stated that if $P \in \sfSP_n$, then  $\Sigma_R(P)$ is plactic-closed. 
This, however, is not true.
For instance, considering the case where $\lambda/\mu = (3,2)/(2)$ and $\tau = \begin{ytableau}
\none & \none & 2\\
3 & 1
\end{ytableau} \in \sfS(\lambda/\mu)$,
we have
\[
\Sigma_R(\sfposet(\tau)) = \{312, 231, 321\}, 
\]
which is not plactic-closed.

The purpose of this subsection is to  prove that $P \in \sfRSP_n$ if and only if $\Sigma_R(P)$ is plactic-closed.
We begin by providing  background knowledge relevant to the plactic congruence.
For instance, see \cite{05BB, 97Ful, 91Sag, 99Stanley}.

For $\sigma \in \SG_n$ and $1 < i < n$, 
we write $\sigma \overset{1}{\cong} \sigma s_i$ if
\[
\sigma(i) < \sigma(i-1) < \sigma(i+1) 
\quad \text{or} \quad 
\sigma(i+1) < \sigma(i-1) < \sigma(i).
\]
And, we write $\sigma \overset{2}{\cong} \sigma s_{i-1}$ if
\[
\sigma(i-1) < \sigma(i+1) < \sigma(i) 
\quad \text{or} \quad 
\sigma(i) < \sigma(i+1) < \sigma(i-1).
\]
The \emph{Knuth equivalence} (or \emph{plactic congruence}) is an equivalence relation $\Keq$ on $\SG_n$ defined by $\sigma \Keq \rho$ if and only if there are $\gamma_1,\gamma_2, \ldots, \gamma_k \in \SG_n$ such that
\[
\sigma = \gamma_1 \overset{a_1}{\cong} \gamma_2 \overset{a_2}{\cong} \cdots \overset{a_{k-1}}{\cong} \gamma_k = \rho,
\]
where $a_1,a_2,\ldots a_{k-1} \in \{1,2\}$.
A subset $S$ of $\SG_n$ is called \emph{plactic-closed} if for any $\sigma \in S$, every $\rho \in \SG_n$ with $\rho \Keq \sigma$ is also an element of $S$,
in other words, $S$ is a union of equivalence classes under $\Keq$.

The \emph{dual Knuth equivalence} (or \emph{dual plactic congruence}) is an equivalence relation $\dKeq$ on $\SG_n$ defined by 
\[
\sigma \dKeq \rho
\quad \text{if and only if} \quad \sigma^{-1} \Keq \rho^{-1}.
\]
A subset $S$ of $\SG_n$ is called \emph{dual plactic-closed} if for any $\sigma \in S$, every $\rho \in \SG_n$ with $\rho \dKeq \sigma$ is also an element of $S$,
in other words, $S$ is a union of equivalence classes under $\dKeq$.

The Knuth and dual Knuth equivalences are closely related to the \emph{Robinson–Schensted correspondence}, which is a one-to-one correspondence between $\SG_n$ and $\bigcup_{\lambda \vdash n} \SYT(\lambda) \times \SYT(\lambda)$.
For $\sigma \in \SG_n$,
we use the notation $(\Ptab(\sigma), \Qtab(\sigma))$ to represent the image of $\sigma$ under this bijection.
We call $\Ptab(\sigma)$ and $\Qtab(\sigma)$ as the \emph{insertion tableau} and \emph{recording tableau} of $\sigma$, respectively.
It is well known that $\Ptab(\sigma) = \Qtab(\sigma^{-1})$ 
and
\begin{align*}
\sigma \Keq \rho 
\quad \text{if and only if} \quad 
\Ptab(\sigma) = \Ptab(\rho)
\quad \text{for $\sigma, \rho \in \SG_n$}.
\end{align*}
Putting these together, one can easily derive that
\begin{align*}
\sigma \dKeq \rho 
\quad \text{if and only if} \quad 
\Qtab(\sigma) = \Qtab(\rho)
\quad \text{for $\sigma, \rho \in \SG_n$}.
\end{align*}
For a subset $S$ of $\SG_n$, $S$ is plactic-closed if and only if $\{\gamma^{-1} \mid \gamma \in S\}$ is dual plactic-closed.
Based on this fact, we will consider the claim that $\Sigma_R(P)$ is plactic-closed and the claim that $\Sigma_L(P)$ is dual plactic-closed to be identical.

Let us collect the terminologies and lemmas necessary for the proof of the main result of this subsection.
Let $T$ be a standard Young tableau of skew shape.
Denote by $\Rect(T)$ the \emph{rectification of $T$}, that is, the unique standard Young tableau of partition shape obtained by applying jeu de taquin slides to $T$ (see \cite[Section 1.2]{97Ful}).
Then, 
\begin{align}\label{eq: Ptab read and Rect}
\Rect(T) = 
\Ptab(\rmread_{\tau_0}(T) w_0) 
\quad \text{for any $T \in \SYT(\lambda/\mu)$.}
\end{align}
Define $T^\rmt$ to be the tableau obtained from $T$ by flipping it along its main diagonal.

\begin{lemma}\label{lem: admissible same Q}
Let $\lambda / \mu$ be a skew partition and $T \in \SYT(\lambda/\mu)$.
Then, 
\[
\Ptab(\rmread_\tau(T)) = \Rect(T)^\rmt
\quad
\text{for any $\tau \in \sfDS(\lambda/\mu)$.}
\]
\end{lemma}

\begin{proof}
It is well known that
\begin{align}\label{eq: Ptab and trans}
\Ptab(\sigma w_0) = \Ptab(\sigma)^\rmt
\quad \text{for any $\sigma \in \SG_n$}    
\end{align}
(for instance, see \cite[Theorems 3.2.3]{91Sag}).
Therefore, 
due to \cref{eq: Ptab read and Rect},
the assertion can be verified by showing  that 
\begin{align}\label{eq: wts in Lem}
\Ptab(\rmread_{\tau}(T)) = \Ptab(\rmread_{\tau^{\lambda/\mu}_0}(T))
\quad 
\text{
for any $\tau \in \sfDS(\lambda/\mu)$.}
\end{align}

Applying Ta\c{s}kin's result \cite[Proposition 3.2.5]{06Taskin} to the weak order\footnote{This order was originally defined in ~\cite[2.5.1]{04Mel}, where it is called the \emph{induced Duflo order}.}  on $\SYT_n$ given in \cite[Definition 3.1.3]{06Taskin}, we derive that for $\sigma, \rho \in \SG_n$ with $\sigma \preceq_R \rho$,
\begin{align}\label{eq: bruhat order and shape}
\Ptab(\sigma) = \Ptab(\rho)
\quad \text{or} \quad
\sh(\Ptab(\rho)) \triangleleft \sh(\Ptab(\sigma)).
\end{align}
Here, $\trianglelefteq$ denotes the dominance order on the set of partitions of $n$.
On the other hand,   \cref{lem: fixed T int} says that  
\begin{align}\label{eq: right order on same tab}
\rmread_{\tau_0^{\lambda/\mu}}(T) \preceq_R \rmread_{\tau}(T) \preceq_R \rmread_{\tau_1^{\lambda/\mu}}(T) \quad \text{for $\tau \in \sfDS(\lambda/\mu)$}.
\end{align}
Note that 
\[
\Ptab(\rmread_{\tau^{\lambda^\rmt/\mu^\rmt}_0}(T^\rmt)w_0)
\underset{\cref{eq: Ptab read and Rect}}{=} \Rect(T^\rmt)
= \Rect(T)^\rmt
\underset{\cref{eq: Ptab read and Rect}}{=} \Ptab(\rmread_{\tau^{\lambda/\mu}_0}(T)w_0)^\rmt.
\]
Since $\rmread_{\tau^{\lambda/\mu}_1}(T) = 
\rmread_{\tau^{\lambda^\rmt/\mu^\rmt}_0}(T^\rmt) w_0$,
it follows from ~\cref{eq: Ptab and trans} that 
\begin{align}\label{eq: same P tau zero one}
\Ptab(\rmread_{\tau^{\lambda/\mu}_1}(T)) 
= \Ptab(\rmread_{\tau^{\lambda/\mu}_0}(T)).
\end{align}
Now, the equality in \cref{eq: wts in Lem} is obtained by combining ~\cref{eq: bruhat order and shape}, ~\cref{eq: right order on same tab}, and \cref{eq: same P tau zero one}.
\end{proof}

We introduce two important results due to Malvenuto \cite{93M}.

\begin{lemma}{\rm(\cite[Theorem 1]{93M})}\label{lem: thm of Mal}
For $P \in \sfP_n$, if $\Sigma_R(P)$ is plactic-closed, then $P$ is a Schur labeled skew shape poset.
\end{lemma}

For $P \in \sfP_n$, we say a subposet $Q$ of $P$ is \emph{convex} if $Q$ satisfies the property that for any $x \in P$ if there exist $y_1, y_2 \in Q$ such that $y_1 \preceq_P x \preceq_P y_2$, then $x \in Q$.
For a subposet $Q = \{i_1 < i_2 < \cdots <i_{|Q|} \}$ of $P \in \sfP_n$, the \emph{standardization of $Q$}, denoted by $\mathsf{st}(Q)$, is the poset obtained from $Q$ by replacing $i_j$ with $j$ for $1 \leq j \leq |Q|$.

\begin{lemma}{\rm(\cite[Corollary 1]{93M})}\label{lem: cor of Mal}
Let $P \in \sfP_n$ such that $\Sigma_R(P)$ is plactic-closed.
For any convex subposet $Q$ of $P$, $\Sigma_R(\mathsf{st}(Q))$ is plactic-closed.
\end{lemma}

Now, we are ready to prove the main result of this subsection.

\begin{theorem}\label{thm: RSP iff DPC}
For $P \in \sfP_n$, $P$ is a regular Schur labeled skew shape poset if and only if $\Sigma_L(P)$ is dual plactic-closed.
\end{theorem}

\begin{proof}
To establish ``only if'' part, let $P \in \sfRSP_n$ and $\lambda/\mu = \sh(\tau_P)$.
Due to \cref{lem: identification of P and S}, we have that 
\[\Sigma_L(P) = \rmread_{\tau_P}(\SYT(\lambda/\mu)).
\]
We claim that $\rmread_{\tau_P}(\SYT(\lambda/\mu))$ is dual plactic-closed.

As mentioned in~\cite[Property A]{85GR}, one can easily see that $\rmread_{\tau_0}(\SYT(\lambda/\mu))$ is dual plactic-closed.
\footnote{
\cite[Property A]{85GR} is stated as ``For any skew diagram $D$ the collection $W^{-1}(D)$ is a union of Knuth equivalence classes''.
Following the notation of this paper, $W^{-1}(\tyd(\lambda / \mu)) = \{(\rmread_{\tau_0}(T) w_0)^{-1} \mid T \in \SYT(\lambda/\mu)\}$.
So, \cite[Property A]{85GR} says that the set $\rmread_{\tau_0}(\SYT(\lambda/\mu)) w_0 := \{\rmread_{\tau_0}(T) w_0 \mid T \in \SYT(\lambda/\mu)\}$ is dual plactic-closed.
Although $\rmread_{\tau_0}(\SYT(\lambda/\mu))$ is different from $\rmread_{\tau_0}(\SYT(\lambda/\mu)) w_0$,  the dual plactic closedness of $\rmread_{\tau_0}(\SYT(\lambda/\mu))$ can be proved in the same way as that of $\rmread_{\tau_0}(\SYT(\lambda/\mu)) w_0$.
}
In addition, by \cref{lem: admissible same Q}, we have
\begin{align}\label{eq: admissible same Q}
\Ptab(\rmread_{\tau_P}(T)) = \Ptab(\rmread_{\tau_0}(T)) 
\quad \text{for all $T \in \SYT(\lambda/\mu)$}.
\end{align}
Therefore, given $T \in \SYT(\lambda/\mu)$, if we show that 
\[
\rmread_{\tau_P}(T) \dKeq \rmread_{\tau_P}(U)
\quad \text{
for all $U \in \SYT(\lambda/\mu)$ with $\rmread_{\tau_0}(T) \dKeq \rmread_{\tau_0}(U)$,}
\]
then we have
\[
\{\gamma \in \SG_n \mid \gamma \dKeq \rmread_{\tau_P}(T)\} \subseteq \rmread_{\tau_P}(\SYT(\lambda/\mu)).
\]

Let $T, U \in \SYT(\lambda/\mu)$ with $\rmread_{\tau_0}(T) \dKeq \rmread_{\tau_0}(U)$.
Since $\rmread_{\tau_0}(\SYT(\lambda/\mu))$ is dual plactic-closed, there exist standard Young tableaux $T_0 := T, T_1, \ldots, T_l := U$ of shape $\lambda/\mu$ 
such that for any $1 \leq k \leq l$, 
\begin{align*}
\rmread_{\tau_0}(T_k) \dKeq\rmread_{\tau_0}(T) 
\ \  \text{and} \ \ 
\rmread_{\tau_0}(T_k) = s_{i_k} \rmread_{\tau_0}(T_{k-1}) \ \  \text{for some $i_k \in [n-1]$}. 
\end{align*}
Combining \cref{eq: admissible same Q} with the equality $\sh(\Ptab(\rmread_{\tau_0}(T_k))) = \sh(\Ptab(\rmread_{\tau_0}(T)))$, we have
\begin{align}\label{eq: shape of Ptab of read}
\sh(\Ptab(\rmread_{\tau_P}(T_k))) = \sh(\Ptab(\rmread_{\tau_P}(T)))
\quad \text{for all $1 \le k \le l$.}
\end{align}
Note that \cref{eq: bruhat order and shape} is equivalent with 
for $\sigma, \rho \in \SG_n$ with $\sigma \preceq_L \rho$,
\begin{align}\label{eq: left bruhat order and shape}
\sigma \dKeq \rho
\quad \text{or} \quad
\sh(\Qtab(\rho)) \triangleleft \sh(\Qtab(\sigma)).
\end{align}
Putting \cref{eq: shape of Ptab of read} together with \cref{eq: left bruhat order and shape}, we have 
$\rmread_{\tau_P}(T) \dKeq \rmread_{\tau_P}(U)$.
Since we chose arbitrary $T,U \in \SYT(\lambda/\mu)$, we conclude that $\rmread_{\tau_P}(\SYT(\lambda/\mu))$ is dual plactic-closed.

To establish the ``if'' part of the assertion, we prove the contraposition,
that is, if $P$ is not a regular Schur labeled skew shape poset, then $\Sigma_L(P)$ is not dual plactic-closed.
If $P$ is not a Schur labeled skew shape poset, then \cref{lem: thm of Mal} says that $\Sigma_L(P)$ is not dual plactic-closed.
So, we assume that $P \in \sfP_n$ is a non-regular Schur labeled skew shape poset.

One can easily check that if $n = 1,2,3$, then $\Sigma_L(P)$ is not dual plactic-closed.
Suppose $n > 3$.
Then, by \cref{lem: characterization: regular},  $\tau_P$ is a non-distinguished Schur labeling.
This implies that there exists $k \in \Z_{>0}$ such that $\sfcnt_{k}(\tau_P)$ is not filled with consecutive integers.
Let $k_0$ be the minimum among these integers and let $m_0$ be the minimum element among $m \in \sfcnt_{k_0}(\tau_P)$ such that $m + 1 \notin \sfcnt_{k_0}(\tau_P)$.
Since $\sfcnt_{k_0}(\tau_P)$ is not filled with consecutive integers, 
we can choose 
$$
m_1 = \min \{m \in \sfcnt_{k_0}(\tau_P) \mid m > m_0\}.
$$
Since $m_0$ and $m_1$ are in the same connected component of the Schur labeling $\tau_P$ and $m_0 < m_1$, we can take $m_{-1} \in \sfcnt_{k_0}(\tau_P)$ such that $m_{-1} < m_1$ and $m_{-1}$ is adjacent to $m_1$. Here, the sentence ``$m_{-1}$ is adjacent to $m_1$'' means that the box containing $m_{-1}$ and that containing $m_{1}$ share an edge.
We note that $m_{-1}$ can be $m_0$.
Because of the choice of $m_1$, we have $m_{-1} < m_0+1 < m_1$.
Let $Q$ be the subposet of $P$ whose underlying set is $\{m_{-1}, m_0 + 1, m_1\}$.
In $P$, $m_1$ covers $m_{-1}$ and $m_0 + 1$ is incomparable with both $m_1$ and $m_{-1}$.
This implies that $Q$ is a convex subposet of $P$. 
In addition, since $m_{-1} < m_0 + 1 < m_1$, we have $\Sigma_L(\mathsf{st}(Q)) = \{123,132,213\}$ or $\Sigma_L(\mathsf{st}(Q)) = \{312,231,321\}$.
Thus, $\Sigma_L(\mathsf{st}(Q))$ is not dual plactic closed.
Combining this with \cref{lem: cor of Mal} yields that $\Sigma_L(P)$ is not dual plactic closed, as desired.
\end{proof}

\subsection{Distinguished filtrations of $\sfM_P$ for $P \in \sfRSP_n$}\label{subsec: distinguished filt}

We begin by introducing the definition of distinguished filtrations.

\begin{definition}\label{def: dist filt}
Let $\mathcal{B} = \{\mathcal{B}_\alpha \mid \alpha \in I\}$ be a linearly independent subset of $\Qsym_n$ 
with the property that  $\mathcal{B}_\alpha$ is $F$-positive for all $\alpha \in I$,
where $I$ is an index set.
Given a finite dimensional $H_n(0)$-module $M$,
a \emph{distinguished filtration of $M$ with respect to $\mathcal{B}$} is an $H_n(0)$-submodule series  of $M$ 
\[
0 =: M_0 \subset M_1 \subset M_2 \subset \cdots \subset M_l := M
\]
such that for all $1 \leq k \leq l$, $\ch([M_k / M_{k-1}]) = \mathcal{B}_\alpha$ for some $\alpha \in I$.
\end{definition}

As seen in ~\cref{eg: not dist filt},
a distinguished filtration of $M$ with respect to $\calB$ may not exist even if $\ch([M])$ expands positively in $\calB$.
This is because the category  $\Hnmod$ is  neither semisimple nor representation-finite when $n > 3$ (\cite{11DY, 02DHT}).

\begin{example}\label{eg: not dist filt}
Let $\calB = \{s_\lambda \mid \lambda \vdash 4 \}$.
For $B = \{2314, 1423, 3214,2413, 1432, 3412\}$,
let $M$ be the $H_4(0)$-module with underlying space $\C B$ and with the $H_4(0)$-action defined by
\begin{align*}
\pi_{i} \cdot \gamma :=
\begin{cases}
\gamma & \text{if $i \in \Des_L(\gamma)$}, \\
0 & \text{if $i \notin \Des_L(\gamma)$ and $s_i\gamma \notin B$,} \\
s_i \gamma & \text{if $i \notin \Des_L(\gamma)$ and $s_i\gamma \in B$}.
\end{cases} 
\end{align*}
The $H_4(0)$-action on $B \cup \{ 0 \}$  is illustrated in the following figure:
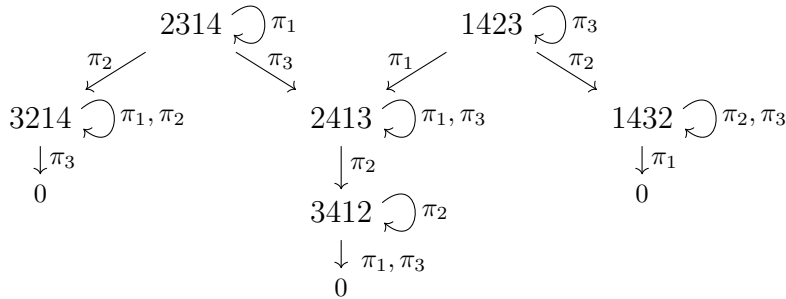
\[
\def \vp {1.25}
\def \hp {2}
\begin{tikzpicture}
\node[] at (0*\hp, 0*\vp) {$3412$};

\node[] at (0*\hp, 1*\vp) {$2413$};
\node[] at (-2*\hp, 1*\vp) {$3214$};
\node[] at (2*\hp, 1*\vp) {$1432$};

\node[] at (-1*\hp, 2*\vp) {$2314$};
\node[] at (1*\hp, 2*\vp) {$1423$};

\node at (0*\hp + 0.2*\hp, 0*\vp) {} edge [out=40,in=320, loop] ();
\node at (0*\hp + 0.625*\hp, 0*\vp) {\footnotesize $\pi_2$};

\node at (-2*\hp + 0.2*\hp, 1*\vp) {} edge [out=40,in=320, loop] ();
\node at (-2*\hp + 0.75*\hp, 1*\vp) {\footnotesize $\pi_1, \pi_2$};

\node at (0*\hp + 0.2*\hp, 1*\vp) {} edge [out=40,in=320, loop] ();
\node at (0*\hp + 0.75*\hp, 1*\vp) {\footnotesize $\pi_1, \pi_3$};

\node at (2*\hp + 0.2*\hp, 1*\vp) {} edge [out=40,in=320, loop] ();
\node at (2*\hp + 0.75*\hp, 1*\vp) {\footnotesize $\pi_2, \pi_3$};

\node at (-1*\hp + 0.2*\hp, 2*\vp) {} edge [out=40,in=320, loop] ();
\node at (-1*\hp + 0.625*\hp, 2*\vp) {\footnotesize $\pi_1$};

\node at (1*\hp + 0.2*\hp, 2*\vp) {} edge [out=40,in=320, loop] ();
\node at (1*\hp + 0.625*\hp, 2*\vp) {\footnotesize $\pi_3$};

\draw[->] (0*\hp, 0.7*\vp) -- (0, 0.3);
\node at (0.15*\hp, 0.5*\vp) {\footnotesize $\pi_2$};

\draw[->] (-1.3*\hp, 1.7*\vp) -- (-1.7*\hp, 1.3*\vp);
\node at (-1.6*\hp, 1.6*\vp) {\footnotesize $\pi_2$};

\draw[->] (-0.7*\hp, 1.7*\vp) -- (-0.3*\hp, 1.3*\vp);
\node at (-0.4*\hp, 1.6*\vp) {\footnotesize $\pi_3$};

\draw[->] (1.3*\hp, 1.7*\vp) -- (1.7*\hp, 1.3*\vp);
\node at (1.6*\hp, 1.6*\vp) {\footnotesize $\pi_2$};

\draw[->] (0.7*\hp, 1.7*\vp) -- (0.3*\hp, 1.3*\vp);
\node at (0.4*\hp, 1.6*\vp) {\footnotesize $\pi_1$};

\draw[->] (0 *\hp, -0.3*\vp) -- (0*\hp, -0.6*\vp);
\node at (0*\hp, -0.8*\vp) {\footnotesize $0$};
\node at (0.35*\hp, -.55*\vp) {\footnotesize $\pi_1, \pi_3$};

\draw[->] (-2 *\hp, 0.7*\vp) -- (-2*\hp, 0.4*\vp);
\node at (-2*\hp, 0.2*\vp) {\footnotesize $0$};
\node at (-1.85*\hp, 0.55*\vp) {\footnotesize $\pi_3$};

\draw[->] (2 *\hp, 0.7*\vp) -- (2*\hp, 0.4*\vp);
\node at (2*\hp, 0.2*\vp) {\footnotesize $0$};
\node at (2.15*\hp, 0.55*\vp) {\footnotesize $\pi_1$};

\end{tikzpicture}
\]
One sees that 
\[
\ch([M]) = s_{(3,1)} + s_{(2,1,1)} = (F_{(3,1)} + F_{(2,2)} + F_{(1,3)}) + (F_{(2,1,1)} + F_{(1,2,1)} + F_{(1,1,2)}).
\]
So, if there exists a distinguished filtration of $M$ with respect to $\calB$, then there exists a three-dimensional $H_4(0)$-submodule $N$ of $M$ such that $\ch([N])$ is equal to either $s_{(3,1)}$ or $s_{(2,1,1)}$.
We claim that such a submodule $N$ does not exist.

Note that
\begin{align}\label{eq: decomposition of M}
M = \C \{2314 - 3214, 1423-1432, 2413, 3412 \} \oplus \C \{3214\} \oplus \C \{1432\}. 
\end{align}
Here, $\C \{3214\}$ and $\C \{1432\}$ are irreducible.
And, $\C \{2314 - 3214, 1423-1432, 2413, 3412 \}$ is indecomposable since it is isomorphic to a submodule of the injective indecomposable module $\bfP_{(1,2,1)}$.
Therefore, \cref{eq: decomposition of M} is a decomposition of $M$ into indecomposables.
The $H_4(0)$-action on $\{2314 - 3214, 1423-1432, 2413, 3412, 3214, 1432\} \cup \{ 0 \}$ is illustrated in the following figure:
\[
\def \vp {1.25}
\def \hp {2}
\begin{tikzpicture}
\node[] at (0*\hp, 0*\vp) {$3412$};

\node[] at (0*\hp, 1*\vp) {$2413$};
\node[] at (-3*\hp - 0.25*\hp, 1*\vp) {$3214$};
\node[] at (3*\hp + 0.25*\hp, 1*\vp) {$1432$};

\node[] at (-1*\hp, 2*\vp) {$2314 - 3214$};
\node[] at (1*\hp, 2*\vp) {$1423 - 1432$};

\draw[->] (-1.2*\hp, 1.7*\vp) -- (-1.2*\hp, 1.4*\vp);
\node at (-1.2*\hp, 1.2*\vp) {\footnotesize $0$};
\node at (-1.05*\hp, 1.55*\vp) {\footnotesize $\pi_2$};

\draw[->] (1.25*\hp, 1.7*\vp) -- (1.25*\hp, 1.4*\vp);
\node at (1.25*\hp, 1.2*\vp) {\footnotesize $0$};
\node at (1.4*\hp, 1.55*\vp) {\footnotesize $\pi_2$};

\node at (0*\hp + 0.2*\hp, 0*\vp) {} edge [out=40,in=320, loop] ();
\node at (0*\hp + 0.625*\hp, 0*\vp) {\footnotesize $\pi_2$};

\node at (-3*\hp + 0.2*\hp - 0.25*\hp, 1*\vp) {} edge [out=40,in=320, loop] ();
\node at (-3*\hp + 0.75*\hp - 0.25*\hp, 1*\vp) {\footnotesize $\pi_1, \pi_2$};

\node at (0*\hp + 0.2*\hp, 1*\vp) {} edge [out=40,in=320, loop] ();
\node at (0*\hp + 0.75*\hp, 1*\vp) {\footnotesize $\pi_1, \pi_3$};

\node at (3*\hp + 0.2*\hp + 0.25*\hp, 1*\vp) {} edge [out=40,in=320, loop] ();
\node at (3*\hp + 0.75*\hp + 0.25*\hp, 1*\vp) {\footnotesize $\pi_2, \pi_3$};

\node at (-1*\hp + 0.5*\hp, 2*\vp) {} edge [out=40,in=320, loop] ();
\node at (-1*\hp + 0.925*\hp, 2*\vp) {\footnotesize $\pi_1$};

\node at (1*\hp + 0.5*\hp, 2*\vp) {} edge [out=40,in=320, loop] ();
\node at (1*\hp + 0.925*\hp, 2*\vp) {\footnotesize $\pi_3$};

\draw[->] (0*\hp, 0.7*\vp) -- (0, 0.3);
\node at (0.15*\hp, 0.5*\vp) {\footnotesize $\pi_2$};


\draw[->] (-0.7*\hp, 1.7*\vp) -- (-0.3*\hp, 1.3*\vp);
\node at (-0.4*\hp, 1.6*\vp) {\footnotesize $\pi_3$};


\draw[->] (0.7*\hp, 1.7*\vp) -- (0.3*\hp, 1.3*\vp);
\node at (0.4*\hp, 1.6*\vp) {\footnotesize $\pi_1$};

\draw[->] (0 *\hp, -0.3*\vp) -- (0*\hp, -0.6*\vp);
\node at (0*\hp, -0.8*\vp) {\footnotesize $0$};
\node at (0.35*\hp, -.55*\vp) {\footnotesize $\pi_1, \pi_3$};

\draw[->] (-3*\hp - 0.25*\hp, 0.7*\vp) -- (-3*\hp - 0.25*\hp, 0.4*\vp);
\node at (-3*\hp - 0.25*\hp, 0.2*\vp) {\footnotesize $0$};
\node at (-2.85*\hp - 0.25*\hp, 0.55*\vp) {\footnotesize $\pi_3$};

\draw[->] (3*\hp + 0.25*\hp, 0.7*\vp) -- (3*\hp + 0.25*\hp, 0.4*\vp);
\node at (3*\hp + 0.25*\hp, 0.2*\vp) {\footnotesize $0$};
\node at (3.15*\hp + 0.25*\hp, 0.55*\vp) {\footnotesize $\pi_1$};

\node at (-1.75*\hp, \vp) {$\oplus$};
\node at (2*\hp, \vp) {$\oplus$};

\end{tikzpicture}
\]
The injective hulls of $\C \{2314 - 3214, 1423-1432, 2413, 3412 \}$, $\C \{3214\}$, and $\C \{1432\}$ are $\bfP_{(1,2,1)}$, $\bfP_{(1,3)}$, and $\bfP_{(3,1)}$, respectively.
This implies that the socle of $M$ is $\C\{3412\} \oplus \C \{3214\} \oplus \C \{1432\}$.
It follows that for every three-dimensional submodule $N$ of $M$, $1 \le \dim \soc(N) \le 3$.
We list all three-dimensional submodules $N$ of $M$ in \cref{tab: three-dimensional submodules of M}.
Based on this, we conclude that there are no $H_4(0)$-submodules $N$ of $M$ such that $\ch([N]) = s_{(3,1)}$ or $s_{(2,1,1)}$.
\end{example}

\begin{table}[t]
\renewcommand*\arraystretch{1.2}
\centering
$\begin{array}{c|c|c}
\text{ Three-dimensional submodules $N$ of $M$} & \dim \soc (N) & \ch([N]) 
\\ \hline \hline
\C \{3214, 1432, 3412\} & 3 & F_{(3,1)} + F_{(1,3)} + F_{(1,2,1)}
\\ \hline
\C \{3214, 1423 - 1432 - 2413, 3412\} & 2 & F_{(3,1)} + F_{(1,1,2)} + F_{(1,2,1)}
\\ \hline
\C \{3214, 2314 - 3214 - 2413, 3412\} & 2 & F_{(3,1)} + F_{(2,1,1)} + F_{(1,2,1)}
\\ \hline
\C \{3214, 2413, 3412\} & 2 & F_{(3,1)} + F_{(2,2)} + F_{(1,2,1)}
\\ \hline
\C \{1432, 1423 - 1432 - 2413, 3412\} & 2 & F_{(1,3)} + F_{(1,1,2)} + F_{(1,2,1)}
\\ \hline
\C \{1432, 2314 - 3214 - 2413, 3412 \} & 2 & F_{(1,3)} + F_{(2,1,1)} + F_{(1,2,1)}
\\ \hline
\C \{1432, 2413, 3412\} & 2 & F_{(1,3)} + F_{(2,2)} + F_{(1,2,1)}
\\ \hline
\C \{2314 - 3214, 2413, 3412\} & 1 & F_{(2,1,1)} + F_{(2,2)} + F_{(1,2,1)}
\\ \hline
\C \{1423-1432, 2413, 3412\} & 1 & F_{(1,1,2)} + F_{(2,2)} + F_{(1,2,1)}
\end{array}$
\caption{The complete list of three-dimensional submodules of $M$ in \cref{eg: not dist filt}}
\label{tab: three-dimensional submodules of M}
\end{table}

Let $f \in \Qsym_n$ and $\calB = \{\calB_\alpha \mid \alpha \in I \}$ be the linearly independent set given in ~\cref{def: dist filt}.
When $f$ expands positively in $\calB$, i.e., 
\begin{align}\label{eq: expand B}
f = \sum_{\alpha \in I} c_\alpha \calB_\alpha \quad (c_\alpha \in \Z_{\geq 0}),
\end{align}
finding an $H_n(0)$-module $M$ such that 
\begin{enumerate}[label = {\rm (C\arabic*)}]
\item $\ch([M]) = f$,
\item it is not a direct sum of irreducible modules, yet it possesses a combinatorial model that can be effectively handled, and
\item it has a distinguished filtration with respect to $\calB$
\end{enumerate}
is a very important problem in the sense that 
this filtration can be considered as a nice representation theoretic interpretation of ~\cref{eq: expand B}.

In this subsection, we focus on the above problem in the case where $\calB$ is $\calS := \{s_{\lambda} \mid \lambda \vdash n \}$ and $f = s_{\lambda/\mu}$ for a skew partition $\lambda/\mu$ of size $n$.
Note that for all $P \in \sfRSP_n$ with 
$\sh(\tau_P) = \lambda/\mu$,
$\sfM_P$ satisfies (C1) and (C2) because $\ch([\sfM_P]) = s_{\lambda/\mu}$ by ~\cref{thm: char of Mp}(2) and it has a combinatorial model $\Sigma_L(P)$.
In the following, we show that $\sfM_P$ satisfies (C3).

\begin{theorem}\label{thm: filt of X lam mu}
For every $P \in \sfRSP_n$,
$\sfM_P$ has a distinguished filtration  with respect to $\calS$.
\end{theorem}

\begin{proof}
To begin with, we choose any total order  $\ll$ on $\SYT_n$ subject to the condition that
\begin{align}\label{eq: total order condition}
T \ll S \quad \text{whenever $\sh(T) \triangleleft \sh(S)$.}
\end{align}
Write $\{\Qtab(\gamma) \mid \gamma \in \Sigma_L(P) \}$ as
\[
 \{T_1 \ll T_2 \ll \cdots \ll T_l \}.
\]
For $0 \leq k \leq l$, set 
\[
B_k := \{\gamma \in  \SG_n \mid \Qtab(\gamma) = T_i \  \text{for some $1 \leq i \leq k$}\}.
\]
It is clear that $\emptyset = B_0 \subset B_1 \subset B_2 \subset \cdots \subset B_l$.
And, by \cref{thm: RSP iff DPC}, we have $B_l = \Sigma_L(P)$.
We claim that 
\begin{align}\label{eq: module filtration candidate of dist}
0 = \C B_0 \subset \C B_1 \subset \C B_2 \subset \cdots \subset \C B_l = \sfM_{P}
\end{align}
is a distinguished filtration of $\sfM_P$ with respect to $\calS$.

First, we show that for $1 \leq k \leq l$, 
\[
\pi_i \cdot \gamma \in  B_k \cup \{0\}
\quad \text{for all $i \in [n-1]$ and $\gamma \in B_k$.}
\]
Take any $i \in [n-1]$ and $\gamma \in B_k$.
If $\pi_i \cdot \gamma = 0$ or $\gamma$, then there is nothing to prove.
Assume that $\pi_i \cdot \gamma = s_i \gamma$.
Then, by the definition of $H_n(0)$-action on $\sfM_P$, we have $\gamma \preceq_L s_i \gamma$.
Combining this inequality with \cref{eq: left bruhat order and shape} yields that 
\[
\gamma \dKeq s_i \gamma
\quad \text{or} \quad
\sh(\Qtab(s_i \gamma)) \triangleleft \sh(\Qtab(\gamma)).
\]
This implies that $s_i \gamma \in B_k$ as desired.

Next, we show that the filtration given in \cref{eq: module filtration candidate of dist}
is distinguished with respect to $\calS$.
For $1 \le k \le l$, $\{\gamma + M_{k-1} \mid \gamma \in B_k \setminus B_{k-1}\}$ is a basis for $M_k / M_{k-1}$ and $B_k \setminus B_{k-1}$ is an equivalence class under $\dKeq$.
It follows that $\ch([M_k / M_{k-1}])$ is a Schur function, more precisely, $\ch([M_k / M_{k-1}]) = s_{\sh(T_k)^\rmt}$.
\end{proof}

\begin{example}\label{eg: distinguished filt}
Let $P = \sfposet(\tau_0^{(4,2,1)/(2,1)})$.
Following the method presented in the proof of ~\cref{thm: filt of X lam mu}, we will construct two distinguished filtrations of $\sfM_P$ with respect to $\{s_{\lambda} \mid \lambda \vdash 4 \}$ by choosing two distinct total orders on $\SYT_4$.

Note that $\{\Qtab(\gamma) \mid \gamma \in \Sigma_L(P)\}$ is given by 
\[
\left\{
Q_1 := \begin{array}{c} 
\begin{ytableau}
1 \\
2 \\
3 \\ 
4
\end{ytableau}
\end{array}, \;
Q_2 := \begin{array}{c} 
\begin{ytableau}
1 & 3\\
2 \\
4
\end{ytableau}
\end{array}, \;
Q_3 := \begin{array}{c} 
\begin{ytableau}
1 & 4\\
2 \\
3
\end{ytableau}
\end{array}, \;
Q_4 := \begin{array}{c} 
\begin{ytableau}
1 & 3 \\
2 & 4
\end{ytableau}
\end{array}, \;
Q_5 := 
\begin{array}{c} 
\begin{ytableau}
1 & 3 & 4 \\
2
\end{ytableau}
\end{array}
\right\}
\]
and $\sh(Q_1) \triangleleft \sh(Q_2) = \sh(Q_3) \triangleleft \sh(Q_4) \triangleleft \sh(Q_5)$.
Choose a total order $\ll_1$(resp.  $\ll_2$) on $\SYT_4$ 
satisfying both   \cref{eq: total order condition} and $Q_2 \ll_1 Q_3$ (resp. $Q_3 \ll_2 Q_2$).
For $1 \leq k \leq 5$, let
\[
B'_k := \{\gamma \in  \SG_4 \mid \Qtab(\gamma) = Q_k\}.
\]
When we use $\ll_1$, 
we let 
$$
B_k := \bigsqcup_{1 \le l \le k} B'_l \quad \text{for $0 \le k \le 5$,}
$$
and when we use $\ll_2$, we let
$$
B_k := \bigsqcup_{1 \le l \le k} B'_l \quad \text{for $k = 0,1,3,4,5$}
\quad \text{and} \quad
B_2 := B'_1 \sqcup B'_3.
$$
Then, 
\[
0 = \C B_0  \subset \C B_1 \subset \C B_2 \subset \C B_3 \subset \C B_4 \subset \C B_5 = \sfM_P
\]
is the desired  distinguished  filtration of $\sfM_P$ with respect to $\{s_\lambda \mid \lambda \vdash 4 \}$.

For the readers' convenience, we draw the $H_4(0)$-action on the basis $\Sigma_L(P) = [2134, 4321]_L$ for $\sfM_P$ and the sets $B'_k$ $(1 \leq k \leq 5)$ in \cref{fig: figure for example}.
\end{example}

\begin{figure}[t]
\[
\def \vp {1.3}
\def \hp {2}
\begin{tikzpicture}
\node[right] at (-1.3*\hp, 0*\vp) {\textcolor{violet}{\small $B'_5$}};
\draw[violet, dashed, thick] (-0.4*\hp, -2.2*\vp) -- (-1.4*\hp, -1.4*\vp) -- (-1.4*\hp, -0.6*\vp) -- (-0.4*\hp, 0.2*\vp) -- (0.4*\hp, 0.2*\vp) -- (0.4*\hp, -2.2*\vp) -- (-0.4*\hp, -2.2*\vp);


\node[right] at (2.2*\hp, -1.3*\vp) {\textcolor{brown}{\small $B'_4$}};
\draw[brown, dashed, thick] (1.6*\hp, -2.2*\vp) -- (0.6*\hp, -1.4*\vp) -- (0.6*\hp, -0.8*\vp) -- (1.4*\hp, -0.8*\vp) -- (2.4*\hp, -1.6*\vp) -- (2.4*\hp, -2.2*\vp) -- (1.6*\hp, -2.2*\vp);


\node[right] at (-2.9*\hp, -2.5*\vp) {\textcolor{teal}{\small $B'_3$}};
\draw[teal, dashed, thick] (-1.4*\hp, -4.2*\vp) -- (-2.4*\hp, -3.4*\vp) -- (-2.4*\hp, -1.8*\vp) -- (-1.6*\hp, -1.8*\vp) -- (-1.6*\hp, -3*\vp) -- (-0.6*\hp, -3.8*\vp) -- (-0.6*\hp, -4.2*\vp) -- (-1.4*\hp, -4.2*\vp);

\node[right] at (2.2*\hp, -3.7*\vp) {\textcolor{blue}{\small $B'_2$}};
\draw[blue, dashed, thick] (0.6*\hp, -4.2*\vp) -- (-0.4*\hp, -3.4*\vp) -- (-0.4*\hp, -2.8*\vp) -- (2.4*\hp, -2.8*\vp) -- (2.4*\hp, -3.4*\vp) -- (1.4*\hp, -4.2*\vp) -- (0.6*\hp, -4.2*\vp);

\node[left] at (-0.5*\hp, -5*\vp) {\textcolor{red}{\small $B'_1$}};
\draw[red, dashed, thick]
(-0.4*\hp, -5.2*\vp)-- (-0.4*\hp, -4.8*\vp) -- (0.4*\hp, -4.8*\vp) -- (0.4*\hp, -5.2*\vp) -- (-0.4*\hp, -5.2*\vp);

\node[] at (0*\hp, 0*\vp) {$2134$};

\node[] at (-1*\hp, -1*\vp) {$3124$};
\node[] at (1*\hp, -1*\vp) {$2143$};

\node[] at (-2*\hp, -2*\vp) {$3214$};
\node[] at (0*\hp, -2*\vp) {$4123$};
\node[] at (2*\hp, -2*\vp) {$3142$};

\node[] at (-2*\hp, -3*\vp) {$4213$};
\node[] at (-0*\hp, -3*\vp) {$4132$};
\node[] at (2*\hp, -3*\vp) {$3241$};

\node[] at (-1*\hp, -4*\vp) {$4312$};
\node[] at (1*\hp, -4*\vp) {$4231$};

\node[] at (0*\hp, -5*\vp) {$4321$};

\node at (0*\hp + 0.2*\hp, 0*\vp) {} edge [out=40,in=320, loop] ();
\node at (0*\hp + 0.625*\hp, 0*\vp) {\footnotesize $\pi_1$};

\node at (-1*\hp + 0.2*\hp, -1*\vp) {} edge [out=40,in=320, loop] ();
\node at (-1*\hp + 0.625*\hp, -1*\vp) {\footnotesize $\pi_2$};
\node at (1*\hp + 0.2*\hp, -1*\vp) {} edge [out=40,in=320, loop] ();
\node at (1*\hp + 0.75*\hp, -1*\vp) {\footnotesize $\pi_1, \pi_3$};

\node at (-2*\hp + 0.2*\hp, -2*\vp) {} edge [out=40,in=320, loop] ();
\node at (-2*\hp + 0.75*\hp, -2*\vp) {\footnotesize $\pi_1, \pi_2$};
\node at (0*\hp + 0.2*\hp, -2*\vp) {} edge [out=40,in=320, loop] ();
\node at (0*\hp + 0.625*\hp, -2*\vp) {\footnotesize $\pi_3$};
\node at (2*\hp + 0.2*\hp, -2*\vp) {} edge [out=40,in=320, loop] ();
\node at (2*\hp + 0.625*\hp, -2*\vp) {\footnotesize $\pi_2$};

\node at (-2*\hp + 0.2*\hp, -3*\vp) {} edge [out=40,in=320, loop] ();
\node at (-2*\hp + 0.75*\hp, -3*\vp) {\footnotesize $\pi_1, \pi_3$};
\node at (0*\hp + 0.2*\hp, -3*\vp) {} edge [out=40,in=320, loop] ();
\node at (0*\hp + 0.75*\hp, -3*\vp) {\footnotesize $\pi_2, \pi_3$};
\node at (2*\hp + 0.2*\hp, -3*\vp) {} edge [out=40,in=320, loop] ();
\node at (2*\hp + 0.75*\hp, -3*\vp) {\footnotesize $\pi_1, \pi_2$};

\node at (-1*\hp + 0.2*\hp, -4*\vp) {} edge [out=40,in=320, loop] ();
\node at (-1*\hp + 0.75*\hp, -4*\vp) {\footnotesize $\pi_2, \pi_3$};
\node at (1*\hp + 0.2*\hp, -4*\vp) {} edge [out=40,in=320, loop] ();
\node at (1*\hp + 0.75*\hp, -4*\vp) {\footnotesize $\pi_1, \pi_3$};

\node at (0*\hp + 0.2*\hp, -5*\vp) {} edge [out=40,in=320, loop] ();
\node at (0*\hp + 0.875*\hp, -5*\vp) {\footnotesize $\pi_1, \pi_2, \pi_3$};

\draw[->] (-0.2*\hp, -0.25*\vp) -- (-0.8*\hp, -0.7*\vp);
\node at (-0.6*\hp, -0.4*\vp) {\footnotesize $\pi_2$};
\draw[->] (0.2*\hp, -0.25*\vp) -- (0.8*\hp, -0.7*\vp);
\node at (0.6*\hp, -0.4*\vp) {\footnotesize $\pi_3$};

\draw[->] (-1.2*\hp, -1.25*\vp) -- (-1.8*\hp, -1.7*\vp);
\node at (-1.6*\hp, -1.4*\vp) {\footnotesize $\pi_1$};
\draw[->] (-0.8*\hp, -1.25*\vp) -- (-0.2*\hp, -1.7*\vp);
\node at (-0.4*\hp, -1.4*\vp) {\footnotesize $\pi_3$};

\draw[->] (1.2*\hp, -1.25*\vp) -- (1.8*\hp, -1.7*\vp);
\node at (1.6*\hp, -1.4*\vp) {\footnotesize $\pi_2$};
\draw[->] (-2*\hp, -2.25*\vp) -- (-2*\hp, -2.7*\vp);
\node at (-2.1*\hp, -2.45*\vp) {\footnotesize $\pi_3$};

\draw[->] (-0.2*\hp, -2.25*\vp) -- (-1.8*\hp, -2.7*\vp);
\node at (-1.25*\hp, -2.45*\vp) {\footnotesize $\pi_1$};
\draw[->] (0*\hp, -2.25*\vp) -- (0*\hp, -2.7*\vp);
\node at (-0.1*\hp, -2.45*\vp) {\footnotesize $\pi_2$};

\draw[->] (1.8*\hp, -2.25*\vp) -- (0.2*\hp, -2.7*\vp);
\node at (0.75*\hp, -2.45*\vp) {\footnotesize $\pi_3$};
\draw[->] (2*\hp, -2.25*\vp) -- (2*\hp, -2.7*\vp);
\node at (2.1*\hp, -2.45*\vp) {\footnotesize $\pi_1$};

\draw[->] (-1.8*\hp, -3.25*\vp) -- (-1.2*\hp, -3.7*\vp);
\node at (-1.4*\hp, -3.4*\vp) {\footnotesize $\pi_2$};

\draw[->] (0.2*\hp, -3.25*\vp) -- (0.8*\hp, -3.7*\vp);
\node at (0.6*\hp, -3.4*\vp) {\footnotesize $\pi_1$};
\draw[->] (1.8*\hp, -3.25*\vp) -- (1.2*\hp, -3.7*\vp);
\node at (1.4*\hp, -3.4*\vp) {\footnotesize $\pi_3$};

\draw[->] (-0.8*\hp, -4.25*\vp) -- (-0.2*\hp, -4.7*\vp);
\node at (-0.4*\hp, -4.4*\vp) {\footnotesize $\pi_1$};
\draw[->] (0.8*\hp, -4.25*\vp) -- (0.2*\hp, -4.7*\vp);
\node at (0.4*\hp, -4.4*\vp) {\footnotesize $\pi_2$};
\end{tikzpicture}
\]
\caption{The $H_4(0)$-action on the basis $\Sigma_L(P) = [2134, 4321]_L$ for $\sfM_P$ and the sets $B'_k$ $(1 \leq k \leq 5)$ in \cref{eg: distinguished filt}}
\label{fig: figure for example}
\end{figure}

\section{Further avenues}\label{sec: Final remark}

In this section, we discuss future directions regarding the classification problem, the decomposition problem, and how to recover $\sfM_P$ for $P \in \sfRSP_n$ from a module of the generic Hecke algebra $H_n(q)$ by specializing $q$ to $0$.

\subsection{The classification problem}\label{subsec: classification}

In ~\cref{thm: classification}, we successfully classify $\sfM_P$ for $P \in \sfRSP_n$.
To be precise, we show that for $P, Q \in \sfRSP_n$, 
\begin{align}\label{eq: final rem class}
\sfM_P \cong \sfM_Q \quad \text{if and only if} \quad \sh(\tau_P) = \sh(\tau_Q).
\end{align}
Recall that $\sfRSP_n = \sfRP_n \cap \sfSP_n$. 
Hence, it would be natural to consider the classification problem for $\{\sfM_P \mid P \in \sfSP_n\}$ and $\{\sfM_P \mid P \in \sfRP_n\}$.

\subsubsection{A remark on the classification problem for $\{\sfM_P \mid P \in \sfSP_n \}$}\label{subsubsec: rem 1}
Since the notion `the shape of $\tau_P$' is available for $P \in \sfSP_n$, 
one may expect that the classification given in ~\cref{eq: final rem class} can be extended to $\{\sfM_P  \mid P \in \sfSP_n\}$.
Unfortunately,  this expectation turns out to be false.

Let
\[
\tau_1 := \begin{array}{c}
\begin{ytableau}
    \none & \none & 2 & 1 \\
    4 & 3
\end{ytableau}
\end{array}, 
\quad
\tau_2 := \begin{array}{c}
\begin{ytableau}
\none & \none & 4 & 2 \\
 3 & 1
\end{ytableau}
\end{array}
\quad \text{and} \quad
\tau_3 := \begin{array}{c}
\begin{ytableau}
\none & \none & 4 & 1 \\
3 & 2
\end{ytableau}
\end{array}.
\]
Then   $\sfM_{\sfposet(\tau_i)}$ $(i = 1,2,3)$ is decomposed into indecomposables as follows:
\begin{align*}
\sfM_{\sfposet(\tau_1)} & \cong \bfP_{(4)} \oplus \bfP_{(2,2)}, \\
\sfM_{\sfposet(\tau_2)} 
& \cong \bfF_{(1,2,1)} \oplus \sfB(4213, 4312) \oplus \bfF_{(3,1)} \oplus \bfF_{(2,2)} \oplus \bfF_{(4)}, \\
\sfM_{\sfposet(\tau_3)} 
& \cong \bfF_{(1,2,1)} \oplus \sfB(4213, 4312) \oplus \sfB(2431, 3421) \oplus \bfF_{(4)},
\end{align*}
where $\sfB(4213, 4312)$ and $\sfB(2431, 3421)$ are $2$-dimensional indecomposable modules.
These decompositions  show that  $\sfM_{\sfposet(\tau_1)}$, $\sfM_{\sfposet(\tau_2)}$, and $\sfM_{\sfposet(\tau_3)}$ are pairwise non-isomorphic although all  $\tau_{\sfposet(\tau_i)}$'s have the same shape.

\subsubsection{A conjecture on the classification problem for $\{\sfM_P \mid P \in \sfRP_n \}$}\label{subsubsec: rem 2}

Note that for $P \in \sfRP_n$, the notion `the shape of $\tau_P$' has not been defined.
This leads us to introduce a 
classification of 
$\{\sfM_P \mid P \in \sfRSP_n\}$ without using this notion.
To be precise, by combining \cref{thm: equiv class X lam mu}
and \cref{thm: classification}, we derive that
for $P, Q \in \sfRSP_n$,
\begin{align}\label{eq: conj on class}
\sfM_P \cong \sfM_Q \quad \text{if and only if} \quad \Sigma_L(P) \Deq \Sigma_L(Q).
\end{align}
We expect that this classification can be extended  to $\sfRP_n$ in its current form.
The validity of this expectation has been checked for values of $n$ up to 6 with the aid of the computer program 
\textsc{SageMath}.
Let us provide an overview of our verification process.
We first classify all left weak Bruhat intervals in $\SG_n$ ($n \leq 6$) up to descent-preserving isomorphism and choose 
a complete list $\frakI_n$ of inequivalent representatives.
Next, we let $\mathfrak{A}_n$ be the set of all unordered pairs
$([\sigma_1, \rho_1]_L, [\sigma_2, \rho_2]_L)$ of intervals in  $\frakI_n$ satisfying that  $[\sigma_1, \rho_1]_L \neq [\sigma_2, \rho_2]_L$ and 
\begin{align}\label{eq: necc cond for iso}
\ch([\sfB(\sigma_1, \rho_1)]) = \ch([\sfB(\sigma_2, \rho_2)]), \;\, \Des_L(\sigma_1) = \Des_L(\sigma_2), \;\,  \Des_L(\rho_1) = \Des_L(\rho_2).
\end{align}
Note that ~\cref{eq: necc cond for iso} is a necessary condition for $\sfB(\sigma_1, \rho_1) \cong \sfB(\sigma_2, \rho_2)$.
Finally, we show that  for all $(I_1, I_2) \in \mathfrak{A}_n$, $\sfB(I_1) \not\cong \sfB(I_2)$.
When $n \leq 5$, there is nothing to prove because $\mathfrak{A}_n = \emptyset$.
When $n = 6$,
$\mathfrak{A}_6$ has fourteen pairs.
Note that 
if $(I_1, I_2) \in \mathfrak{A}_6$,  then $(w_0 \cdot I_1 \cdot {w_0}, w_0 \cdot I_2 \cdot {w_0}) \in \mathfrak{A}_6$ and 
\[
\sfB(I_1) \cong \sfB(I_2) \quad \underset{\text{\cite[Theorem 4]{22JKLO}}}{\Longleftrightarrow} \quad \sfB(w_0 \cdot I_1 \cdot w_0) \cong \sfB(w_0 \cdot I_2 \cdot w_0). 
\]
Therefore, it suffices to examine seven pairs $(I_1^{(k)}, I_2^{(k)})$ listed in ~\cref{tab: non des pres modules}.
\begin{table}[t]
\renewcommand*\arraystretch{1.2}
\centering
$\begin{array}{c|c|c}
 k & I_1^{(k)} & I_2^{(k)} 
\\ \hline \hline
1 & [123456, 426351]_L & [123456, 624153]_L 
\\ \hline
2 & [123456, 354612]_L & [123456, 561324]_L
\\ \hline
3 & [123456, 356412]_L & [123456, 561342]_L
\\ \hline
4 & [123456, 563124]_L &  [123456, 534612]_L
\\ \hline
5 & [123456, 536412]_L & [123456, 563142]_L
\\ \hline
6 & [123456, 465312]_L & [123456, 645132]_L
\\ \hline
7 & [123456, 564213]_L & [123456, 546231]_L
\end{array}$
\caption{Seven pairs $(I_1^{(k)}, I_2^{(k)})$ in $\mathfrak{A}_6$}
\label{tab: non des pres modules}
\end{table}
For $3 \leq k \leq 7$, 
using \cref{lem: radical cond}, 
one can see that the projective covers of $\sfB(I_1^{(k)})$ and $\sfB(I_2^{(k)})$ are not isomorphic.
Therefore, $\sfB(I_1^{(k)})$ and $\sfB(I_2^{(k)})$ are not isomorphic.
For $k = 1, 2$, one can see that $\sfB(I_1^{(k)})$ and $\sfB(I_2^{(k)})$ are not isomorphic  in a brute force manner.

Let us give another evidence for our expectation.
Specifically, we show that \cref{eq: conj on class} holds when $P \in \sfRSP_n$, $Q \in \sfRP_n$, and $\ch([\sfM_P])$ is a Schur function. 
This can be derived from the proposition presented below.

\begin{proposition}\label{prop: small evi for conj}
Let $P$ be a poset in $\sfRSP_n$ such that $\ch([\sfM_P])$ is a Schur function.

\begin{enumerate}[label = {\rm (\arabic*)}]
\item If $Q \in \sfP_n$ satisfies that $\sfM_Q \cong \sfM_P$, 
then $Q \in \sfRSP_n$.
\item  The isomorphism class of $\sfM_P$ within $\{\sfM_Q \mid Q \in \sfP_n \}$
is equal to the isomorphism class of $\sfM_P$ within $\{\sfM_Q \mid Q \in \sfRSP_n \}$ as sets.
\end{enumerate}
\end{proposition}

\begin{proof}
(1) Suppose that $\ch([\sfM_P]) = s_{\lambda}$ for some $\lambda \vdash n$.
By ~\cite[Theorem 2.2]{05W}, $\sh(\tau_P)$ is either $\lambda$ or $\lambda^\circ$, where $\lambda^\circ$ denotes the skew partition whose Young diagram is obtained by rotating $\tyd(\lambda)$ by $180^\circ$.

First, we consider the case where $\sh(\tau_P) = \lambda$.
Let $f: \sfM_P \ra \sfM_Q$ be an $H_n(0)$-module isomorphism.
By ~\cref{thm: interval descriptrion for regular skew schur poset},
we see that 
$\Sigma_L(P) = [\rmread_{\tau_P} (T_{\lambda}), \rmread_{\tau_P}(T'_{\lambda})]_L$ and therefore $\rmread_{\tau_P} (T_{\lambda})$ is a cyclic generator of $\sfM_P$.
In addition, in view of \cite[Lemma 3.12]{19Searles}, we have that 
\begin{equation}\label{negation of descent inclusion}
\Des_L(\rmread_{\tau_P}(T)) \not\supseteq \Des_L(\rmread_{\tau_P}(T_{\lambda}))
\text{ for all $T \in \SYT(\lambda) \setminus \{T_{\lambda} \}$.}
\end{equation}
Combining ~\eqref{negation of descent inclusion} with the equality $\ch([\sfM_P]) = \ch([\sfM_Q])$, we can deduce  that there exists a unique $\sigma \in \Sigma_L(Q)$ such that $\Des_L(\sigma) \supseteq \Des_L(\rmread_{\tau_P}(T_{\lambda}))$.
This fact implies that  $f(\rmread_{\tau_P}(T_{\lambda})) = c \sigma$ for some nonzero $c \in \C$.
We may assume that $c = 1$ by considering the isomorphism $\frac{1}{c}f$ instead of $f$. 
Since $f$ is an $H_n(0)$-module isomorphism,
$\Sigma_L(Q)$ is equal to $f(\Sigma_L(P))$ and therefore is a left weak Bruhat interval. Furthermore, it holds that   
\[
\Des_L(f(\gamma)) = \Des_L(\gamma) \quad \text{for all $\gamma \in \Sigma_L(P)$.}
\]
As a consequence, we obtain a descent-preserving isomorphism $f|_{\Sigma_L(P)}: \Sigma_L(P) \ra \Sigma_L(Q)$.
Now the assertion follows from  ~\cref{thm: equiv class X lam mu}.

Next, consider the case where $\sh(\tau_P) = \lambda^\circ$.
Let $\overline{P}^*$ and $\overline{Q}^*$ be the posets in $\sfP_n$ whose orders are defined by 
\[
u \preceq_{\overline{P}^*} v
\Longleftrightarrow 
n + 1 - v \preceq_P n + 1 - u
\quad \text{and} \quad
u \preceq_{\overline{Q}^*} v
\Longleftrightarrow 
n + 1 - v \preceq_Q n + 1 - u,
\]
respectively.
Since $P$ is a poset in $\sfRSP_n$ with $\sh(\tau_P) = \lambda^\circ$,  $\overline{P}^*$ is a poset in $\sfRSP_n$ with $\sh(\tau_{\overline{P}^*}) = \lambda$.
By \cite[Theorem 3.6(a)]{23CKO}, we have $\sfM_{\overline{P}^*} \cong \bfT^+_{\autophi} (\sfM_{P})$ and $\sfM_{\overline{Q}^*} \cong \bfT^+_{\autophi} (\sfM_{Q})$, which implies that $\sfM_{\overline{Q}^*} \cong \sfM_{\overline{P}^*}$.
It follows from the first case that $\overline{Q}^* \in \sfRSP_n$, thus $Q \in \sfRSP_n$.

(2) It follows from (1).
\end{proof}

Based on these evidences, we propose the following conjecture.

\begin{conjecture}\label{conj: first one}
Let $P, Q \in \sfRP_n$.
If $\sfM_P \cong \sfM_Q$, then $\Sigma_L(P) \Deq \Sigma_L(Q)$.
\end{conjecture}

We remark that the converse of ~\cref{conj: first one} holds due to 
\cref{Prop: desc pres isom and 0-Hecke}.

\subsection{The decomposition problem of $\sfM_P$ for $P \in \sfRSP_n$}\label{subsec: decomp}

A Young diagram of skew shape is called a \emph{ribbon} if it does not contain any $2\times 2$ square. 
For simplicity, we call a skew partition a \emph{ribbon} if the corresponding Young diagram is a ribbon.
Note that our ribbons are not necessarily connected.
Consider a skew partition 
\[
\lambda/\mu = \lambda^{(1)} / \mu^{(1)} \star \lambda^{(2)} / \mu^{(2)} \star \cdots \star \lambda^{(k)} / \mu^{(k)}
\]
such that  $\lambda^{(i)} / \mu^{(i)}$ is connected for all $1 \le i \le k$.
We say that \emph{$\lambda / \mu$ contains a disconnected ribbon} if there exists an index $1 \le j \le k-1$ such that both $\lambda^{(j)} / \mu^{(j)}$ and $\lambda^{(j+1)} / \mu^{(j+1)}$ are ribbons.
With this notation,   we state the following proposition.
\begin{proposition}\label{prop: regard decomp}
Let $P \in \sfRSP_n$.
\begin{enumerate}[label = {\rm (\arabic*)}]
\item If $\sh(\tau_P)$ is connected, then $\sfM_P$ is indecomposable.
\item If $\sh(\tau_P)$ contains a disconnected ribbon, then $\sfM_P$ is not indecomposable.
\end{enumerate}
\end{proposition}

\begin{proof}
(1) It follows from \cref{lem: proj cov and inj hull}.

(2) 
Suppose that $\sh(\tau_P)$  contains a disconnected ribbon.
Let $\lambda/\mu = \sh(\tau_P)$.
Write $\lambda/\mu$ as $\lambda^{(1)} / \mu^{(1)} \star \lambda^{(2)} / \mu^{(2)} \star \cdots \star \lambda^{(k)} / \mu^{(k)}$, where  $\lambda^{(i)} / \mu^{(i)}$ is connected for all $1 \le i \le k$ and  both $\lambda^{(j)} / \mu^{(j)}$ and $\lambda^{(j+1)} / \mu^{(j+1)}$ are ribbons for some $1 \leq j \leq k-1$.

In ~\cref{sec: appendix}, we constructed an $H_n(0)$-module $X_{\lambda/\mu}$ satisfying that  
$X_{\lambda/\mu} \cong \sfM_P$. 
From now on, we will prove the assertion for $X_{\lambda/\mu}$ instead of $\sfM_P$.
By \cref{prop: property of X}(1), we have the $H_n(0)$-module isomorphism 
\[
X_{\lambda/\mu} \cong X_{\lambda^{(1)}/\mu^{(1)}} \boxtimes \cdots \boxtimes X_{\lambda^{(k)}/\mu^{(k)}}.
\]
Set $X^{(1)} := X_{\lambda^{(1)} / \mu^{(1)}}  
\boxtimes \cdots \boxtimes  X_{\lambda^{(j-1)} / \mu^{(j-1)}}$ and $X^{(2)} := X_{\lambda^{(j+2)} / \mu^{(j+2)}} \boxtimes \cdots \boxtimes X_{\lambda^{(k)} / \mu^{(k)}}$.
Since $\lambda^{(j)}/\mu^{(j)}$ and $\lambda^{(j+1)}/\mu^{(j+1)}$ are ribbons, 
$X_{\lambda^{(j)}/\mu^{(j)}} \cong \bfP_\alpha$ and $X_{\lambda^{(j+1)}/\mu^{(j+1)}}  \cong \bfP_\beta$, where
$\alpha = \balproj(\lambda^{(j)} / \mu^{(j)})$ and $\beta = \balproj(\lambda^{(j+1)} / \mu^{(j+1)})$.
Therefore, 
\[
X_{\lambda/\mu}
\cong X^{(1)} \boxtimes \bfP_\alpha \boxtimes \bfP_\beta \boxtimes X^{(2)}.
\]
Combining
~\cref{lem: Huang decomp gen Pa} with the  fact that $\boxtimes$ is distributive over $\oplus$, we derive the $H_n(0)$-module isomorphism
\[
X_{\lambda/\mu} \cong 
(X^{(1)} \boxtimes \bfP_{\alpha \cdot \beta} \boxtimes X^{(2)})
\oplus
(X^{(1)}
\boxtimes \bfP_{\alpha \odot \beta} \boxtimes X^{(2)}).
\]
This shows $X_{\lambda/\mu}$ is not indecomposable.
\end{proof}

The contraposition of \cref{prop: regard decomp}(2) says that 
if $\sfM_P$ is indecomposable, then $\sh(\tau_P)$ does not contain any disconnected ribbon.
We ask if the converse is true.
In the case where $\sh(\tau_P)$ is connected, it is true by \cref{prop: regard decomp}(1).
In the case where $\sh(\tau_P)$ is disconnected,
we verified its validity when $|P| \leq 6$.
Indeed, this was done by showing that 
$\Endo(\sfM_P)$ has no idempotent  except for $0$ and $\id$.
Refer to the following example.

\begin{example}\label{eg: disconnected and indecomp}
Let $\lambda/\mu = (3,3,1)/(1,1)$ and $P = \sfposet(\tau_0^{\lambda/\mu})$.
Then, $\Sigma_L(P) = [21435, 42531]_L$ is a basis for $\sfM_P$.
Let $f \in \Endo(\sfM_P)$ be an idempotent  and let $$
f(21435) = \sum_{\gamma \in [21435, 42531]_L} c_\gamma \gamma \quad (c_\gamma \in \C).
$$
Note that 
$$
\{\gamma \in [21435, 42531]_L \mid \Des_L(21435) \subseteq \Des_L(\gamma) \} = \{21435, 21543, 42531\}.
$$
Since $f$ is an $H_5(0)$-module homomorphism, 
this equality 
implies that 
 $c_\gamma = 0$ for all $\gamma \in [21435, 42531]_L \setminus \{21435,21543,42531\}$.
In addition, $c_{21543} = 0$ since 
$$
\pi_1 \pi_2 \cdot 21435 = 0 \quad \text{and} \quad
\pi_1 \pi_2 \cdot f(21435) =  c_{21543} \, 32541.
$$
Hence, $f - c_{21435} \, \id$ is an $H_5(0)$-module homomorphism such that 
\[
(f - c_{21435} \id) (\gamma) = \begin{cases}
c_{42531} 42531 & \text{if $\gamma = 21435$},\\
0 & \text{if $\gamma \in [21435, 42531]_L \setminus \{21435\}$,}
\end{cases}
\]
and therefore $(f - c_{21435} \id)^2 = 0$.
Since $f$ is an idempotent, the possible values for $c_{21435}$ are $0$ or $1$.
Using the fact that $f$ is an idempotent again, we have that $c_{42531} = 0$.
As a consequence, $f$ is $0$ or $\id$.
\end{example}

Based on the above discussion, we propose the following conjecture.

\begin{conjecture}\label{conj: on indecomp}
Let $P \in \sfRSP_n$.
Suppose that $\sh(\tau_P)$ is disconnected and does not contain any disconnected ribbon.
Then,  $\sfM_P$ is indecomposable.
\end{conjecture}

\subsection{Recovering $\sfM_P$ for $P \in \sfRSP_n$ from an  $H_n(q)$-module by  specializing $q$ to $0$}\label{subsec: generic Hecke}
Let $q \in \C$.
The Hecke algebra $H_n(q)$ is the associative $\C$-algebra with $1$ generated by $T_1, T_2, \ldots, T_{n-1}$ subject to the following relations:
\begin{align*}
T_i^2 &= (q-1) T_i + q \quad \text{for $1\le i \le n-1$},\\
T_i T_{i+1} T_i &= T_{i+1} T_i T_{i+1}  \quad \text{for $1\le i \le n-2$},\\
T_i T_j &=T_j T_i \quad \text{if $|i-j| \ge 2$}.
\end{align*}
Let $q \in \C$ be generic,  that is,  $q$ is neither zero nor a root of unity.
It is well known that $H_n(q)$ is isomorphic to the group algebra $\C[\SG_n]$, and thus 
the category of left finite dimensional $H_n(q)$-modules is semisimple
and there exists a ring isomorphism (\cite[Section 3.2]{97KT})
\[
\textbf{ch}_q: \bigoplus_{n\ge 0} \calG_0(H_n(q))  \ra \Sym,
\quad
[V^\lambda(q)] \mapsto s_\lambda.
\]
Here, $\bigoplus_{n\ge 0} \calG_0(H_n(q))$ is the Grothendieck ring of the tower   of generic Hecke algebras equipped with addition and multiplication from direct sum and induction product,
$\Sym$ is the ring of symmetric functions, and $V^\lambda(q)$ is the irreducible $H_n(q)$-module attached to a partition $\lambda$ of size $n$.
The explicit description of $V^{\lambda}(q)$ can be found in \cite[p.7]{92KW}.

Let $P \in \sfRSP_n$.
Viewing $q$ as an indeterminate, one may ask if 
$\sfM_P$ can be obtained from an $H_n(q)$-module by specializing $q$ to $0$.
However, it should be noted that the process of `specializing 
$q$ to 
$0$' depends on the choice of bases for the
$H_n(q)$-module under consideration, as illustrated in the example below.

\begin{example}
The irreducible $H_3(q)$-module $V^{(2,1)}(q)$ has the underlying space $\C\{v_1, v_2\}$ and the $H_3(q)$-action defined by
\begin{align*}
\begin{cases}
T_1 \cdot v_1 = -v_1, \\
T_2 \cdot v_1 = v_2,
\end{cases}
\quad \text{and} \quad\quad
\begin{cases}
T_1 \cdot v_2 = -q^2 v_1  + q v_2, \\
T_2 \cdot v_2 = qv_1 + (q - 1)v_2.
\end{cases}
\end{align*}
By the specialization $q = 0$, we have the $H_3(0)$-action on $\C\{v_1, v_2\}$ defined by
\begin{align*}
\begin{cases}
\opi_1 \cdot v_1 = -v_1, \\
\opi_2 \cdot v_1 = v_2,
\end{cases}
\quad \text{and} \quad\quad
\begin{cases}
\opi_1 \cdot v_2 = 0, \\
\opi_2 \cdot v_2 = - v_2.
\end{cases}
\end{align*}
The resulting module is isomorphic to $\bfT^{+}_\autotheta(\sfM_{P_1})$, where $P_1 = \sfposet(\tau_0^{(2,1)}) \in \sfRSP_3$.

On the other hand, if we choose the basis
 $\{w_1 := qv_1 - v_2, w_2 := (q^2 - q) v_1 - q v_2\}$ for $V^{(2,1)}(q)$,
 then we have
\begin{align*}
\begin{cases}
T_1 \cdot w_1 = w_2, \\
T_2 \cdot w_1 = -w_1,
\end{cases}
\quad \text{and} \quad\quad
\begin{cases}
T_1 \cdot w_2 = qw_1 + (q-1)w_2, \\
T_2 \cdot w_2 = -q^2 w_1 + q w_2.
\end{cases}
\end{align*}
By the specialization $q = 0$, we have the $H_3(0)$-action on $\C\{w_1, w_2\}$ defined by
\begin{align*}
\begin{cases}
\opi_1 \cdot w_1 = w_2, \\
\opi_2 \cdot w_1 = -w_1,
\end{cases}
\quad \text{and} \quad\quad
\begin{cases}
\opi_1 \cdot w_2 = -w_2, \\
\opi_2 \cdot w_2 = 0.
\end{cases}
\end{align*}
The resulting module is isomorphic to $\bfT^{+}_\autotheta(\sfM_{P_2})$, where $P_2 = \sfposet(\tau_0^{(2,2)/(1)}) \in \sfRSP_3$.
It is worthwhile to remark   that while $\bfT^{+}_\autotheta(\sfM_{P_1})$ and $\bfT^{+}_\autotheta(\sfM_{P_2})$ have the same quasisymmetric characteristic  
$\textbf{ch}_q([V^{(2,1)}(q)])$,  they are not isomorphic.
\end{example}

We expect that for $P \in \sfRSP_n$, $\bfT^{+}_\autotheta(\sfM_P)$ can be obtained from an $H_n(q)$-module, whose image under $\textbf{ch}_q$ equals $K_P$, by applying the specialization $q = 0$ to a suitable basis.

\appendix

\section{A tableau description of $\sfM_P$ for $P \in \sfRSP_n$}\label{sec: appendix}

Let $P \in \sfRSP_n$.
Note that $\Sigma_L(P)$ is a basis of $\sfM_P$ consisting of permutations.
Here, we construct an $H_n(0)$-module that is isomorphic to $\sfM_P$ and has a tableau basis.

For a skew partition $\lambda / \mu$ of size $n$,
consider the bijection 
$$f: \SYT(\lambda/\mu) \ra \Sigma_L(\sfposet(\tau_0^{\lambda/\mu})), \quad T \mapsto \rmread_{\tau_0}(T).
$$
Let $\widetilde{f}: \C\SYT(\lambda/\mu) \ra \sfM_{\sfposet(\tau_0^{\lambda/\mu})}$ be the $\C$-linear isomorphism obtained by extending $f$ by linearity.
We endow $\C\SYT(\lambda/\mu)$ with an $H_n(0)$-module structure by letting
\[
h \cdot x := \widetilde{f}^{-1}(h \cdot \widetilde{f}(x)) \quad \text{for $h \in H_n(0)$ and $x \in \C\SYT(\lambda/\mu)$}.
\] 
One can see that 
for $T \in \SYT(\lambda/\mu)$ and $1 \le i \le n-1$,
\[
\pi_i \cdot T =
\begin{cases}
T & \text{if $i$ is strictly left of $i+1$ in $T$,}\\
0 & \text{if $i$ and $i+1$ are in the same column of $T$,}\\
s_i \cdot T & \text{if $i$ is strictly right of $i + 1$ in $T$.}
\end{cases}
\]
Here, $s_i \cdot T$ is the tableau obtained from $T$ by swapping $i$ and $i+1$.
We denote the resulting module by $X_{\lambda/\mu}$.
By ~\cref{thm: classification}, we have
\begin{enumerate}[label = $\bullet$]
\item $\sfM_P \cong X_{\sh(\tau_P)}$ for $P \in \sfRSP_n$, and

\item $X_{\lambda/\mu} \not\cong X_{\nu/\kappa}$ for distinct skew partitions $\lambda/\mu$, $\nu/\kappa$ of size $n$.
\end{enumerate}
Therefore, $X_{\sh(\tau_P)}$ can be viewed as a representative of the isomorphism class of $\sfM_P$ in the category $\Hnmod$.

\begin{remark}

(1) 
For a composition $\alpha$, Searles ~\cite{19Searles}  constructed an indecomposable $0$-Hecke module $\mathsf{X}_\alpha$ whose image under the 
quasisymmetric characteristic is an extended Schur function.
In particular, when $\alpha$ is a partition, our $X_\alpha$ is identical to 
$\mathsf{X}_{\alpha}$.

(2)
For a generalized composition $\bal$, let $\lambda/\mu$ be a unique  skew partition satisfying the conditions that $\balproj(\lambda/\mu) = \bal$ and $\lambda/\mu$ is a ribbon.
Then, $X_{\lambda/\mu} \cong \bfP_\bal$.

\end{remark}

The following proposition shows how $X_{\lambda/\mu}$'s behave with respect to induction product, restrictions, and (anti-)automorphism twists of $\autophi$ and $\hautotheta$.

\begin{proposition}\label{prop: property of X}
We have the following isomorphisms.
\begin{enumerate}[label = {\rm (\arabic*)}]
    \item For skew partitions $\lambda/\mu$ of size $n$ and $\nu/\kappa$ of size $m$, 
    \[
    X_{\lambda/\mu} \boxtimes X_{\nu/\kappa} \cong X_{\lambda/\mu \star \nu/\kappa} \quad \text{as $H_{n+m}(0)$-modules.}
    \]
    \item For a skew partition $\lambda/\mu$ of size $n$ and
     $1 \leq k \leq n-1$, 
    \[
    X_{\lambda/\mu} \downarrow_{H_k(0) \otimes H_{n-k}(0)} \cong \bigoplus_{\substack{|\nu/\mu| = k \\ \mu \subset \nu \subset \lambda}} X_{\overline{\nu/\mu}} \otimes  X_{\overline{\lambda/\nu}} \quad \text{as $H_k(0) \otimes H_{n-k}(0)$-modules}.
    \]
    Here, $\overline{\nu/\mu}$ and $\overline{\lambda/\nu}$ denote the basic skew partitions whose Young diagrams are obtained from $\tyd(\nu/\mu)$ and $\tyd(\lambda/\nu)$, respectively,  by removing empty rows and empty columns.

    \item For a skew partition $\lambda/\mu$ of size $n$,
    \[
\bfT^{+}_\autophi(X_{\lambda/\mu}) \cong X_{(\lambda/\mu)^\circ} \quad \text{and} \quad  \bfT^{-}_\hautotheta(X_{\lambda/\mu}) \cong X_{\lambda^\rmt/\mu^\rmt}.
     \]
Here, $(\lambda/\mu)^\circ$ is the skew partition whose Young diagram is obtained by  rotating $\tyd(\lambda/\mu)$ by $180^\circ$. 
\end{enumerate}
\end{proposition}
\begin{proof}
The first assertion follows from
\cite[Lemma 4]{22JKLO},
the second from \cite[Theorem 2]{22JKLO}, and the third from \cite[Theorem 4]{22JKLO}.
\end{proof}

\paragraph{{\bf Acknowledgments.}}
The authors are deeply grateful to the anonymous referees for their meticulous reading of the manuscript and their invaluable advice. We would especially like to express our sincere thanks to the referee for bringing references \cite{84LS, 97LLT} and the contents of \cref{subsec: generic Hecke} to our attention.
\medskip 

\paragraph{{\bf Funding statement.}}
The first author was supported by the National Research Foundation of Korea(NRF) grant funded by the Korean Government (No. NRF-2020R1A5A1016126) and Basic Science Research Program through NRF funded by the Ministry of Education (No. RS-2023-00240377).
The second author was supported by NRF grant funded by the Korean Government (No. NRF-2020R1F1A1A01071055), Basic Science Research Program through NRF funded by the Ministry of Education (No. RS-2023-00271282),  NRF grant funded by the Korea government(MSIT) (No. RS-2024-00342349), and the Sogang University Research Grant of 2024(No. 202412001.01).
The third author was supported by NRF grant funded by the Korean Government (No. NRF-2020R1F1A1A01071055) and by NRF grant funded by the Korea government(MSIT) (No. RS-2024-00342349).

\bibliographystyle{plain}

\end{document}